\documentclass[a4paper,11pt]{amsart}
\usepackage{amssymb,amscd,amsmath,amsthm}
\usepackage[all,cmtip]{xy}
\usepackage{csquotes}
\usepackage{hyperref}
\usepackage{floatrow}
\usepackage{microtype}
\usepackage[all,cmtip]{xy}
\usepackage{thmtools,thm-restate}

\numberwithin{equation}{section}
\declaretheorem[numberwithin=section,name=Theorem]{thm}
\declaretheorem[name=Theorem,numbered=no]{thm*}
\declaretheorem[name=Corollary, sibling=thm]{cor}
\declaretheorem[name=Lemma, sibling=thm]{lem}
\declaretheorem[name=Proposition, sibling=thm]{prop}
\newtheorem*{prop*}{Proposition}
\theoremstyle{remark}
\newtheorem{rem}[thm]{Remark}
\newtheorem{ques}[thm]{Question}
\theoremstyle{definition}
\newtheorem{defn}[thm]{Definition}
\newtheorem*{defn*}{Definition}
\newtheorem{exmp}[thm]{Example}
\newtheorem*{exmp*}{Example}
\usepackage[margin=1.0in]{geometry}
\usepackage{tikz}
\usepackage[font=small,labelsep=none]{caption}

\begin{document}

\begin{abstract}
We show that if $G$ is a  group  of type $FP_{n+1}^{\mathbb{Z}_2}$ that is coarsely separated into three essential, coarse disjoint, coarse complementary  components by a coarse $PD_n^{\mathbb{Z}_2}$ space $W,$ then $W$ is at finite Hausdorff distance from a subgroup $H$ of $G$; moreover, $G$ splits over a subgroup commensurable to a subgroup of $H$. We use this to deduce that splittings of the form $G=A*_HB$, where $G$  is of type $FP_{n+1}^{\mathbb{Z}_2}$ and $H$ is a coarse $PD_n^{\mathbb{Z}_2}$ group such that both  $|\mathrm{Comm}_A(H): H|$ and $|\mathrm{Comm}_B(H): H|$ are greater than two, are invariant under quasi-isometry.
\end{abstract}
\title[Quasi-Isometry Invariance of Splittings over Coarse $PD_n^{\mathbb{Z}_2}$ Groups]{Quasi-Isometry Invariance of Group Splittings over Coarse Poincar\'e Duality Groups}
\author[A.J. Margolis]{Alexander J. Margolis}
\address{Alexander J. Margolis, Mathematical Institute, Andrew Wiles Building, University of Oxford, Oxford OX2 6GG, UK}
\thanks{The author was supported by EPSRC Grant 1499630.}
\email{margolis@maths.ox.ac.uk}
\date{\today}
\maketitle

\section{Introduction}\label{sec:intro}
One of the aims of geometric group theory is to understand the connection between the algebraic and large-scale geometric   properties of finitely generated groups.
We say that a group $G$ \emph{splits} over a subgroup $H$ if  either 
 $G=A*_H B$ where $H$ is a proper subgroup of $A$ and $B$, or 
 $G=A*_H$.

We show that  group splittings can often be detected from the large-scale geometry of a group. The most celebrated result along these lines is the following theorem of Stallings, providing a correspondence between the number of ends of a group --- a large-scale geometric property, and whether a group splits over a finite subgroup --- an algebraic property.

\begin{thm*}[\cite{stallings1968torsionfree}, \cite{stallings1971grouptheory}]\label{thm:stallings}
A finitely generated group splits over a finite subgroup if and only if it has more than one end.
\end{thm*}

Along with a theorem of Dunwoody \cite{dunwoody1985accessibility}, Stallings' theorem allows us to decompose groups into `smaller' pieces via graph of groups decompositions. A result of Papasoglu and Whyte \cite{papasoglu2002quasi} classifies finitely presented groups up to quasi-isometry in terms of their one-ended vertex groups. Papasoglu \cite{papasoglu2005quasi} gives a geometric characterisation of splittings over two-ended groups, generalising Stallings' theorem. Using Papaoslgu's theorem, recent progress has been made by Cashen and Martin \cite{cashen2016quasi} towards a classification of one-ended groups up to quasi-isometry in terms of their JSJ decompositions (a graph of groups decomposition that encodes all splittings over two-ended groups). Generalising Stallings' theorem  and Papasolgu's theorem to splittings over more complicated groups allows us to better understand the structure of groups up to quasi-isometry.
 
Consider the Cayley graph $\Gamma$ of a group $G$  with  respect to some finite generating set. We say that $C\subseteq G$ is a \emph{coarse complementary component} of $W\subseteq G$ if for some $R\geq 0$, $C\backslash N_R(W)$ is the vertex set of a union of components of $\Gamma \backslash N_R(W)$. The motivation behind this definition is that quasi-isometries preserve unions of complementary components, but do not necessarily preserve a single complementary component. Thus the notion of  `coarse complementary components' is a natural one when working with quasi-isometries.

 A coarse complementary component is said to be \emph{deep} if it is not contained in $N_R(W)$ for any $R\geq 0$; otherwise it is said to be \emph{shallow}. A collection of deep coarse complementary components of $W$ is said to be \emph{coarse disjoint} if the intersection of any pair is shallow. We say that $W$ \emph{coarsely $n$-separates} $G$ if there exist $n$ deep, coarse disjoint, coarse complementary components of $W.$  We say that $W$ \emph{coarsely separates} $G$ if $W$ coarsely 2-separates $G$. Coarse $n$-separation is a quasi-isometry invariant. The following proposition relates coarse separation to group splittings.

\begin{prop*}[{\cite[Lemma 2.2]{papasoglu2012splittings}}]
If a finitely generated group $G$ splits over a finitely generated subgroup $H,$ then $H$ coarsely separates $G$.
\end{prop*}

A group has more than one end if and only if it is coarsely separated by a point.  Consequently,  Stallings' theorem may be rephrased as follows:
\begin{thm*}
A  finitely generated group splits over a finite subgroup if and only if it is coarsely separated by a point. 
\end{thm*}
We prove a partial generalisation of Stallings' theorem, giving a large-scale geometric criterion that guarantees the existence of a group splitting. We use this to show that group splittings are often invariant under quasi-isometry.
Triangle groups provide examples of groups which don't split, but have finite index subgroups that do split over two-ended subgroups; hence admitting a  splitting over  a two-ended group is not invariant under  quasi-isometry. 
However,  there is a theorem by Papasoglu which partially generalises  Stallings' theorem. In the following theorem, a \emph{line} is defined to be  a coarsely embedded copy of $\mathbb{R}$.

\begin{thm*}[\cite{papasoglu2005quasi}]
Let $G$ be a finitely presented one-ended group which is not virtually a surface group. A line coarsely separates  $G$ if and only if $G$ splits over a two-ended subgroup.
\end{thm*}

This is known to be false if one drops the condition that $G$ is finitely presented. In \cite{papasoglu2012splittings}, Papasoglu constructs a line that coarsely separates the lamplighter group --- a finitely generated group that is not finitely presented and doesn't split over a two-ended subgroup.

To construct a splitting using the geometry of a group, we first construct a subgroup which coarsely separates the group, and then show that the group splits over a subgroup  commensurable to it. Much work has already been done on the second step, for example see  \cite{dunwoody2000algebraic}. 

Before stating our results, we need to define a few terms. For a ring $R$, a group $G$ is said to be of \emph{type $FP_n^R$} if it admits a partial projective resolution $$P_n\rightarrow P_{n-1}\rightarrow \dots \rightarrow P_0\rightarrow R \rightarrow 0$$ of the trivial $RG$-module $R$, such that each $P_i$ is finitely generated as an $RG$-module. A group is  of type $F_n$ if it  has a classifying space with finite $n$-skeleton. If $G$ is of type $F_n$, it is of type $FP_n^R$ for any $R$. These are examples of \emph{finiteness properties}, generalising the notions  of  being finitely generated and finitely presented. 

We will work with groups of type $FP_n^{\mathbb{Z}_2}$. The use of $\mathbb{Z}_2$ coefficients  is fairly natural in the context of group splittings and coarse separation. For example, the number of ends can be detected using cohomology with $\mathbb{Z}_2$ coefficients \cite{scott1979topological} and  Dunwoody's accessibility theorem holds for groups of type $FP_2^{\mathbb{Z}_2}$ \cite{dunwoody1985accessibility}.  We remark that a group of type $FP_n^\mathbb{Z}$ (often simply denoted as $FP_n$) is necessarily of type $FP_n^{\mathbb{Z}_2}$. 

Stallings' theorem holds for finitely generated groups, which are necessarily of type $FP_1^{\mathbb{Z}_2}$, whereas Papasoglu's theorem only holds for finitely presented groups, which are necessarily of type $FP_2^{\mathbb{Z}_2}$. This  motivates the principle that when examining splittings over groups that are  more complicated  than finite and two-ended groups, we  assume the ambient group has higher finiteness properties.

Coarse $PD_n^{\mathbb{Z}_2}$ spaces are defined in \cite{kapovich2005coarse}. They are spaces which have the same large-scale homological properties as $\mathbb{R}^n$. For example, the universal cover of a closed aspherical $n$-dimensional manifold is a coarse $PD_n^{\mathbb{Z}_2}$ space. To construct a splitting, we make the assumption that a coarse $PD_n^{\mathbb{Z}_2}$ space coarsely separates a group of type $FP_{n+1}^{\mathbb{Z}_2}$.

When dealing with the case of a line coarsely separating a group as in Papasoglu's theorem (as well as similar results such as \cite{bowditch1998cut}),  one needs a geometric criterion for recognizing virtually surface groups.  In doing this, one uses the Tukia, Gabai and Casson-Jungreis  theorem on convergence groups acting on the circle (\cite{tukia88homeomorphic}, \cite{gabai92convergence} and \cite{casson94convergence}). Since this theorem has no analogue in higher dimensions, we cannot rule out generalisations of triangle groups; we therefore make the assumption that a coarse $PD_n^{\mathbb{Z}_2}$ space coarsely $3$-separates a group.  Unfortunately, even this assumption is not sufficient for our purposes. Therefore, rather than working with deep components, we use what we call \emph{essential} components.
 This is a generalisation of the essential components found in \cite{papasoglu2005quasi}.

\begin{figure}[ht]
\begin{floatrow}[2]
\ffigbox{\includegraphics[scale=0.61]{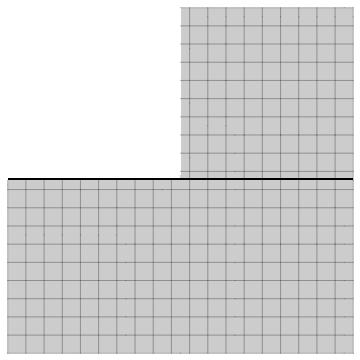}}{\caption{}\label{fig:ess1}}
\ffigbox{\includegraphics[scale=0.61]{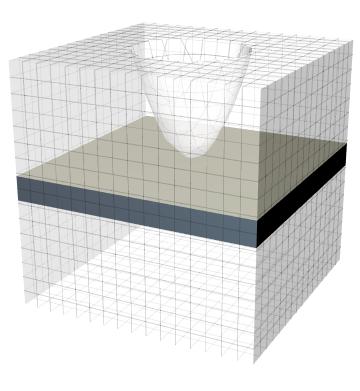}}{\caption{}\label{fig:ess2}}
\end{floatrow}
\end{figure}

Essential components are  necessarily deep and are invariant under quasi-isometry. An example is that of a coarse $PD_n^{\mathbb{Z}_2}$ space which coarsely separates a coarse $PD_{n+1}^{\mathbb{Z}_2}$ space into two `coarse $PD_{n+1}^{\mathbb{Z}_2}$ half-spaces' (see the coarse Jordan separation theorem of \cite{kapovich2005coarse}). One can think of a copy of $\mathbb{R}^n$ (or $\mathbb{H}^n$) coarsely separating $\mathbb{R}^{n+1}$ (or $\mathbb{H}^{n+1}$) into two half-spaces. Each coarse $PD_{n+1}^{\mathbb{Z}_2}$ half-space is essential. More generally, any coarse complementary component that contains such a half-space is necessarily essential. The definition of essential components is rather technical, so we do not define it here. However, we do give  some examples of essential and non-essential components.

In Figure \ref{fig:ess1}, a line coarsely separates the space into two deep components.
In Figure \ref{fig:ess2},  a plane coarsely separates the space into two deep components. In both cases, the bottom component is essential and the top component is not. These examples illustrate the only ways in which a coarse complementary component fails to be essential. In  Figure \ref{fig:ess1}, the boundary of the top component is not the entire line, but only half the line. It will be shown in Proposition \ref{prop:deepcondition} that this cannot occur for essential components. In Figure \ref{fig:ess2}, the boundary of the top component \emph{is} the entire plane, so the component fails to be essential in a more subtle way. It is not essential because it contains a non-trivial 1-cycle at infinity --- a `hole' that cannot be filled as we move away from the plane.

We are now in a position to state our main theorem; all the hypotheses of the theorem are invariant under quasi-isometry.

\begin{restatable*}{thm}{main}\label{thm:main}
Let $G$ be a group of type $FP_{n+1}^{\mathbb{Z}_2}$ and let $W\subseteq G$ be a coarse $PD_n^{\mathbb{Z}_2}$ subspace. Suppose $G$ contains three essential, coarse disjoint, coarse complementary components of $W.$ Then there exists a subgroup $H\leq G$, contained in $N_R(W)$ for some $R\geq 0$, such that $G$ splits over $H.$
\end{restatable*}

We now discuss some consequences of  Theorem \ref{thm:main} to demonstrate its applicability to situations in which its rather technical hypotheses are not a priori known to hold.

In deciding whether a group splits over a certain class of subgroups (e.g. virtually $\mathbb{Z}^n$ subgroups), it simplifies the argument if we make the natural  assumption that it does not split over a `smaller' class of subgroups (e.g. virtually $\mathbb{Z}^r$ subgroups for  $r<n$). For instance in Papasoglu's theorem, which determines if a group splits over a two-ended group, we assume that the ambient group is  one-ended; therefore, Stallings' theorem says it cannot split over a finite subgroup.

We want to find a higher dimensional analogue of a space being one-ended that rules out splittings over  certain classes of subgroups. The right generalisation of one-endedness is  acyclicity at infinity over $\mathbb{Z}_2$, which will be defined in Section \ref{sec:topatinf}. If a group is of type $FP_{n+1}^{\mathbb{Z}_2}$ and is $(n-1)$-acyclic at infinity over $\mathbb{Z}_2$, then it cannot split over a virtually $\mathbb{Z}^r$ subgroup for any $r<n$.

A finitely generated group $G$ has  one end if and only if it is $0$-acyclic at infinity over $\mathbb{Z}_2$.  If $G$ is the fundamental group of a closed aspherical $n$-manifold, or more generally is a  coarse $PD_n^{\mathbb{Z}_2}$  space, then it is $(n-2)$-acyclic at infinity over $\mathbb{Z}_2$. 

We say  $W\subseteq G$ is \emph{essentially embedded} if every deep coarse complementary component of $W$ is essential.  A group is a \emph{coarse $PD_n^{\mathbb{Z}_2}$ group} if, when equipped with the word metric with respect to some finite generating set, it is a coarse $PD_n^{\mathbb{Z}_2}$ space. We prove the following criterion which determines when a  coarse $PD_n^{\mathbb{Z}_2}$ subgroup is essentially embedded. 

\begin{restatable*}{prop}{essentialminimal}
\label{prop:essentialminimal}
Let $G$ be a group of type $FP_{n+1}^{\mathbb{Z}_2}$ that is $(n-1)$-acyclic at infinity over $\mathbb{Z}_2$ and let $H\leq G$ be a coarse $PD_n^{\mathbb{Z}_2}$ group. Then $H$ is essentially embedded if and only if no infinite index subgroup of $H$ coarsely separates $G$.
\end{restatable*}

Combining this with Theorem \ref{thm:main} and the observation that the hypotheses of Theorem \ref{thm:main} are invariant under quasi-isometry, we deduce the following:

\begin{restatable*}{thm}{qiinv}
\label{thm:qiinv}
Let $G$ be a group of type $FP_{n+1}^{\mathbb{Z}_2}$ that is  $(n-1)$-acyclic at infinity over $\mathbb{Z}_2$. Suppose $H\leq G$ is a coarse $PD_n^{\mathbb{Z}_2}$  group that coarsely $3$-separates $G$, and no infinite index subgroup of $H$ coarsely separates $G$. Then for any quasi-isometry $f:G\rightarrow G'$ there is a subgroup $H'\leq G'$, at finite Hausdorff distance from $f(H)$, such that $G'$ splits over $H'$.
\end{restatable*}

It is shown in \cite{brown2000improper} that if $G$ is the fundamental group of a finite graph of groups whose vertex and edge groups satisfy appropriate finiteness and acyclicity at infinity conditions, then so does $G$. We therefore deduce the following:

\begin{restatable*}{cor}{splittings}\label{cor:splittings}
Suppose $G=A*_HB$ (or $G=A*_H$) is a splitting, where $H$ is a coarse $PD_n^{\mathbb{Z}_2}$ group, $A$ and $B$ are of type $FP_{n+1}^{\mathbb{Z}_2}$ and are $(n-1)$-acyclic over infinity over $\mathbb{Z}_2$, $H$ coarsely $3$-separates $G$, and no infinite index subgroup of $H$ coarsely separates $G$. 
Then for any quasi-isometry $f:G\rightarrow G'$ there is a subgroup $H'\leq G'$, at finite Hausdorff distance from $f(H)$, such that $G'$ splits over $H'$.
\end{restatable*}

 This corollary is particularly useful when combined with Theorem 8.7 of \cite{vavrichek2012coarse}, which gives an algebraic characterisation of when the 3-separating hypothesis of Corollary \ref{cor:splittings} holds.

There are several examples of groups that are acyclic at infinity.
Say $G$ is the extension of $N$ by $Q$, where $N$ and $Q$ are groups of type $F_{n+1}$  that are $r$ and $s$-acyclic at infinity over $\mathbb{Z}_2$ respectively. Then Theorem  17.3.6 in \cite{geoghegan2008topological} tells us that  $G$ is $\min(n,s+r+2)$-acyclic at infinity over $\mathbb{Z}_2$. For example, if $N$ and $Q$ are finitely presented one-ended groups, then $G$ is $1$-acyclic at infinity over $\mathbb{Z}_2$. Results from \cite{brady2001connectivity} and \cite{davis2002topology}  give conditions for Coxeter and right-angled Artin groups to be acyclic at infinity. These results allow us to apply Theorem  \ref{thm:qiinv} and Corollary \ref{cor:splittings}.

Theorem \ref{thm:qiinv} can be simplified if $H$ is a virtually polycyclic group and thus necessarily a coarse $PD_n^{\mathbb{Z}_2}$ group. In this case, one can drop the condition that no infinite index subgroup of $H$ coarsely separates $G$, since it is implied by $G$ being $(n-1)$-acyclic at infinity.

\begin{restatable*}{cor}{qiinvpolycyc}
\label{cor:qiinvpolycyc}
Let $G$ be a group of type $FP_{n+1}^{\mathbb{Z}_2}$ that is  $(n-1)$-acyclic at infinity over $\mathbb{Z}_2$. Suppose $H\leq G$ is a virtually polycyclic subgroup of Hirsch length $n$ that coarsely 3-separates $G$. Then for any quasi-isometry $f:G\rightarrow G'$, there is a subgroup $H'\leq G'$, at finite Hausdorff distance from $f(H)$, such that $G'$ splits over $H'$.
\end{restatable*}

Corollary \ref{cor:qiinvpolycyc} may be coupled with the quasi-isometric rigidity of virtually $\mathbb{Z}^n$ groups. If a group $G$ is of type $FP_{n+1}^{\mathbb{Z}_2}$, is  $(n-1)$-acyclic at infinity over $\mathbb{Z}_2$ and is coarsely 3-separated by a virtually $\mathbb{Z}^n$ group, then any group quasi-isometric to $G$ splits over a virtually $\mathbb{Z}^n$ subgroup. A similar statement holds for virtually nilpotent groups.

Two subgroups $H,K\leq G$ are \emph{commensurable} if $H\cap K$ has finite index in both $H$ and $K$. The \emph{commensurator} of $H$ is the subgroup $$\mathrm{Comm}_G(H):=\{g\in G\mid H \textrm{ and } g^{-1}Hg \textrm{ are commensurable}\}.$$  We  deduce the following from Theorem \ref{thm:main}:

\begin{restatable*}{thm}{maincomm}
\label{thm:maincomm}
Let $G$ be a group of type $FP_{n+1}^{\mathbb{Z}_2}$ that is the fundamental group of a finite graph of groups $\mathcal{G}$. Suppose $\mathcal{G}$ contains an edge $e$ with associated edge monomorphisms $i_0:G_e\rightarrow G_v$ and $i_1:G_e\rightarrow G_w$ such that the following holds: $G_e$ is a coarse $PD_n^{\mathbb{Z}_2}$ group and $|\mathrm{Comm}_{G_v}(i_0(G_e)): i_0(G_e)|$ and $|\mathrm{Comm}_{G_w}(i_1(G_e)): i_1(G_e)|$ are both  greater than one and not both equal to two. If $f:G\rightarrow G'$ is a quasi-isometry, then $f(G_e)$ has finite Hausdorff distance from some subgroup $H'\leq G'$, and $G'$ splits over a subgroup commensurable to a subgroup of $H'$.
\end{restatable*}

Let $G$ be a group with Cayley graph $\Gamma$ with respect to some finite generating set. For a subgroup $H\leq G$, we let $\Gamma_H$ denote the quotient of $\Gamma$ by the
left action of $H$. We say that $H$ is a \emph{codimension one subgroup} if $\Gamma_H$ has more than one end. If $G$ splits over $H$, then $H$ is a codimension one subgroup, but the converse is not true in general. 

We prove two results which don't have a 3-separating hypothesis.   In the case where $G=\mathrm{Comm}_G(H)$  we obtain the following:

\begin{restatable*}{thm}{comm}\label{thm:comm}
Let $G$ be a group of type $FP_{n+1}^{\mathbb{Z}_2}$ and suppose $G$ splits over a coarse $PD_{n}^{\mathbb{Z}_2}$ subgroup $H$ and that $\mathrm{Comm}_G(H)=G$. Let $f:G\rightarrow G'$ be a quasi-isometry. Then either: \begin{enumerate}
\item $G'$ is  a coarse $PD_{n+1}^{\mathbb{Z}_2}$ group;
\item $f(H)$ has finite Hausdorff distance from some $H'\leq G'$, and $G'$ splits over $H'$.
\end{enumerate}
\end{restatable*}

This is a coarse geometric generalisation of a result from \cite{dunwoody1993splitting}, in which it is shown that if $G=\mathrm{Comm}_G(H)$ and $H$ is a codimension one subgroup of $G$, then $G$ splits over a subgroup commensurable to $H$.

A construction by Sageev \cite{sageev1995ends} shows that the existence of a codimension one subgroup is equivalent to an essential action   on a CAT(0) cube complex (as defined in \cite{sageev1995ends}). As a consequence, it can be shown that any group which has a codimension one subgroup cannot have property (T).

Under suitable hypotheses, we show that having a codimension one subgroup is a quasi-isometry invariant.  To state our result, we define the following:

\begin{defn*}
A group $G$ is  a \emph{coarse $n$-manifold group} if it is of type $FP_{n}^{\mathbb{Z}^2}$ and $H^n(G,\mathbb{Z}_2 G)$ has a non-zero, finite dimensional, $G$-invariant subspace.
\end{defn*}
The class of coarse $n$-manifold groups is closed under quasi-isometry and contains the fundamental group of every  $n$-manifold and every coarse $PD_n^{\mathbb{Z}_2}$ group. It is shown in \cite{kleinercohomology} that coarse $2$-manifold groups are virtually surface groups (see also  \cite{bowditch2004planar}). 
\begin{restatable*}{thm}{codimsbgp}\label{thm:codimsbgp}
Let $G$ be a group of type $FP_{n+1}^{\mathbb{Z}_2}$ and let $H\leq G$ be a coarse $PD_n^{\mathbb{Z}_2}$ group. Suppose $G$ contains two essential, coarse disjoint, coarse complementary components of $H$ and $f:G\rightarrow G'$ is a quasi-isometry. Then either: \begin{enumerate}
\item $G'$ is  a coarse $(n+1)$-manifold group;
\item $G'$ contains a codimension one subgroup.
\end{enumerate}
\end{restatable*}

We remark that the dichotomy in Theorems \ref{thm:comm} and \ref{thm:codimsbgp} loosely resembles the dichotomy in the Algebraic Torus Theorem \cite{dunwoody2000algebraic} which  states that if $H\leq G$ is  a virtually polycyclic subgroup of Hirsch length $n$ (so is necessarily coarse $PD_n^{\mathbb{Z}_2}$) and $H$ is a codimension one subgroup of $G$, then either $G$ is a coarse $PD_{n+1}^{\mathbb{Z}_2}$ group or $G$ splits over a virtually polycyclic subgroup of Hirsch length $n$.

Our work builds on results by Vavrichek \cite{vavrichek2012coarse}, Mosher--Sageev--Whyte in \cite{mosher2003quasi} and \cite{mosher2011quasiactions} and Papasoglu in \cite{papasoglu2007group}.
We remark that one of the conditions Vavrichek uses is the \emph{non-crossing} condition, which is not known to be invariant under quasi-isometry. Under our hypotheses, we obtain the non-crossing condition automatically (see Lemma \ref{lem:noncrossing}). 

Although the splittings considered in this paper are not canonical, there is a canonical object,  the JSJ tree of cylinders (\cite{guirardel2011cylinders}), which  encodes all possible splittings of the form considered in this paper. This tree of cylinders is a quasi-isometry invariant \cite{margolis2017cylinders}.

This paper is organized as follows.
In Section \ref{sec:coarsegeom}, we develop coarse geometric preliminaries. In Section \ref{sec:metriccomplex}, we review a notion of cohomology which is invariant under coarse isometries.  For finitely generated groups, this agrees with cohomology with group ring coefficients. There is a more general notion of coarse cohomology due to Roe (see \cite{roe2003lectures}). However, our approach to coarse cohomology is more amenable to the quantitative methods essential in our work. 
Our approach to coarse cohomology makes use of the theory of \emph{metric complexes} as defined in the appendix of \cite{kapovich2005coarse}. In our exposition, metric complexes can be replaced with \emph{bounded geometry CW complexes} as defined in \cite{mosher2003quasi}, with the caveat that this weakens our results from groups of type $FP_n^{\mathbb{Z}_2}$ to groups of type $F_n$.

In Section \ref{sec:coarsepdn}, we discuss the notion of coarse $PD_n^{\mathbb{Z}_2}$  spaces as defined in  \cite{kapovich2005coarse}. We also prove a lemma about coarse $PD_n^{\mathbb{Z}_2}$  spaces that is needed to deduce Theorem \ref{thm:qiinv}. In Section \ref{sec:coarsecompcomp}, we introduce coarse complementary components and describe the coarse Mayer--Vietoris sequence in coarse cohomology.

Section \ref{sec:final} is the heart of this paper. In it, we define essential components and obtain the above results. We make use of the notion of a \emph{mobility set} due to Kleiner and contained in an unpublished manuscript \cite{kleinercohomology}. We give a self-contained exposition of the parts of \cite{kleinercohomology} that we use.

I would like to thank my supervisor Panos Papasoglu for his advice and encouragement, and Bruce Kleiner for allowing me to use results contained in his unpublished manuscript.

\section{Coarse Geometric Preliminaries}\label{sec:coarsegeom}

Let $(X,d)$ be a metric space. For $x\in X$ and $\emptyset\neq A\subseteq X$, we let $d(x,A):=\inf \{d(a,x)\mid a\in A\}$.  
We define $N^X_r(A):=\{x\in X\mid d(x,A)\leq r\}$. For each $x\in X$, we let $N^X_r(x):=N_r^X(\{x\})$, which is just the closed $r$ ball around $x$. When unambiguous, we denote $N_r^X$ by $N_r$.
If  $A,B\subseteq X$ are non-empty, we define the \emph{Hausdorff distance} to be $$d_\mathrm{Haus}(A,B):=\inf\{r\geq 0\mid A\subseteq N_r(B) \textrm{ and } B\subseteq N_r(A)\}.$$
 A \emph{$t$-chain} of length $n$ from $x$ to $y$ consists of a sequence $x=x_0,x_1,\dots, x_n=y$ such that $d(x_i,x_{i+1})\leq t$. 

\subsection{Coarse Embeddings}
The material here is fairly standard, see \cite{roe2003lectures} for more details.
A function $\phi:\mathbb{R}_{\geq 0}\rightarrow \mathbb{R}_{\geq 0}$ is \emph{proper} if the inverse images of  compact sets are compact. 
\begin{defn}
Let $(X,d_X)$ and $(Y,d_Y)$ be metric spaces and $\eta,\phi:\mathbb{R}_{\geq 0}\rightarrow \mathbb{R}_{\geq 0}$ be  proper non-decreasing functions. We say a map $f:X\rightarrow Y$ is an  \emph{$(\eta,\phi)$-coarse embedding} if for all $x,y\in X$, $$\eta(d_X(x,y))\leq d_Y(f(x),f(y)) \leq \phi(d_X(x,y)).$$
We say that $f$ is a \emph{coarse embedding} if there exist proper non-decreasing functions $\eta$ and $\phi$ such that $f$ is an $(\eta,\phi)$-coarse embedding. We say that $\eta$ (resp. $\phi$)  is the \emph{lower} (resp. \emph{upper}) \emph{distortion function} of $f$.
\end{defn}
In the literature, coarse embeddings are also known as \emph{uniform embeddings} or \emph{uniformly proper embeddings}, sometimes with the additional assumption that the upper distortion function is affine.

\begin{rem}\label{rem:propinv}
For each proper non-decreasing function $\eta:\mathbb{R}_{\geq 0}\rightarrow \mathbb{R}_{\geq 0}$, we define another proper non-decreasing function $\tilde{\eta}:\mathbb{R}_{\geq 0}\rightarrow \mathbb{R}_{\geq 0}$ by $\tilde{\eta}(R):=\max(\eta^{-1}([0,R]))$. 
 We observe that  whenever $\eta(S)\leq R$, then $S\leq \tilde \eta(R)$. Conversely, if $R<\eta(S)$, then $\tilde \eta(R)\leq S$.
\end{rem}

\begin{defn}
A map $f:X\rightarrow Y$ is \emph{$B$-dense} if $N^Y_B(f(X))=Y$. We say that $f$ is \emph{coarsely surjective} if it is $B$-dense for some $B$. 
Two maps $f,g:X\rightarrow Y$ are \emph{$r$-close} if $\mathrm{sup}_{x\in X}d_Y(f(x),g(x))\leq r$. We say that $f$ and $g$ are \emph{close} if they are $r$-close for some $r$.
\end{defn}

Using Remark \ref{rem:propinv}, we verify that a coarse embedding $f:X\rightarrow Y$ is coarsely surjective if and only if there exists a coarse embedding $g:Y\rightarrow X$ such that $fg$ and $gf$ are close to $\mathrm{id}_Y$ and $\mathrm{id}_X$ respectively. 
We then say that $g$ is a \emph{coarse inverse} to $f$ and that $f$ is a \emph{coarse isometry}. We say $f:X\rightarrow Y$ is an \emph{$(\eta,\phi,B)$-coarse isometry} if it is an $(\eta,\phi)$-coarse embedding  and is $B$-dense.

If the distortion functions of a coarse embedding $f$ are affine, then we say $f$ is a  \emph{quasi-isometric embedding}. A coarsely surjective quasi-isometric embedding is known as a \emph{quasi-isometry}. 
We will see examples of coarse isometries and quasi-isometries in the context of group theory in Section \ref{sec:gpgeom}.

A metric space is said to be \emph{coarse geodesic} (resp. \emph{quasi-geodesic}) if it is coarsely isometric (resp. quasi-isometric) to a geodesic metric space. If $f:X\rightarrow Y$ is a coarse embedding and $X$ is a quasi-geodesic metric space, then the upper distortion function of $f$ can always be assumed to be affine. Consequently, a coarse isometry between quasi-geodesic metric spaces is a quasi-isometry.

We now show that if a space is coarse geodesic, one can approximate it by a simplicial complex known as the \emph{Rips complex}.
\begin{defn}
Let $(X,d)$ be a metric space. For each $r\geq 0$, we define the \emph{Rips complex} $P_r(X)$ to be the simplicial complex with vertex set  $X$, where $\{x_0,\dots, x_n\}\subseteq X$ spans an $n$-simplex if for all $i,j$ $$d(x_i,x_j)\leq r.$$ We define the \emph{Rips graph} $P^1_r(X)$ to be the $1$-skeleton of $P_r(X)$.
\end{defn}

If $P^1_r(X)$ is connected, it can be endowed with the path metric in which edges have length 1.  The following proposition shows that if $X$ is coarse geodesic and $r$ is sufficiently large, $X$ and $P^1_r(X)$ are coarsely isometric.

\begin{prop}\label{prop:coarsegeodmetr}
Let $(X,d)$ be a metric space. The following are equivalent:
\begin{enumerate}
\item\label{prop:coarsegeodmetr1} $X$ is  coarse geodesic;
\item\label{prop:coarsegeodmetr2} there exists a $t>0$ and a proper non-decreasing function $\eta$ such that for all $x,y\in X$, there is a $t$-chain from $x$ to $y$ of length at most $\eta(d(x,y));$
\item\label{prop:coarsegeodmetr3} there exists a $t>0$ such that for all $r\geq t$, the Rips graph $P^1_r(X)$ is connected and the inclusion $X\rightarrow P^1_r(X)$ is a coarse isometry.
\end{enumerate}
\end{prop}

\begin{proof}
$(\ref{prop:coarsegeodmetr1}) \implies (\ref{prop:coarsegeodmetr2})$: There exists a geodesic metric space $X'$ and an $(\eta,\phi,B)$-coarse isometry $f:X'\rightarrow X$. For all $x,y\in X$, there are $x',y'\in X$ such that $d_X(f(x'),x),d_X(f(y'),y)\leq B$. Letting $d:=d_{X'}(x',y')$, there is a geodesic $p:[0,d]\rightarrow X'$ from $x'$ to $y'$. We choose $n\in\mathbb{N}$ such that $n-1<d\leq n$. Let $x_{-1}=x$, $x_i=f(p(i))$ for $0\leq i <n$, $x_n=f(y')$ and $x_{n+1}=y$. Letting $t:=\max(B,\phi(1))$, we see $x_{-1},\dots, x_{n+1}$ is a  $t$-chain of length $n+2$ from $x$ to $y$. By Remark \ref{rem:propinv}, $$n+2\leq d_{X'}(x',y')+3\leq \tilde{\eta}(d_X(x,y)+2B)+3.$$

$(\ref{prop:coarsegeodmetr2}) \implies (\ref{prop:coarsegeodmetr3})$: Suppose there exist $\eta$ and $t$ such that $(\ref{prop:coarsegeodmetr2})$ holds, and let $r\geq t$.  Any two points of $X$ can be joined by a $t$-chain; hence  $P^1_r(X)$ is connected. Let $d_r$ be the induced path metric on  $P^1_r(X)$. Any $x,y\in X$ can be joined by a $t$-chain of length $n$ with $n\leq \eta(d(x,y))$. Such a $t$-chain corresponds to an edge path  of length $n$ in $P^1_r(X)$, so $d_r(x,y)\leq n$. As $(P_r^1(X),d_r)$ is a geodesic metric space,  $x$ and $y$ can be joined by an $r$-chain of length $d_r(x,y)$. By the triangle inequality, we see that  $d(x,y)\leq d_r(x,y) r$. This implies that $$\frac{d(x,y)}{r} \leq d_r(x,y)\leq \eta(d(x,y)),$$ verifying that the inclusion $X\rightarrow P^1_r(X)$ is a coarse isometry.

$(\ref{prop:coarsegeodmetr3}) \implies (\ref{prop:coarsegeodmetr1})$: This is clear, since $P^1_r(X)$  is a geodesic metric space.
\end{proof}

\subsection{Finitely Generated Groups as Geometric Objects}\label{sec:gpgeom}

\begin{defn}
Let $(X,d)$ be a metric space. For $x,y\in X$, a \emph{1-geodesic} between $x$ and $y$ is a $1$-chain $x=x_0,x_1,\dots, x_n=y$ such that $d(x_i,x_j)=|i-j|$ for each $i,j$. We say that $X$ is \emph{1-geodesic} if every pair of points can be joined by a $1$-geodesic.
\end{defn}

Let $G$ be a  group with finite generating set $S$. We equip $G$ with the word metric $d_S$ with respect to $S$.  This metric is unique up to quasi-isometry --- if $S'$ is another finite generating set then the identity map $id_G:(G,d_S)\rightarrow (G,d_{S'})$ is a quasi-isometry. Unless otherwise stated, we always assume that finitely generated groups are equipped with  the word metric with respect to some finite generating set. 
All finitely generated groups equipped with the word metric are  1-geodesic. 

If $H\leq G$ is a finitely generated subgroup equipped with the word metric, then the inclusion $\iota: H\rightarrow G$ is a coarse embedding,  but not necessarily a quasi-isometric embedding (see Corollary 1.19 and Remark 1.20 of \cite{roe2003lectures}).  Thus the \emph{intrinsic} geometry of $H$ ($H$ equipped with the word metric) and the \emph{extrinsic} geometry of $H$ ($H$ considered as a subspace of $G$, where $G$ is itself equipped with the word metric) are the same up to coarse isometry.

In the case where $H$ is a \emph{distorted} subgroup, i.e. $\iota$ is not a quasi-isometric embedding, then $\iota$ is a coarse isometry onto its image, but not a quasi-isometry onto its image. This explains why coarse isometries are necessary in this paper, and the notion of quasi-isometry is not sufficient.

\begin{defn}
Let $G$ be a group and $H\leq G$ be a subgroup. We say  $A\subseteq G$ is \emph{$H$-finite} if there is a finite $R\subseteq G$ such that $A\subseteq HR$. 
\end{defn}
\begin{rem}\label{rem:finindex}
If $H\leq K\leq G$, then $K$ is $H$-finite if and only if $H$ has finite index in $K$.
\end{rem}

\begin{lem}\label{lem:hfinitenbhd}
Let $G$ be a finitely generated group and $H$ be a subgroup. Then $K\subseteq G$ is $H$-finite   if and only if it is contained in $N_r(H)$ for some $r\geq 0$. In particular, if $H\leq K\leq G$, then $H$ has finite index in $K$ if and only if $K\subseteq N_r(H)$ for some $r\geq 0$.
\end{lem}
\begin{proof}
If $X\subseteq N_r(H)$, then every $x\in X$ can be written as $ht$, where $h\in H$ and $t\in G$ is a word of length at most $r$. Hence $X\subseteq H T$, where $T$ contains all the finitely many words in $G$ of length at most $r$.
Conversely, if $X\subseteq H T$ for some finite $T\subseteq G$, we let $r:=\mathrm{max}(\{\mathrm{length}(t)\mid t \in T\})$. Then for all $x\in X$, we can write $x=h t$ for some $h\in H$ and $t\in T$, showing that  $x\in N_r(H)$.
\end{proof}

\begin{lem}[See Lemma 2.2 of \cite{mosher2011quasiactions}]\label{lem:cocompactintersec}
Let $G$ be a finitely generated group and let $H,K\leq G$ be subgroups. Then for every $r,s\geq 0$,  $K\cap H$ has finite Hausdorff distance from  $N_r(K)\cap N_s(H)$.
\end{lem}

\begin{proof}
For every  $g\in N_r(K) \cap N_s(H)$, we choose  $k_g\in K$ and $h_g\in H$ such that  $g\in N_r(k_g) \cap N_s(h_g)$; therefore $d(e,k_g^{-1}h_g)\leq r+s$. As the set $\Lambda:=\{k_g^{-1}h_g\mid g\in N_r(K) \cap N_s(H)\}$ is finite, we choose  $g_1,\dots, g_n\in N_r(K) \cap N_s(H)$ such that $\Lambda:=\{k_{g_1}^{-1}h_{g_1},\dots, k_{g_n}^{-1}h_{g_n}\}$. 

We pick $R\geq 0$ large enough so that $\bigcup_{i=1}^n (N_r(k_{g_i}) \cap N_s(h_{g_i}) )\subseteq N_R(e)$. We now claim that $K\cap H\subseteq N_r(K) \cap N_s(H)=N_R(K\cap H)$. Indeed, if $g\in N_r(K) \cap N_s(H)$, then $k_g^{-1}h_g=k_{g_i}^{-1}h_{g_i}$ for some $1\leq i\leq n$. Therefore, as $k_gk_{g_i}^{-1}=h_gh_{g_i}^{-1}\in K\cap H$, we see that 
\begin{align*}
g\in N_r(k_g) \cap N_s(h_g)=k_gk_{g_i}^{-1}(N_r(k_{g_i}) \cap N_s(h_{g_i}))\subseteq N_R(k_gk_{g_i}^{-1})\subseteq N_R(K\cap H).&\qedhere \end{align*}\end{proof}

\begin{defn}
Two subgroups $K,H\leq G$ are said to be \emph{commensurable} if $K\cap H$ has finite index in both $K$ and $H$.
\end{defn}

\begin{prop}\label{prop:sbgphdist}
Let $G$ be a finitely generated group and  $H,K\leq G$ be subgroups. Then:
\begin{enumerate}
\item $H\subseteq N_r(K)$ for some $r\geq 0$ if and only if $H$ is commensurable to a subgroup of $K;$
\item $d_\mathrm{Haus}(K,H)$ is finite  if and only if  $H$ and $K$  are commensurable.

\end{enumerate}
\end{prop}
\begin{proof}
We observe that $H$ is commensurable to a subgroup of $K$ if and only if $H\cap K$ has finite index in $H$. Suppose there exists an $r\geq 0$ such that $H\subseteq N_r(K)$. By Lemma \ref{lem:cocompactintersec}, $H=H\cap N_r(K)$ has finite Hausdorff distance from $H\cap K$; hence by Lemma \ref{lem:hfinitenbhd}, $H\cap K$ has finite index in $H$. Conversely, Lemma \ref{lem:hfinitenbhd} says that if $H\cap K$ has finite index in $H$, then there exists an $r\geq 0$ such that $H\subseteq N_r(H\cap K)\subseteq N_r(K)$. This proves  (1); (2)  follows  from (1).
\end{proof}

\subsection{Coarse Uniform Acyclicity} Let $R$ be a commutative ring with unity.

\begin{defn}\label{defn:counac}
Let  $\lambda:\mathbb{R}_{\geq 0}\rightarrow \mathbb{R}_{\geq 0}$ and $\mu:\mathbb{R}_{\geq 0}\times \mathbb{R}_{\geq 0} \rightarrow \mathbb{R}_{\geq 0}$ be functions such that $\lambda(i)\geq i$ and $\mu(i,r)\geq r$ for all $i,r\in \mathbb{R}_{\geq 0}$.  We say that a metric space $X$ is \emph{$(\lambda,\mu)$-coarsely uniformly $n$-acyclic over $R$} if for every $k\leq n$, $x\in X$ and $i,r\in \mathbb{R}_{\geq 0}$, the map $$\tilde{H}_k(P_i(N_r(x));R)\rightarrow \tilde{H}_k(P_{\lambda(i)}(N_{\mu(i,r)}(x));R),$$ induced by inclusion, is zero.
We say that $X$ is  \emph{coarsely uniformly $n$-acyclic over $R$}  if it is $(\lambda,\mu)$-coarsely uniformly $n$-acyclic over $R$ for some suitable $\lambda$ and $\mu$.  If $X$ is coarsely uniformly $n$-acyclic over $R$ for every $n$, then we say it is \emph{coarsely uniformly acyclic over $R$}.
\end{defn}

\begin{exmp}\label{exmp:ripshyp}
If $G$ is a hyperbolic group, then  $P_i(G)$ is contractible for $i$ sufficiently large. Thus for all $n\in \mathbb{N}$ there exists a  $\mu$ such that for all $r\geq 0$, $\tilde{H}_k(P_i(N_r(e));R)\rightarrow \tilde{H}_k(P_i(N_{\mu(i,r)}(e));R)$ is zero for all $k\leq n$. As $G$ acts cocompactly on $P_i(G)$, we thus see that $G$ is coarsely uniformly $n$-acyclic over $R$. Hence $G$ is coarsely uniformly acyclic over $R$.
\end{exmp}

\begin{prop}\label{prop:unifacyccoarseinv}
If  $Y$ is $(\lambda,\mu)$-coarsely uniformly $n$-acyclic over $R$ and $f:X\rightarrow Y$ is an $(\eta,\phi,B)$-coarse isometry, then $X$ is $(\lambda',\mu')$-coarsely uniformly $n$-acyclic over $R$, where $\lambda'$ and $\mu'$ depend only on $\lambda,\mu, \eta,\phi$ and $B$.
\end{prop}

\begin{proof}
We choose $g$, a coarse inverse to $f$, such that $gf$ is $A$-close to $\mathrm{id}_X$ and both $f$ and $g$ have upper distortion function $\psi$. One can choose $\psi$ and $A$ depending only on $\eta$, $\phi$ and $B$.
For $i,r\in \mathbb{R}_{\geq 0} $ and $x\in X$, we let $i_1:=\psi(i)$ and $r_1:=\psi(r)$. Then $f$ induces a map $f_\#:C_\bullet(P_i(N^X_r(x)))\rightarrow C_\bullet(P_{i_1}(N^Y_{r_1}(f(x))))$ given by  $[x_0,\dots, x_n]\mapsto [f(x_0),\dots, f(x_n)]$ on each oriented $n$-simplex of $P_i(N^X_r(x))$.

We define $i_2:=\lambda(i_1)$, $r_2:=\mu(i_1,r_1)$ and  let $C_\bullet(P_{i_1}(N^Y_{r_1}(f(x))))\xrightarrow{\tau_\#} C_\bullet(P_{i_2}(N^Y_{r_2}(f(x))))$ be the inclusion. Then for $k\leq n$,  the map $\tilde{H}_k(P_{i_1}(N^Y_{r_1}(f(x))))\xrightarrow{\tau_*} \tilde{H}_k(P_{i_2}(N^Y_{r_2}(f(x))))$, induced by $\tau_\#$, is zero. Letting $i_3:=\psi(i_2)+A$ and $r_3:=\psi(r_2)+A$, we see that $g$ induces the chain map $C_\bullet(P_{i_2}(N^Y_{r_2}(f(x))))\xrightarrow{g_\#} C_\bullet(P_{i_3}(N^X_{r_3}(x)))$ given by  $[y_0,\dots, y_n]\mapsto [g(y_0),\dots, g(y_n)]$.  We observe that  $$g_\#\tau_\#f_\#([x_0,\dots x_n])=[gf(x_0),\dots gf(x_n)]$$ for every oriented $n$-simplex of $P_i(N^X_r(x))$. 

We define  the chain homotopy $h_\#:C_\bullet(P_i(N^X_r(x)))\rightarrow C_{\bullet+1}(P_{i_3}(N^X_{r_3}(x)))$ by  $$[x_0,\dots, x_n]\rightarrow \sum_{i=0}^n(-1)^n[x_0,\dots, x_i,gf(x_i),\dots gf(x_n)]$$ on each oriented $n$-simplex. 
We note that $\partial h_\#+h_\#\partial =g_\#\tau_\#f_\#-\iota_\#$, where $\iota_\#$ is the inclusion $C_\bullet(P_i(N^X_r(x)))\rightarrow C_\bullet(P_{i_3}(N^X_{r_3}(x)))$. Thus $\iota_*=g_*\tau_*f_*=0$ on the level of reduced homology for $k\leq n$.
\end{proof}

\begin{defn}\label{defn:bddgeom}
A metric space $(X,d)$ has \emph{bounded geometry} if for any $r>0$, there is a $K_r\in \mathbb{N}$ such that $|\{y\in X\mid d(x,y)\leq r\}|\leq K_r$ for all $x\in X$.
\end{defn}

\begin{exmp}
A non-trivial finitely generated group $G$, equipped with the word metric with respect to some finite generating set $S$, is a bounded geometry metric space. However, the Cayley graph of $G$ with the respect to $S$ is not a bounded geometry metric space.
\end{exmp}

\begin{prop}\label{prop:coarsegeodacyc}
A bounded geometry metric space is coarsely uniformly $0$-acyclic over $R$ if and only if it is coarse geodesic.
\end{prop}
\begin{proof}
Let $X$ be a bounded geometry metric space which is $(\lambda,\mu)$-coarsely uniformly $0$-acyclic over $R$. Without loss of generality, we can choose $\mu$ such that $r\mapsto \mu(0,r)$ is non-decreasing. For all $x\in X$ and $r\geq 0$, the map $\tilde{H}_0(P_0(N_r(x));R)\rightarrow \tilde{H}_0(P_{\lambda(0)}(N_{\mu(0,r)}(x));R)$, induced by inclusion, is zero; thus $N_{r}(x)$ is contained in a single connected component of $P_{\lambda(0)}(N_{\mu(0,r)}(x))$. Therefore, for all $x,y\in X$, there exists a $\lambda(0)$-chain in $N_{\mu(0,d(x,y))}(x)$ from $x$ to $y$ of length $n_{x,y}$. We may assume $n_{x,y}$ is minimal, so that no element of the $\lambda(0)$-chain is repeated, hence $n_{x,y}$ is bounded by the size of $N_{\mu(0,d(x,y))}(x)$. Since $X$ has bounded geometry, the function  $\eta(r):= \sup_{x\in X}\left | N_{\mu(0,r)}(x)\right|$ is proper and non-decreasing. Thus $n_{x,y}\leq \eta(d(x,y))$, so Proposition \ref{prop:coarsegeodmetr} tells us that $X$ is coarse geodesic.

For the converse, we observe that whenever $X$ is a geodesic metric  space, every $N_r(x)$ is connected; therefore, $X$ is coarsely uniformly $0$-acyclic. By Proposition \ref{prop:unifacyccoarseinv}, we see that every coarse geodesic space is also coarsely uniformly $0$-acyclic.
\end{proof}

\section{Metric Complexes and Coarse Cohomology}\label{sec:metriccomplex}
In this section, we define a notion of cohomology that is invariant under coarse isometry. The  cohomology we use here is defined only for bounded geometry metric spaces that are coarsely uniformly $(n-1)$-acyclic, and only then in dimensions at most $n$.
Our  coarse cohomology is  a special case of a more general notion due to Roe (see Appendix B and \cite{roe2003lectures} for more details). However, our approach is more amenable to quantitative methods than the approach taken by Roe. 

We make heavy use of technology from \cite{kapovich2005coarse} in the next two sections.
In particular, we use metric complexes as defined in the appendix of \cite{kapovich2005coarse}. Before diving straight into the theory of metric complexes,  we first give a brief  explanation of how metric complexes naturally arise when working with homological finiteness properties of groups.

\begin{defn}
A group is said to be of \emph{type $F_n$} if it acts freely, cocompactly and cellularly on an $(n-1)$-connected CW complex.
For a ring $R$, a group $G$ is said to be of \emph{type $FP_n^R$} if it admits a partial projective resolution $$P_n\rightarrow P_{n-1}\rightarrow \dots \rightarrow P_0\rightarrow R \rightarrow 0$$ of the trivial $RG$-module $R$, such that each $P_i$ is finitely generated as an $RG$-module.
\end{defn}  It is easily seen that a group of type $F_n$ is of type $FP_n^R$ for any $R$. A group of type $FP_n^\mathbb{Z}$  is of type $FP_n^R$ for any commutative ring $R$. It is shown in \cite{bestvina1997morse} that for $n\geq 2$, there exist groups of type $FP_n^R$ that are not of type $F_n$.

In \cite{drutulectures}, Drutu and Kapovich define a condition known as \emph{coarse $n$-connectedness}. Under the presence of a cocompact group action, this is  analogous to our definition of coarse uniform $n$-acyclicity, using homotopy groups instead of reduced homology groups.  The following theorem characterises groups of type $F_n$ and $FP^R_n$ in terms of their coarse geometry.

\begin{thm}[\cite{kapovich2005coarse},\cite{drutulectures}]\label{thm:finitenesscoarse} If $G$ is a finitely generated group, then:
\begin{enumerate}
\item $G$ is of type $F_n$ if and only if it is coarsely $(n-1)$-connected;
\item $G$ is of type $FP_n^R$ if and only if it is coarsely uniformly  $(n-1)$-acyclic over $R$.
\end{enumerate}
\end{thm}

We give an outline of how to construct an $(n-1)$-connected CW complex needed to prove (1). We proceed by induction. Suppose $G$ is coarsely $(n-1)$-connected and one has already constructed an $(n-1)$-dimensional, $(n-2)$-connected CW complex $X$, admitting a free, cellular and cocompact $G$-action.  The coarse $(n-1)$-connected condition ensures that one need only attach finitely many orbits of $n$-cells to $X$ to obtain an $(n-1)$-connected CW complex $X'$. In doing this, we use the Hurewicz theorem --- as $X$ is $(n-2)$-connected, $H_{n-1}(X)\cong \pi_{n-1}(X)$ can be generated by finitely many orbits of spherical $(n-1)$-cycles. We equivariantly attach $n$-cells to these spherical cycles to kill off $\pi_{n-1}(X)$ and obtain a suitable $(n-1)$-connected complex admitting a free cocompact $G$-action.

To prove (2), we would like to replicate the above argument, replacing homotopy groups with reduced homology groups  and thus building a CW complex that is $(n-1)$-acyclic over $R$ and admits a free, cellular, cocompact $G$-action --- this would certainly show that $G$ is of type $FP_n^R$. However, if one tries to replicate the above inductive argument, one runs into problems. 

Indeed, suppose a group is  coarsely uniformly $(n-1)$-acyclic and one has constructed an appropriate $(n-1)$-dimensional, $(n-2)$-acyclic CW complex $X$. Since $X$ is not necessarily $(n-2)$-connected, we cannot assume that $H_{n-1}(X)$ is generated by spherical cycles. Thus we cannot necessarily kill $H_{n-1}(X)$ by attaching $n$-cells. However, there are only finitely many orbits of (possibly non-spherical) cycles that generate $H_{n-1}(X)$. Thus we can equivariantly attach finitely many orbits of what we temporarily call  \emph{pseudo-cells}  and obtain some `complex' $X'$ that is $(n-1)$-acyclic and admits a free cocompact  $G$-action. This complex is an example of a \emph{metric complex} introduced by Kapovich and Kleiner in \cite{kapovich2005coarse}.

The complex $X'$ isn't necessarily  a CW complex since the pseudo-cells defined above may not have spherical boundary. Consequently, it may not have a topological realisation.  However, it does have a well-defined $(n-1)$-acyclic chain complex $$F\xrightarrow{\partial} C_{n-1}(X)\rightarrow \dots C_0(X)\rightarrow 0$$  which extends the cellular chain complex of $X$. Here, $F$ is a free $R$-module with $R$-basis the pseudo-cells described above.   The boundary map $F\xrightarrow{\partial} C_{n-1}(X)$ takes each pseudo-cell to the $(n-1)$-cycle it is `attached' to. This shows us that $G$ is of type $FP_n^R$.

We now proceed to formally define metric complexes. For those uncomfortable with metric complexes, we remark that much of this paper still works if one replaces metric complexes with CW complexes. More specifically, one needs to use bounded geometry CW complexes as defined in \cite{mosher2003quasi}. Doing this weakens the hypothesis of our results from groups of type $FP_{n+1}^{\mathbb{Z}_2}$ to groups of type $F_{n+1}$, but the argument is otherwise virtually unchanged.

\subsection{Metric Complexes}
Let $X$ be a bounded geometry metric space. We fix a commutative ring $R$ with unity. A \emph{free $R$-module over $X$} consists of a triple $(M,\Sigma,p)$, where $M$ is a free $R$-module, $\Sigma$ is a basis of $M$ and $p$ is a  map $p:\Sigma\rightarrow X$. We say that $X$ is the \emph{control space}, $\Sigma$ is the \emph{standard basis} and $p$ is the \emph{projection map}. 

For convenience, we will often denote a free module over $X$ simply as $M$, where the choice of $\Sigma$ and $p$ are implicit.   We say that a free module over $X$ has \emph{finite type} if $|p^{-1}(x)|$ is uniformly bounded. For each $\sigma=\sum_{b\in \Sigma}r_b b\in M$, we define the \emph{support} of $\sigma$ to be $\mathrm{supp}(\sigma):=\{p(b)\mid r_b\neq 0\}\subseteq X$. We define $\mathrm{diam}(\sigma)$  to be the diameter of $\mathrm{supp}(\sigma)$.

Say $X'$ is also a bounded geometry metric space and that $(M,\Sigma,p)$ and  $(M',\Sigma',p')$ are  free $R$-modules over $X$ and  $X'$ respectively. If there exists an $r\geq 0$, an $R$-module homomorphism  $\hat{f}:M\rightarrow M'$ and a map $f:X\rightarrow X'$ such that 
$$ \mathrm{supp}(\hat{f}(\sigma))\subseteq N_r(f(p(\sigma)))$$ for every $\sigma\in \Sigma$,
then we say that $\hat{f}$ has \emph{displacement at most $r$ over $f$}. If there exists such an $r$, then we say that $\hat{f}$ has \emph{finite displacement over $f$}. When $X=X'$, we say that $\hat{f}$ has \emph{finite displacement} when it has finite displacement over $\mathrm{id}_X$.

 \begin{lem}\label{lem:compdispmaps}
Let $(M,\Sigma,p)$, $(M',\Sigma',p')$ and $(M'',\Sigma'',p'')$ be free modules over $X$, $X'$ and $X''$ respectively. Say $f:X\rightarrow X'$ is an arbitrary map and $g:X'\rightarrow X''$ is an $(\eta,\phi)$-coarse embedding.
If $\hat{f}:M\rightarrow M'$ has displacement at most $r$ over $f$ and $\hat{g}:M'\rightarrow M''$ has displacement at most $s$ over $g$, then $\hat{g}\hat f$ has displacement at most $s+\phi(r)$ over $gf$.
\end{lem}
\begin{proof}
For every  $\sigma\in \Sigma$, we can write $\hat{f}(\sigma)=\sum_{b\in \Sigma'}r_bb$. If $r_b\neq 0$, then $p'(b)\in \mathrm{supp}(\hat{f}(\sigma))\subseteq N_r(f(p(\sigma)))$. Therefore $d_{X''}(g(p'(b)),(g\circ f)(p(\sigma)))\leq \phi(r)$, so that $$\mathrm{supp}(\hat{g}(b))\subseteq N_s(g(p'(b)))\subseteq N_{s+\phi(r)}((g\circ f)(p(\sigma))).$$ Thus $\mathrm{supp}(\hat{g}\hat f(\sigma))\subseteq N_{s+\phi(r)}((g\circ f)(p(\sigma)))$.
\end{proof}

\begin{defn}
An \emph{$R$-metric complex} consists of the tuple $(X,C_\bullet,\Sigma_\bullet,p_\bullet)$, where $X$ is a bounded geometry metric space and $C_\bullet$ is a chain complex such that: \begin{enumerate}
\item each $(C_i,\Sigma_i,p_i)$ is a free $R$-module over $X$ of finite type and each boundary map $\partial_i$ has finite displacement;
\item the composition $\varepsilon\circ \partial_1:C_1\rightarrow R$ is zero, where  $\varepsilon:C_0\rightarrow R$ is the standard augmentation given by  $\sigma\mapsto 1_R$ for each $\sigma\in \Sigma_0$;
\item the projection map $p_0:\Sigma_0\rightarrow X$ is onto.
\end{enumerate}
\end{defn}

When unambiguous, we will denote an $R$-metric complex $(X,C_\bullet,\Sigma_\bullet,p_\bullet)$ by $(X,C_\bullet)$, or even just by $C_\bullet$. We will also often refer to $R$-metric complexes simply as metric complexes, where the choice of ring $R$ is implicit.

\begin{exmp}\label{exampl:rips}
If $X$ is a bounded geometry metric space, then every $(X,C_\bullet(P_i(X);R))$ is an $R$-metric complex, where $C_\bullet(P_i(X);R)$ is the simplicial chain complex of the Rips complex $P_i(X)$. The standard basis $\Sigma_k$ of $C_k(P_i(X);R)$ is the set of $k$-simplicies of $P_i(X)$ and the projection map can be given by any map  $p_k:\Sigma_k\rightarrow X$ such that $p_k(\sigma)\in \sigma$ for each $\sigma\in \Sigma_k$. 
\end{exmp}

\begin{exmp}
A CW (or simplicial) complex with bounded geometry, as defined in  \cite{mosher2003quasi}, \cite{kapovich2005coarse} and \cite{drutulectures}, is a special case of a metric complex.
\end{exmp}

\begin{defn}
Let $[C_\bullet]_n$ denote the $n$-truncation of $C_\bullet$, i.e. the chain complex  $$\cdots \rightarrow 0 \rightarrow C_n\rightarrow C_{n-1}\rightarrow \cdots \rightarrow C_0\rightarrow 0.$$ We call $[C_\bullet]_n$ the \emph{$n$-skeleton} of $C_\bullet$. We define the \emph{dimension} of $(X,C_\bullet)$ to be $\sup\{n\in \mathbb{N}\mid C_n\neq 0\}$, which may be finite or infinite.
\end{defn}

We say that a metric complex $(X,C_\bullet)$  has \emph{$n$-displacement at most $r$} if for each $i\leq n$, the boundary map $\partial_i$  has displacement at most $r$. 
Given metric complexes $(X,C_\bullet)$, $(Y,D_\bullet)$ and a map $f:X\rightarrow Y$, we say that a chain map $f_\#:C_\bullet\rightarrow D_\bullet$  (resp. chain homotopy $h_\#:C_\bullet\rightarrow D_{\bullet+1}$) has  \emph{$n$-displacement at most $r$ over $f$} if for each $i\leq n$, $f_i$ (resp. $h_i$) has displacement at most $r$ over $f$.

Given a topological space (or CW complex) $X$ and a subspace (or subcomplex) $K$, one can define the subchain complex $C_\bullet(K)\subseteq C_\bullet(X)$. Things aren't so simple with metric complexes, since the displacement of the boundary maps  means that it is possible that $p_n(\sigma)\in K$, but  $\mathrm{supp}(\partial\sigma)$ is not contained in $K$. This motivates the following definition:

\begin{defn}
Let $(X,C_\bullet,\Sigma_\bullet,p_\bullet)$ be a metric complex. We say  a metric complex $(K,C'_\bullet,\Sigma'_\bullet,p'_\bullet)$ is a \emph{subcomplex} of $(X,C_\bullet,\Sigma_\bullet,p_\bullet)$ if $K\subseteq X$, $C'_\bullet$ is subchain complex of $C_\bullet$ and for each $i$,  $\Sigma_i'\subseteq\Sigma_i$ and $p'_i=p_i|_{\Sigma'_i}$. For any $K\subseteq X$, we define the subcomplex generated by $K$ to be the largest metric complex $(K,C_\bullet[K],\Sigma'_\bullet,p'_\bullet)$ that is a subcomplex of $(X,C_\bullet,\Sigma_\bullet,p_\bullet)$. We let $C_\bullet[K]_n$ denote the $n$-skeleton of $C_\bullet[K]$.
\end{defn}

\begin{rem}\label{rem:subcomplexexplicit}
The subcomplex $(K,C_\bullet[K],\Sigma'_\bullet,p'_\bullet)$ can be described explicitly as follows. We let $\Sigma'_0=p^{-1}(K)$, with $C_0[K]$ the free module generated by $\Sigma'_0$. We inductively define  $$\Sigma'_{k+1}:=\{\sigma\in \Sigma_{k+1}\mid \partial \sigma\in C_k[K] \textrm{ and } p_{k+1}(\sigma)\in K\}$$ and let $C_{k+1}[K]$ be the free module generated by $\Sigma'_{k+1}$.

\end{rem}

\begin{lem}\label{lem:dispmetric}
Let $(X,C_\bullet,\Sigma_\bullet,p_\bullet)$ be a metric complex and suppose $\mathrm{supp}(\rho)\subseteq K$ for some $\rho\in C_n$. If $C_\bullet$ has $n$-displacement at most $r$, then $\rho \in C_n[N_{nr}(K)]$.
\end{lem}
\begin{proof}
We prove this by induction. The case $n=0$ is clear, since each  $\sigma\in \Sigma_0$ lies in $C_\bullet[K]$ precisely when $p_0(\sigma)\in K$.
Let $k<n$. We assume for all $L\subseteq X$ and $\rho\in \Sigma_{k}$, that if $p_{k}(\rho)\in L$, then $\rho\in C_\bullet[N_{kr}(L)]$. Suppose  $p_{k+1}(\sigma)\in K$ for some $\sigma\in \Sigma_{k+1}$. Since $\partial_{k+1}$ has displacement at most $r$, $\mathrm{supp}(\partial_{k+1}\sigma)\subseteq N_r(K)$.  Therefore $\partial_{k+1}\sigma$  is contained in $C_\bullet[N_{kr}(N_r(K))]\subseteq C_\bullet[N_{(k+1)r}(K)]$. By Remark \ref{rem:subcomplexexplicit}, we see that $\sigma \in  C_\bullet[N_{(k+1)r}(K)]$. Hence if $\mathrm{supp}(\rho)\subseteq K$ for some $\rho\in C_{k+1}$, then $\rho \in C_{k+1}[N_{(k+1)r}(K)]$.
\end{proof}
\begin{cor}\label{cor:findispmap}
Suppose $(X,C_\bullet)$ and $(Y,D_\bullet)$ are metric complexes and $f_\#:C_\bullet\rightarrow D_\bullet$ has finite $n$-displacement over $f:X\rightarrow Y$. Then there exists an $r\geq 0$, depending on the $n$-displacement of $f_\#$ and $D_\bullet$, such that $f_\#([C_\bullet]_n)\subseteq D_\bullet[N_r^Y(f(X))]$.
\end{cor}

\begin{defn}
We define the cochain complex $C^\bullet[K]:=\mathrm{Hom}_{R}(C_\bullet[K],R)$ and let $\delta^{k}:C^k[K]\rightarrow C^{k+1}[K]$ denote the coboundary map dual to the boundary map $\partial_{k+1}$. For every $\alpha\in C^k[K]$, we define its \emph{support} to be $\mathrm{supp}(\alpha):=\{p_k(\sigma) \mid \sigma\in \Sigma_k\textrm{ and } \alpha(\sigma)\neq 0\}$. 

As  $X$ has bounded geometry and each $\partial_k$ has finite displacement, we see that  $\delta^k$ preserves cochains of finite support. We thus define $C^\bullet_c[K]$ to be the subcochain complex of $C^\bullet[K]$ consisting of cochains with finite support. 
\end{defn}

We let $H_k(C_\bullet[K])$, $H^k(C_\bullet[K])$ and $H^k_c(C_\bullet[K])$ denote the $k^\mathrm{th}$ homology/cohomology of $C_\bullet[K]$, $C^\bullet[K]$ and $C^\bullet_c[K]$ respectively. We can also take the reduced homology $\tilde{H}_k(C_\bullet[K])$, which is the $k^{th}$ homology of the augmented chain complex $\cdots \rightarrow C_1[K]\rightarrow C_0[K]\xrightarrow{\varepsilon} R$.

\begin{defn}
Given  a metric complex $(X,C_\bullet)$ and $K\subseteq X$, we define $$\hat{H}^k_c(C_\bullet[K])\leq H^k_c(C_\bullet[K]):=\ker(H^k_c(C_\bullet[K])\xrightarrow{\iota^*} H^k(C_\bullet[K])),$$ where the map $\iota^*$ is induced by the inclusion $C^\bullet_c[K]\xrightarrow{\iota} C^\bullet[K]$. We call $\hat{H}^k_c$  the \emph{modified cohomology with compact supports}.
\end{defn}

 It will often be convenient to denote $H_k(C_\bullet[K])$, $H^k(C_\bullet[K])$, $H^k_c(C_\bullet[K])$ and $\hat H^k_c(C_\bullet[K])$    by  $H_k[K]$, $H^k[K]$, $H^k_c[K]$ and $\hat H^k_c[K]$  respectively. This notation implicitly assumes the choice of some metric complex $(X,C_\bullet)$, but can be used when there  is no ambiguity.

For $K\subseteq X$, we define $C_\bullet[X,K]=C_\bullet[X]/C_\bullet[K]$. Then there is a long exact sequence $$\cdots \rightarrow H_2[X,K]\rightarrow H_1[K]\rightarrow H_1[X]\rightarrow H_1[X,K]\rightarrow H_0[K]\rightarrow H_0[X]\rightarrow H_0[X,K]\rightarrow 0$$ and similar such sequences for cohomology and cohomology with compact supports.

\begin{defn}
Let $(X,C_\bullet)$ and  $(Y,D_\bullet)$ be metric complexes. We say a chain map $f_\#:C_\bullet\rightarrow D_\bullet$  (resp. chain homotopy $h_\#:C_\bullet\rightarrow D_{\bullet+1}$) is \emph{proper} if there is a coarse embedding $f:X\rightarrow Y$ such that for every $k$, $f_\#$ (resp. $h_\#$) has finite $k$-displacement over $f$.
\end{defn}

The following proposition shows that proper chain maps between metric complexes induce maps in cohomology with compact supports.
\begin{prop}\label{prop:propchainmap}
Let $(X,C_\bullet,\Sigma_\bullet,p_\bullet)$ and $(Y,D_\bullet)$ be metric complexes.
\begin{enumerate}
\item \label{prop:propchainmapmap} A proper chain map $f_\#:[C_\bullet]_n\rightarrow [D_\bullet]_n$ induces maps  $\hat f^*:\hat H^k_c(D_\bullet)\rightarrow \hat H^k_c(C_\bullet)$ for $k\leq n$.
\item \label{prop:propchainmaphomotop} If proper chain maps $f_\#,g_\# :[C_\bullet]_n\rightarrow [D_{\bullet}]_n$ are chain homotopic via a proper chain homotopy $h_\#:[C_\bullet]_{n-1}\rightarrow [D_{\bullet+1}]_{n-1}$, then  $\hat f^*=\hat g^*$ for $k\leq n$.
\end{enumerate}
\end{prop}
\begin{proof}
(\ref{prop:propchainmapmap}): The chain map $f_\#$ induces a dual map $f^\#$ given by $\alpha\mapsto\alpha f_\#$. As $f_\#$ is proper, it has $n$-displacement at most $r$ over some coarse embedding $f:X\rightarrow Y$. 
We claim that for $k\leq n$ and $\alpha\in D^k_c$, $f(\mathrm{supp}(f^\#\alpha))\subseteq N_r^Y(\mathrm{supp}(\alpha))$. Indeed, suppose $x\in \mathrm{supp}(f^\#\alpha)$; then there exists a $\Delta\in p_k^{-1}(x)$ such that $\alpha(f_\#\Delta)=(f^\#\alpha)(\Delta)\neq 0$. Thus $\mathrm{supp(\alpha)}$ intersects  $\mathrm{supp}(f_\#\Delta)\subseteq N^Y_r(f(x))$, so $f(x)\in N_r^Y(\mathrm{supp}(\alpha))$, proving the claim. 

Since $f$ is a coarse embedding, the above claim demonstrates that $f^\#$ preserves cochains of finite support. 
Suppose $k\leq n$ and $\alpha\in D^k_c$ is a cocycle such that $\alpha=\delta\beta$ for some $\beta\in D^{k-1}$; then $f^\#\alpha=\delta f^\# \beta$ where $f^\#\beta\in C^{k-1}$. It is straightforward to verify that $[\alpha]\mapsto [f^\#\alpha]$ is a well-defined map in modified cohomology with compact supports.

(\ref{prop:propchainmaphomotop}): A proper chain homotopy $h_\#:[C_\bullet]_{n-1}\rightarrow [D_{\bullet+1}]_{n-1}$ induces a dual map $h^\#$ such that $\delta h^\#+h^\#\delta=g^\#-f^\#$ in dimensions at most $n-1$. The above argument shows that $h^\#$ preserves cochains of finite support. Suppose $k\leq n$ and $\alpha\in D^k_c$ is a cocycle such that $\alpha=\delta\beta$ for some $\beta\in D^{k-1}$. Then $(g^\#-f^\#)\alpha=\delta(g^\#-f^\#)\beta=\delta(\delta h^\#+h^\#\delta)\beta=\delta h^\# \alpha$. Thus $\hat f^*=\hat g^*$ for $k\leq n$.
\end{proof}
\begin{rem}
Using ordinary (not modified) cohomology with compact supports,  an analogue of Proposition \ref{prop:propchainmap} holds for $k<n$. 
This is our motivation for introducing modified cohomology with compact supports, as it gives us an invariant one dimension higher than we  would otherwise have.
\end{rem}

\begin{defn}
A  \emph{group action} of $G$ on  $(X,C_\bullet,\Sigma_\bullet,p_\bullet)$ consists of a pair $(\rho,\hat{\rho})$, where $\rho:G \curvearrowright X$ and $\hat\rho: G \curvearrowright C_\bullet$ are group actions by isometries and chain automorphisms respectively, such that for each $i$ and $g\in G$, $\hat{\rho}(g)(\Sigma_i)=\Sigma_i$ and $p_i$ is $G$-equivariant. This action is \emph{free} (resp. \emph{cocompact}) if the action $\rho:G\curvearrowright X$ is free (resp. cocompact). 
\end{defn}

Whenever a metric complex $(G,C_\bullet)$ admits a $G$-action, it will always be assumed that $\rho:G \curvearrowright G$ is the action by left multiplication.

\begin{exmp}
Our conditions for a metric complex to admit a $G$-action are reasonably restrictive. For instance, if a group $G$ has torsion, then for sufficiently large $r$, $G$ cannot act on $(G,P_r(G))$ despite the fact that $G$ has a natural simplicial action on $P_r(G)$. This is because some $g\in G\backslash \{e\}$ fixes a simplex of $P_r(G)$, so the projections maps can never be $G$-equivariant. If $G$ is torsion-free, then projection maps can be chosen so that $G$ acts freely on $(G,P_r(G))$.
\end{exmp}

\subsection{Uniformly Acyclic Complexes}
\begin{defn}[cf. Definition \ref{defn:counac}]
Let $\mu:\mathbb{R}_{\geq 0}\rightarrow \mathbb{R}_{\geq 0}$ be a function such that $\mu(r)\geq r$ for each $r$. An $R$-metric complex $(X,C_\bullet)$ is said to be \emph{$\mu$-uniformly $n$-acyclic} if for every $x\in X$, $r\in \mathbb{R}_{\geq 0}$ and $k\leq n$, the map  $$\tilde{H}_k[B_r(x)]\rightarrow \tilde{H}_k[B_{\mu(r)}(x)],$$ induced by inclusion, is zero.
We say $(X,C_\bullet)$ is \emph{uniformly $n$-acyclic} if it is $\mu$-uniformly $n$-acyclic for some $\mu$.
We say that $(X,C_\bullet)$ is \emph{uniformly acyclic} if it is uniformly $n$-acyclic for all $n$.
\end{defn}

As in Example \ref{exmp:ripshyp}, if $G$ is a hyperbolic group, then  $C_\bullet(P_i(G))$ is uniformly acyclic  for $i$ sufficiently large. Unfortunately, for a general coarsely uniformly $(n-1)$-acyclic metric space $X$,  $C_\bullet(P_i(X))$  is not necessarily uniformly $(n-1)$-acyclic for large  $i$.

The following proposition allows us to construct a uniformly acyclic metric complex for any coarsely uniformly acyclic metric space, formalizing the construction discussed in the introduction to this section.

\begin{restatable}[See Lemma 5.10 and Proposition 11.4 of \cite{kapovich2005coarse}]{prop}{controlspace}\label{prop:controlspace}
Let $X$ be a bounded geometry metric space. 
\begin{enumerate}
\item\label{eqn:controlspacemain}
If $X$ is $(\lambda,\mu)$-coarsely uniformly $(n-1)$-acyclic over $R$, then it is the control space of a $\mu'$-uniformly $(n-1)$-acyclic $R$-metric complex $(X,C_\bullet)$ of $n$-displacement at most $d$, where $d$ and $\mu'$ depend only on $\lambda$ and $\mu$.
\item \label{eqn:controlspaceinf} Suppose that for each $n$, $X$ is $(\lambda_n,\mu_n)$-coarsely uniformly $(n-1)$-acyclic over $R$. Then $X$ is the control space of a uniformly acyclic $R$-metric complex $(X,C_\bullet)$ such that for each $n$,  $(X,C_\bullet)$  is $\mu'_n$-uniformly $(n-1)$-acyclic and has $n$-displacement at most $d_n$, where $d_n$ and $\mu'_n$ depend only on $\lambda_i$ and $\mu_i$ for $i\leq n$.
\item \label{eqn:controlspaceequiv} Suppose a group $G$ acts freely on $X$. Then  (\ref{eqn:controlspacemain}) and (\ref{eqn:controlspaceinf})  hold, and the resulting metric complex $(X,C_\bullet)$ can be chosen so that it admits a free $G$-action.
\end{enumerate}
\end{restatable}

Coupled with the preceding theorem, the following lemma is essential in allowing us to define coarse cohomology. It is a metric complex version of Proposition 6.47 and Corollary 6.49 of \cite{drutulectures}.
\begin{restatable}{lem}{extendmapsunif}\label{lem:extendmapsunif}
Suppose $(X,C_\bullet,\Sigma_\bullet,p_\bullet)$ and $(Y,D_\bullet,\Sigma'_\bullet,p'_\bullet)$ are $R$-metric complexes,  $(Y,D_\bullet)$ is  $\mu$-uniformly $(n-1)$-acyclic and  $C_\bullet$ and $D_\bullet$ have $n$-displacement at most $d_1$ and $d_2$ respectively. Then:
\begin{enumerate}
\item
every $(\eta,\phi)$-coarse embedding $f:X\rightarrow Y$  induces a chain map $f_\#:[C_\bullet]_n\rightarrow D_\bullet$ of $n$-displacement at most $M=M(\mu,\phi,d_1,d_2)$ over $f$;
\item for every $(\eta,\phi)$-coarse embedding $f:X\rightarrow Y$ and every pair of chain maps $f_\#,g_\#:[C_\bullet]_n\rightarrow D_\bullet$ of $n$-displacement at most $r$ over $f$, there exists a chain homotopy $h_\#:[C_\bullet]_{n-1}\rightarrow D_{\bullet+1}$ between $f_\#$ and $g_\#$ which has $(n-1)$-displacement at most $N=N(\mu,\phi,d_1,d_2,r)$ over $f$.
\end{enumerate}
\end{restatable}

We also need the following Lemma,  used in the proof of Proposition \ref{prop:controlspace}. It is similar to Lemma \ref{lem:extendmapsunif}, replacing the metric complex $(Y,D_\bullet)$ by suitable Rips complexes.
\begin{restatable}{lem}{extendmapscoarse}\label{lem:extendmapscoarse}
Suppose  $X$ and $Y$ are bounded geometry metric spaces, $(X,C_\bullet, \Sigma_\bullet,p_\bullet)$ is an  $R$-metric complex with $n$-displacement at most $d$ and $Y$ is $(\lambda,\mu)$-coarsely uniformly $(n-1)$-acyclic over $R$.
\begin{enumerate}
\item There exists an $i=i(\lambda)$ such that any $(\eta,\phi)$-coarse embedding $f:X\rightarrow Y$ induces a chain map $$f_\#:[C_\bullet]_n\rightarrow C_\bullet(P_i(Y);R)$$ of $n$-displacement at most $M=M(\lambda,\mu, \phi,d)$ over $f$. 
\item For $i\geq 0$, suppose  $f_\#,g_\#:[C_\bullet]_n\rightarrow C_\bullet(P_i(Y);R)$ are chain maps of  $n$-displacement at most $r$ over an $(\eta,\phi)$-coarse embedding $f:X\rightarrow Y$. Then there exists a $j=j(i,\lambda)$ and a chain homotopy $$h_\#:[C_\bullet]_{n-1}\rightarrow C_{\bullet+1}(P_j(Y);R),$$  of $(n-1)$-displacement at most $N=N(i,\lambda,\mu, \phi, r)$ over $f$, between $\iota_\# f_\#$ and $\iota_\# g_\#$, where $\iota_\#:C_{\bullet}(P_i(Y);R)\rightarrow C_{\bullet}(P_j(Y);R)$ is the inclusion. 
\end{enumerate}
\end{restatable}

Proposition \ref{prop:controlspace} and Lemmas \ref{lem:extendmapsunif} and \ref{lem:extendmapscoarse} can be proved by applying techniques found in \cite{kapovich2005coarse} and \cite{drutulectures}. For the reader's convenience, we include these proofs in Appendix \ref{app:metriccomp}.

\subsection{Coarse Cohomology}\label{sec:coarsecohom}
Suppose $(X,C_\bullet)$ and $(X,C'_\bullet)$ are two  uniformly $(n-1)$-acyclic $R$-metric complexes. 
Applying Lemma \ref{lem:extendmapsunif} to $\mathrm{id}_X$, we define proper chain maps $f_\#:[C_\bullet]_n\rightarrow [C'_\bullet]_n$ and $g_\#:[C'_\bullet]_n\rightarrow [C_\bullet]_n$ of finite $n$-displacement over the identity. By Lemma \ref{lem:compdispmaps}, both  $g_\#f_\#$ and $f_\#g_\#$ have finite $n$-displacement over the identity, thus Lemma \ref{lem:extendmapsunif} says they are chain homotopic to $\mathrm{id}_{[C_\bullet]_{n-1}}$ and $\mathrm{id}_{[C'_\bullet]_{n-1}}$ respectively.
By Proposition \ref{prop:propchainmap}, $f_\#$ and $g_\#$ induce maps  $\hat f^*:\hat H^k_c(C'_\bullet)\rightarrow\hat H^k_c(C_\bullet)$ and $\hat g^*:\hat H^k_c(C_\bullet)\rightarrow\hat H^k_c(C'_\bullet)$ for $k\leq n$. Proposition \ref{prop:propchainmap} also tells us that each of $\hat f^*\hat g^*$ and $\hat g^*\hat f^*$ is the identity, thus $\hat f^*$ is an isomorphism.

Moreover, any two chain maps $f_\#,f'_\#:[C_\bullet]_n\rightarrow [C'_\bullet]_n$ of finite displacement over $\mathrm{id}_X$ are properly chain homotopic by Lemma \ref{lem:extendmapsunif}. Thus by Proposition \ref{prop:propchainmap}, $\hat f^*$ doesn't depend on the choice of $f_\#$. Consequently, for $k\leq n$ the isomorphism $\hat f^*:\hat H^k_c(C'_\bullet)\rightarrow\hat H^k_c(C_\bullet)$ is canonical, so $\hat H^k_c(C_\bullet)$ is effectively an invariant of $X$, independent of the choice of $C_\bullet$. This discussion demonstrates that the following is well-defined:

\begin{defn}\label{def:coarsecohom}
Let $X$ be a bounded geometry metric space which is coarsely uniformly $(n-1)$-acyclic over $R$. Let $(X,C_\bullet)$ be a uniformly $(n-1)$-acyclic $R$-metric complex, whose existence is guaranteed by Proposition \ref{prop:controlspace}. For $k\leq n$, we define the \emph{coarse cohomology} of $X$ to be $H^k_\mathrm{coarse}(X;R):=\hat H^k_c(C_\bullet)$. 
\end{defn}

\begin{rem}\label{rem:0thcohom}
It follows easily from the definition of modified cohomology with compact supports that $\hat{H}^0_c(C_\bullet)=0$ for every metric complex $(X,C_\bullet)$. Thus for every bounded geometry metric space, $H^0_\mathrm{coarse}(X;R)=0$.
\end{rem}

The subsequent lemma demonstrates that it is only necessary to use modified cohomology with compact supports to define $H^n_\mathrm{coarse}(X;R)$   when $X$ is coarsely uniformly $(n-1)$-acyclic over $R$ and not coarsely uniformly $n$-acyclic over $R$. Otherwise, we can just take ordinary (not modified) cohomology with compact supports of a suitable metric complex.
\begin{lem}\label{lem:modifvsordin}
Suppose $X$ is an infinite metric space, $n>0$ and $(X,C_\bullet,\Sigma_\bullet,p_\bullet)$ is a uniformly $(n-1)$-acyclic $R$-metric complex. Then for $k<n$, $H^k_c(C_\bullet)=\hat H^k_c(C_\bullet)$.
\end{lem}
\begin{proof}
Since $C_\bullet$ is $(n-1)$-acyclic,  $H^k(C_\bullet)=0$ for $0<k<n$ and $H^0(C_\bullet)\cong R$. This can be seen using the universal coefficients theorem for cohomology, or can be calculated directly. Therefore, for $0<k<n$, $H^k_c(C_\bullet)=\ker(H^k_c(C_\bullet)\xrightarrow{\iota^*}  H^k(C_\bullet))=\hat H^k_c(C_\bullet)$.

Suppose $\alpha\in C^0_c$ is a cocycle and  $\Delta\in \Sigma_0$. For any $\Delta'\in \Sigma_0$, $\Delta-\Delta'$ is a reduced $0$-cycle, hence there is a 1-chain $\rho\in C_1$ such that $\partial\rho=\Delta-\Delta'$. Therefore $\alpha(\Delta)-\alpha(\Delta')=\alpha(\partial\rho)=\delta\alpha(\rho)=0$. Since $X$ is infinite and $\alpha$ has finite support, we see that $\alpha(\Delta)=0$. As $\Delta$ was arbitrary,  $\alpha=0$; hence $H^0_c(C_\bullet)=0$. Therefore $\hat H^0_c(C_\bullet)=H^0_c(C_\bullet)=0$.
\end{proof}

\begin{prop}\label{prop:coarsecohommap}
Let $X$ and $Y$ be bounded geometry metric spaces that are coarsely uniformly  $(n-1)$-acyclic over $R$. Then a coarse embedding $f:X\rightarrow Y$ induces a  homomorphism $f^*:H^k_\mathrm{coarse}(Y;R)\rightarrow H^k_\mathrm{coarse}(X;R)$ for all $k\leq n$.
Moreover, for any bounded geometry metric spaces $X$, $Y$ and $Z$ that are coarsely uniformly  $(n-1)$-acyclic over $R$, then:
\begin{enumerate}
\item $(\mathrm{id}_X)^*=\mathrm{id}_{H^k_\mathrm{coarse}(X;R)}$;
\item if $f,g:X\rightarrow Y$ are  close coarse embeddings, then $f^*=g^*$;
\item if $f:X\rightarrow Y$ and $g:Y\rightarrow Z$ are coarse embeddings, then $(gf)^*=f^*g^*$.
\end{enumerate}
In particular, if $f:X\rightarrow Y$ is a coarse isometry, then $f^*$ is an isomorphism.
\end{prop}
\begin{proof}
Let $(X,C_\bullet)$, $(X,C'_\bullet)$, $(Y,D_\bullet)$ and $(Y,D'_\bullet)$ be uniformly $(n-1)$-acyclic $R$-metric complexes, which necessarily exist by Proposition \ref{prop:controlspace}.
$$\begin{CD}
\hat H^k_c(D'_\bullet)  @>\hat f'^*>> \hat H^k_c(C'_\bullet)\\
@VV(\mathrm{id}_Y)^*V @VV(\mathrm{id}_X)^*V\\
\hat H^k_c(D_\bullet) @>\hat f^*>> \hat H^k_c(C_\bullet)
\end{CD}\hspace{0.3cm}.$$ 
 We now use Part (1) of Lemma \ref{lem:extendmapsunif} and Proposition \ref{prop:propchainmap} to construct the diagram shown above. 
 Part (2) of Lemma \ref{lem:extendmapsunif} and Proposition \ref{prop:propchainmap} ensure that this diagram commutes. Therefore, in view of the discussion preceding Definition \ref{def:coarsecohom}, $f$ does indeed  induce a map $f^*:H^k_\mathrm{coarse}(Y;R)\rightarrow H^k_\mathrm{coarse}(X;R)$, independent of the choice of $C_\bullet$ and $D_\bullet$. 
 Parts (1)--(3) follow from Lemmas \ref{lem:compdispmaps} and \ref{lem:extendmapsunif}.
\end{proof}

\begin{rem}\label{rem:quantity}
Let $X$ and $Y$ be bounded geometry metric spaces that are $(\lambda,\mu)$-coarsely uniformly $(n-1)$-acyclic over $R$, and suppose $f:X\rightarrow Y$ is an $(\eta,\phi)$-coarse embedding. Using Proposition \ref{prop:controlspace}, we can construct $\mu'$-uniformly $(n-1)$-acyclic  $R$-metric complexes $(X,C_\bullet)$ and $(Y,D_\bullet)$ of $n$-displacement at most $d$, where $\mu'$ and $d$ depend only on $\lambda$ and $\mu$.  Therefore, by Lemma \ref{lem:extendmapsunif}, we can construct a map $f_\#:[C_\bullet]_n\rightarrow D_\bullet$ of $n$-displacement at most $D=D(\lambda,\mu, \eta,\phi)$ over $f$. Thus there is some quantitative information associated to the induced map $f^*:H^k_\mathrm{coarse}(Y;R)\rightarrow H^k_\mathrm{coarse}(X;R)$. This idea will be developed fully in Section \ref{sec:mvseq}.
\end{rem}

There is a more general notion of coarse cohomology  due to Roe  that doesn't require coarse uniform acyclicity (see \cite{roe2003lectures}). When defined, $H^*_\mathrm{coarse}$ is  naturally  isomorphic to Roe's coarse cohomology. However, working with anti-\v{C}ech approximations \`a la Roe, we lose the quantitative information (e.g. Remark \ref{rem:quantity}) that we have when working with metric complexes. As Kapovich and Kleiner say in \cite{kapovich2005coarse}, ``one inevitably loses quantitative information
which is essential in many applications of coarse topology to
quasi-isometries and geometric group theory". We shall give an account of Roe's coarse cohomology in Appendix \ref{app:roecoarse}.

\subsection{Finiteness Properties of Groups and Group Cohomology}\label{sec:groupcohom}
We recall that a finitely generated group $G$ can be thought of as a bounded geometry metric space by endowing it with the word metric with respect to some finite generating set. Consider a free $RG$-module $M$ with $RG$-basis $\mathcal{B}$, so that $G\cdot \mathcal{B}$ is an $R$-basis of $M$. We can define a projection $p:G\cdot \mathcal{B}\rightarrow G$ so that for each $gb\in G\cdot \mathcal{B}$, $p(gb)=g$. Then $(M,G\cdot \mathcal{B},p)$ has the structure of a free $R$-module over $G$. Moreover, $M$ is finitely generated as an $RG$-module if and only if it is of  finite type. If $f:M\rightarrow N$ is an $RG$-module homomorphism between free $RG$-modules and $M$ is finitely generated as an $RG$-module, then $f$ has  finite displacement over $\mathrm{id}_G$.

A group is of type $FP_n^R$ if the trivial $RG$-module $R$ has a partial projective resolution of length $n$ consisting of finitely generated $RG$-modules. Such a partial projective resolution is in fact an $n$-dimensional, uniformly $(n-1)$-acyclic metric complex with control space $G$.

Conversely, suppose one has an $n$-dimensional uniformly $(n-1)$-acyclic metric complex $(G,C_\bullet)$ with control space $G$, admitting a $G$-action. This is a partial projective resolution of the trivial $RG$-module $R$ of length $n$, consisting of finitely generated $RG$-modules. By applying Brown's lemma \cite{brown1987finiteness}, it is straightforward to see that a group is of type $FP_n^R$ if and only if it is coarsely uniformly $(n-1)$-acyclic over $R$, proving Part (2) of Theorem \ref{thm:finitenesscoarse}.

\begin{prop}[{\cite{brown1982cohomology},\cite{geoghegan1986note},\cite{roe2003lectures}}]\label{prop:chargroupcohom}
Let $G$ be a group of type $FP_n^R$. Then for $i\leq n$, $H^i(G,RG)\cong H^i_\mathrm{coarse}(G;R)$ as right $RG$-modules.
\end{prop}
This follows easily by applying \cite[VIII Lemma 7.4]{brown1982cohomology} which states that if $M$ is an $RG$-module, there is a natural isomorphism $\mathrm{Hom}_{RG}(M,RG)\cong \mathrm{Hom}_c(M,R)$. Here, $\mathrm{Hom}_c(M,R)$ consists of all $R$-module homomorphisms $f:M\rightarrow R$ such that for every $m\in M$, $f(gm)=0$ for all but finitely many $g\in G$.  A little more care is required in the $i=n$ case, in which one has to use modified cohomology with compact supports; this is done using an argument identical to one found in \cite{geoghegan1986note}.  Alternatively, one can use Appendix \ref{app:roecoarse} and \cite[Example 5.21]{roe2003lectures}, which shows that Roe's coarse cohomology is isomorphic to  group cohomology with group ring coefficients.

\subsection{Inverse Limits}\label{sec:invlim}
We briefly review  some basic properties of inverse systems  necessary to discuss the topology at infinity of spaces. 
All statements will be true in the category of $R$-modules and chain complexes of $R$-modules, and can be found in \cite{weibel1994introduction}. 
\begin{defn}

An \emph{inverse system} indexed by $\mathbb{N}$ is a sequence $(A_i)_{i\in \mathbb{N}}$  of objects equipped with morphisms $f_{i}^j:A_j\rightarrow A_i$ for all $j\geq i$,  such that: \begin{itemize}
\item $f_i^i$ is the identity for all $i$;
\item $f_i^k=f_i^j\circ f_j^k$ for all $i\leq j\leq k$.
\end{itemize} The maps $f_\bullet^\bullet$ are known as \emph{bonds}.
An \emph{inverse limit} of an inverse system $(A_i)$ consists of an object $A$ and morphisms $f_i:A\rightarrow A_i$ for each $i$, such that:
\begin{itemize}
\item $f_i=f_i^j f_j$ for all $i\leq j$;
\item if there is an object $Z$ and morphisms $q_i:Z\rightarrow A_i$ such that $q_i=f_i^jq_j$ for all $i\leq j$, then there is a morphism $h:Z\rightarrow A$ such that $q_i=f_i h$ for all $i$.
\end{itemize}
We denote the inverse limit by $\varprojlim A_\bullet$, which always exists (for $R$-modules and chain complexes of $R$-modules) and is unique up to isomorphism.
\end{defn}

There are dual notions of  \emph{direct systems} and \emph{direct limits}, and we denote the direct limit of the direct system $(A_\bullet)$ by $\varinjlim A_\bullet$.

\subsection{Topology at Infinity}\label{sec:topatinf}
For the remainder of this section, we set $R=\mathbb{Z}_2$ and assume all homology and cohomology is taken with coefficients in $\mathbb{Z}_2$ (or $\mathbb{Z}_2G$ when considering group cohomology). 
We give a characterisation of coarse cohomology in terms of topology at infinity. This section uses an argument from \cite{geoghegan1985free} (see also \cite{geoghegan2008topological}), although it is considerably simpler here as we are working over $\mathbb{Z}_2$. 

Let $X$ be a bounded geometry space which is  coarsely uniformly $n$-acyclic over $\mathbb{Z}_2$ and let $(X,C_\bullet)$ be a uniformly $n$-acyclic $\mathbb{Z}_2$-metric complex. A \emph{finite filtration} of $X$ is a nested sequence $K_1\subseteq K_2\subseteq K_3\subseteq \dots $ of finite subsets of $X$ such that $\cup_i K_i=X$. The \emph{$k^\mathrm{th}$ reduced homology at infinity} of $X$ is defined to be the inverse limit $\varprojlim_i \tilde{H}_k[X\backslash K_i]$, where the bonds of this inverse system are the maps induced by inclusion. We say that $X$ is \emph{$r$-acyclic at infinity over $\mathbb{Z}_2$} if $\varprojlim_i \tilde{H}_k[X\backslash K_i]$ vanishes for $k\leq r$. These notions are independent of the choice of finite filtration and metric complex.

\begin{prop}[\cite{geoghegan1985free},\cite{geoghegan2008topological}]\label{prop:topatinf}
Let $(X,C_\bullet,\Sigma_\bullet,p_\bullet)$ be a uniformly $n$-acyclic $\mathbb{Z}_2$-metric complex, where $n\geq 0$ and $X$ is infinite. For $k\leq n$ and $d\in \mathbb{N}$, $\varprojlim_i \tilde{H}_k[X\backslash K_i]$ is $d$-dimensional precisely when $H^{k+1}_\mathrm{coarse}(X;\mathbb{Z}_2)$ is. In particular for $r\leq n$, $X$ is $r$-acyclic at infinity over $\mathbb{Z}_2$ if and only if $H^k_\mathrm{coarse}(X;\mathbb{Z}_2)=0$ for $k\leq r+1$.
\end{prop}
\begin{proof}

We first show that $H_{k}[X\backslash K_i]$ is finitely generated for every $i\in\mathbb{N}$ and every $k\leq n$. Let $j>i$ be large enough so that for $k\leq n$,  every element  of $\Sigma_k$ lies in either $C_\bullet[K_j]_n$ or $C_\bullet[X\backslash K_i]_n$. This can be done using Lemma \ref{lem:dispmetric}. A standard argument produces the Mayer--Vietoris sequence $$\cdots \rightarrow H_{k}[K_j\backslash K_i]\rightarrow H_{k}[K_j]\oplus H_{k}[X\backslash K_i] \rightarrow H_{k}[X]\rightarrow \cdots.$$ As $K_j$ and $K_j\backslash K_i$ are finite, $H_{k}[K_j\backslash K_i]$ and $H_{k}[K_j]$ are finitely dimensional. Since $H_{k}[X]$ is zero, we see $H_{k}[X\backslash K_i]$ is a finite dimensional vector space of dimension $d_i^k$. 

For each $K_i$ we have  a long exact sequence $$ \tilde{H}^0[X] \rightarrow \tilde{H}^0[X\backslash K_i] \rightarrow H^1[X,X\backslash K_i]\rightarrow H^1[X]\rightarrow \cdots \rightarrow H^{n+1}[X].$$ It is easy to see that there is an isomorphism $\varinjlim C^\bullet[X,X\backslash K_i]\cong C^\bullet_c[X]$, and that cohomology commutes with direct limits. Since $(X,C_\bullet)$ is uniformly $n$-acyclic, we obtain the isomorphisms $\varinjlim\tilde{H}^k[X\backslash K_i]\cong \hat H^{k+1}_c[X]=\mathrm{ker}(H^{k+1}_c[X]\rightarrow H^{k+1}[X])$ for $k\leq n$, showing that $\varinjlim\tilde{H}^k[X\backslash K_i]\cong  H^{k+1}_\mathrm{coarse}(X;\mathbb{Z}_2)$ for $k\leq n$.

Since we are working over $\mathbb{Z}_2$ and each $\tilde H_k[X\backslash K_i]$ is finite dimensional, there are natural isomorphisms $\tilde H^k[X\backslash K_i]\cong \mathrm{Hom}_{\mathbb{Z}_2}(\tilde H_k[X\backslash K_i],\mathbb{Z}_2)$ and $\tilde H_k[X\backslash K_i]\cong \mathrm{Hom}_{\mathbb{Z}_2}(\tilde H^k[X\backslash K_i],\mathbb{Z}_2)$.
Thus either $\varprojlim_i \tilde{H}_k[X\backslash K_i]$ and $\varinjlim_i \tilde{H}^k[X\backslash K_i]$ are both finite dimensional, in which case there is an isomorphism $$\varinjlim_i \tilde{H}^k[X\backslash K_i]\cong \mathrm{Hom}_{\mathbb{Z}_2}(\varprojlim_i \tilde{H}_k[X\backslash K_i],\mathbb{Z}_2),$$ or both are infinite dimensional. The result now follows from Remark \ref{rem:0thcohom}, which says that  $H^0_\mathrm{coarse}(X;\mathbb{Z}_2)=0$.
\end{proof}

If $X$ is an unbounded coarse geodesic metric space, then Proposition \ref{prop:coarsegeodmetr} says that there is a $j$ such that the inclusion $X\rightarrow P_j(X)$ is a coarse isometry. We define the number of ends of $X$ to be equal to the number of ends of $P_j(X)$. Then the above argument shows that the number of ends of $X$ is equal to $\mathrm{dim}(H^1_\mathrm{coarse}(X;\mathbb{Z}_2))+1$.  In particular, $X$ is one-ended precisely when $H^1_\mathrm{coarse}(X;\mathbb{Z}_2)=0$. We think of $r$-acyclicity at infinity over $\mathbb{Z}_2$ as a higher dimensional analogue of being one-ended.

These ideas are well illustrated in the case of a hyperbolic group $G$, where the `topology at infinity' can be interpreted as the topology of the Gromov boundary $\partial G$. Bestvina and Mess show in \cite{bestvina1991boundary} that when $G$ is hyperbolic, there is an isomorphism $H^k(G,RG)\cong \check{H}^{k-1}(\partial G)$, where $\check{H}^{k-1}(\partial G)$ is the reduced \v{C}ech cohomology of the boundary. Since $H^k(G,RG)\cong H^{k}_\mathrm{coarse}(G;R)$, this gives a very concrete interpretation of coarse cohomology.

\section{Coarse Poincar\'e Duality}\label{sec:coarsepdn}

\subsection{Coarse Poincar\'e Duality Spaces}
We introduce the notion of coarse $PD_n^R$ spaces, which roughly speaking, are spaces which have the same coarse cohomological properties as $\mathbb{R}^n$.  The definition we use is found in the appendix of \cite{kapovich2005coarse}.

Let $X$ be a bounded geometry metric space. An \emph{$R$-chain complex over $X$} is a chain complex $C_\bullet$, such that each $C_i$ is a free $R$-module over $X$ and each boundary map $\partial:C_i\rightarrow C_{i-1}$ has finite displacement. We note that  $C_\bullet$ is an $R$-chain complex over $X$ for every $R$-metric complex $(X,C_\bullet)$. We extend the notion of finite displacement chain maps and homotopies from $R$-metric complexes to $R$-chain complexes over a metric space.

If $(X,C_\bullet,\Sigma_\bullet,p_\bullet)$ is a metric complex, then each $C^k_c$ is a finite type free module over $X$ with standard basis $\Omega_k$ dual to $\Sigma_k$. Each coboundary map $C^k_c\rightarrow C^{k+1}_c$ has finite displacement. Thus for any $n$,  $C^{n-\bullet}_c$ is an $R$-chain complex over $X$.

We recall that an $R$-metric complex $(X,C_\bullet)$ admits an augmentation $\varepsilon:C_0\rightarrow R$ such that $\varepsilon(\Delta)=1_R$ for each $\Delta\in \Sigma_0$. We thus define reduced homology groups $\tilde{H}_k(C_\bullet)$ of this augmented chain complex. Suppose we are given a homomorphism $\alpha:C^n_c\rightarrow R$ that is zero on all coboundaries, which we also call an \emph{augmentation of $C^{n-\bullet}_c$}. We then define an augmented cochain complex $\cdots \rightarrow C^0_c\rightarrow \dots \rightarrow C^n_c \xrightarrow{\alpha} R \rightarrow 0$ and calculate the \emph{reduced cohomology} of this cochain complex, which we denote by $\tilde H^k_c(C_\bullet)$. 

There is an ambiguity here, since $\tilde H^k_c(C_\bullet)$ may also refer to the cohomology of the cochain complex obtained by dualizing the augmented chain complex $\cdots \rightarrow C_0\xrightarrow{\varepsilon} R$ and restricting to cochains of finite support. We do not use this form of reduced cohomology in this section --- we always assume that reduced cohomology refers to the former notion, where the choice of a suitable augmentation $\alpha:C^n_c\rightarrow R$ is implicit.

\begin{defn}
A  \emph{coarse $PD^R_n$ complex} consists of a uniformly acyclic $R$-metric complex $(X,C_\bullet)$,  equipped with finite displacement chain maps 
$$C_c^{n-\bullet}\xrightarrow{P} C_{\bullet}\textrm{ and } C_{\bullet}\xrightarrow{\bar{P}} C_c^{n-\bullet}$$ 
over $\mathrm{id}_X$ and finite displacement chain homotopies  $\bar{P}\circ P\overset{\Phi}{\simeq} \mathrm{id}_{C^{n-\bullet}_c}$ and $P\circ \bar{P}\overset{\bar{\Phi}}{\simeq} \mathrm{id}_{C_\bullet}$ over $\mathrm{id}_X$. We call $P$ and $\bar P$ the \emph{duality maps}. We say that $X$ is a \emph{coarse $PD^R_n$ space} if it is the control space of some coarse $PD^R_n$ complex. 

We say that  a group $G$ \emph{acts on a coarse $PD_n^R$ complex} if it acts on the underlying metric complex $(X,C_\bullet)$, and the maps $P$, $\overline{P}$, $\Phi$ and $\overline{\Phi}$ are all $G$-equivariant.
\end{defn}

\begin{rem}\label{rem:augmentation}

If $(X,C_\bullet)$ is a coarse $PD_n^R$ complex, then the cochain complex admits an augmentation $\alpha:C^{n}_c\rightarrow R$  given by the composition $C^n_c\xrightarrow{P}C_0\xrightarrow{\varepsilon}R$. With respect to this augmentation, the maps $P$ and $\bar P$ are \emph{augmentation preserving}, i.e. $\varepsilon  (P\sigma)=\alpha(\sigma)$ for all $\sigma\in C^n_c$ and $\varepsilon(\gamma)=\alpha (\bar P \gamma)$  for all $\gamma\in C_0$.  We  note that $\alpha$ induces an isomorphism $H^n_c(C_\bullet)\cong \mathbb{Z}_2$.

\end{rem}

\begin{prop}\label{prop:pdncoarseisom}
Suppose $X$ is a coarse $PD_n^R$ space, $(Y,D_\bullet)$ is a uniformly acyclic $R$-metric complex and $f:X\rightarrow Y$ is a coarse isometry. Then $(Y,D_\bullet)$ is a coarse $PD_n^R$ complex. 
\end{prop}
\begin{proof}
There exists a coarse $PD_n^R$ complex $(X,C_\bullet)$ endowed with chain maps and homotopies $P$, $\overline{P}$, $\Phi$ and $\overline{\Phi}$ as above . Since $f$ is a coarse isometry, it has coarse inverse $g$. By Lemma \ref{lem:extendmapsunif}, we can construct finite displacement chain maps $f_\#:C_\bullet\rightarrow D_\bullet$ over $f$ and $g_\#:D_\bullet\rightarrow C_\bullet$ over $g$, as well as finite displacement chain homotopies $g_\#f_\#\simeq \mathrm{id}_{C_\bullet}$ and $f_\#g_\#\simeq \mathrm{id}_{D_\bullet}$ over $\mathrm{id}_X$ and $\mathrm{id}_Y$ respectively.  By the proof of Proposition \ref{prop:propchainmap}, $f_\#$ and $g_\#$ induce maps on cochains with compact supports. 
Lemma \ref{lem:compdispmaps} now tells us that $$D_c^{n-\bullet}\xrightarrow{f_\#\circ P\circ f^\#} D_{\bullet} \textrm{ and } D_{\bullet}\xrightarrow{g^\#\circ  \bar{P}\circ g_\#} D_c^{n-\bullet}$$ are finite displacement chain maps over $\mathrm{id}_Y$. Furthermore,  there are finite displacement chain homotopies $$f_\#\circ P\circ f^\#\circ g^\#\circ  \bar{P}\circ g_\#\simeq \mathrm{id}_{D_\bullet} \textrm{ and }g^\#\circ  \bar{P}\circ g_\#\circ f_\#\circ P\circ f^\# \simeq \mathrm{id}_{D^{n-\bullet}_c}$$ over $\mathrm{id}_Y$, which can be explicitly written down in terms of existing homotopies and chain maps.
\end{proof}

\begin{rem}
If $X$ is a coarse $PD_n^R$ space and $f:X\rightarrow Y$ is a coarse isometry, then by Propositions \ref{prop:unifacyccoarseinv} and  \ref{prop:controlspace}, there necessarily exists a uniformly acyclic $(Y,D_\bullet)$. Thus Proposition \ref{prop:pdncoarseisom} tells us that being a coarse $PD_n^R$ space is invariant under coarse isometries.
\end{rem}

\begin{defn}
A \emph{coarse $PD_n^{R}$ group} is a finitely generated group that, when equipped with a word metric with respect to a finite generating set, is a coarse $PD_n^R$ space.
\end{defn}

\subsection{Poincar\'e Duality Groups}
For the convenience of the reader, we briefly define $PD^R_n$ groups and relate them to coarse $PD_n^R$ groups. Although knowledge of $PD_n^R$ groups are not needed for our results, $PD^R_n$ groups provide a source of examples of coarse $PD_n^R$ groups.

\begin{defn}
 A group $G$ is a \emph{$PD_n^R$ group}  if: \begin{enumerate}
\item it has a finite length projective resolution of the trivial $RG$-module $R$ in which each module is finitely generated;
\item $H^i(G,RG)=0$ for $i\neq n$, and $H^n(G,RG)=R$.
\end{enumerate} 
\end{defn}

A $PD_n^\mathbb{Z}$ group $G$ is called \emph{orientable} if the action of $G$ on $H^n(G,\mathbb{Z}G)\cong \mathbb{Z}$ is trivial. For $PD_n^{\mathbb{Z}_2}$ groups, we no longer have to worry about orientation. A  $PD_n^{\mathbb{Z}}$ group is automatically a $PD_n^{R}$ group for any commutative ring   $R$ with unity, as shown in Proposition V.3.7 of \cite{dicks89groups}.

\begin{prop}[{\cite[Section 11.2]{kapovich2005coarse}}]\label{prop:pdfgpcmpx}
A group $G$ acts freely and cocompactly on a coarse $PD_n^R$ complex if and only if $G$ is a $PD_n^R$ group.
\end{prop} 

\begin{exmp}[{\cite[VIII Example 10.1]{brown1982cohomology}}]
The fundamental group of a closed aspherical $n$-manifold is a $PD_n^\mathbb{Z}$ group. 
\end{exmp}

\begin{exmp}[Page 166, Example 1 in \cite{bieri1981homological}]\label{exmp:polycylic}
If $G$ is a torsion-free polycyclic group of Hirsch length $n$, then $G$ is a $PD_n^\mathbb{Z}$ group.
\end{exmp}
Every virtually polycyclic has a finite index subgroup that is a torsion-free polycyclic group. Hence every virtually polycyclic group is a coarse $PD_n^\mathbb{Z}$ group.
The following example shows that there are coarse $PD_n^R$ groups  that are not $PD_n^R$ groups.
\begin{exmp}
The infinite cyclic group $\mathbb{Z}$ is easily seen to be a $PD_1^\mathbb{Z}$ group. It is quasi-isometric to $H:=\mathbb{Z}_2*\mathbb{Z}_2$, which is known not to be a $PD_1^\mathbb{Z}$ group since it has torsion (see \cite{brown1982cohomology}).
Proposition \ref{prop:unifacyccoarseinv} and Proposition \ref{prop:controlspace} allow us to construct a uniformly acyclic metric complex $(H,D_\bullet)$ that admits a free $H$-action. By Proposition \ref{prop:pdncoarseisom}, $(H,D_\bullet)$ is a coarse $PD_1^\mathbb{Z}$ complex and hence $H$ is a coarse $PD_1^\mathbb{Z}$ group. However, even though $H$ acts freely and cocompactly on $(H,D_\bullet)$ \emph{as a metric complex}, Proposition \ref{prop:pdfgpcmpx} tells us it cannot act freely and cocompactly on $(H,D_\bullet)$ \emph{as a $PD_1^\mathbb{Z}$ complex}. In other words, the duality maps and homotopies cannot be $H$-equivariant. 
\end{exmp}

\subsection{A Technical Lemma}
The rest of this section will be devoted to a proof of the following Lemma, which plays an important role in the proof of Theorem \ref{thm:qiinv}.

\begin{lem}\label{lem:finindex}
 Let $Y$ be a  bounded geometry metric space which is coarsely uniformly $n$-acyclic over $\mathbb{Z}_2$ for some $n\geq 0$. Let $X\subseteq Y$ be a subspace which is a  $PD_n^{\mathbb{Z}_2}$ space, and let $f:X\rightarrow Y$ be the inclusion map. Suppose $G$ is a finitely generated group acting freely on  $Y$ such that $GX=X$ and the action of $G$ restricted to $X$ is cocompact. If for $k\leq n$, the maps $f^*:H^k_\mathrm{coarse}(Y;\mathbb{Z}_2)\rightarrow H^k_\mathrm{coarse}(X;\mathbb{Z}_2)$, induced by $f$, are isomorphisms, then $Y=N_r(X)$  for some $r<\infty$.
\end{lem}
\begin{proof}
We first restrict to the case where $Y$ is 1-geodesic. Since $Y$ is coarsely uniformly $0$-acyclic, Proposition \ref{prop:coarsegeodacyc} tells us it is coarse geodesic. Hence Proposition \ref{prop:coarsegeodmetr} says that for some $s$ large enough, $P_s^1(Y)$ is connected and the inclusion $Y\rightarrow P_s^1(Y)$ is a coarse isometry. We thus remetrise $Y$ by identifying it with the $0$-skeleton of  $P_s^1(Y)$, endowed with the subspace metric. This procedure doesn't alter our hypotheses.

We use Proposition \ref{prop:controlspace} to construct  $\mathbb{Z}_2$-metric complexes $(X,C_\bullet,\Sigma_\bullet,p_\bullet)$ and $(Y,D_\bullet,\Sigma_\bullet',p_\bullet')$ that are uniformly acyclic and uniformly $n$-acyclic respectively and admit a $G$-action. Without loss of generality, we may assume that $C_\bullet$ is $(n+1)$-dimensional. By Proposition \ref{prop:pdncoarseisom}, we observe that $(X,C_\bullet)$ is a $PD_n^{\mathbb{Z}_2}$ complex with (not necessarily $G$-equivariant) duality maps $P$ and $\bar P$ and homotopies $\Phi$ and $\bar \Phi$. By Lemma \ref{lem:extendmapsunif}, there exists a chain map $f_\#:C_\bullet\rightarrow D_\bullet$ that has finite $(n+1)$-displacement over $f$. This induces a cochain map $f^\#:D^{n-\bullet}_c\rightarrow C^{n-\bullet}_c$.

We fix some $\tilde{x}\in X$ and define $K^X_i:=N^X_i(\tilde x)$ and $K^Y_i:=N^Y_i(\tilde x)$ for each $i\in \mathbb{N}$. For every $j\geq i$, we let $\iota_\#^{i,j}:C_\bullet[X\backslash K_j^X]\rightarrow C_\bullet[X\backslash K_i^X]$ and $\iota^\#_{i,j}:D^\bullet_c[Y,K_j^Y]\rightarrow D^\bullet_c[Y,K_i^Y]$ be the maps  induced by inclusion. For each $i$, there exists a short exact sequence \begin{align}\label{eqn:finindses1}
   1\rightarrow D^{n-\bullet}_c[Y,K^Y_i]\xrightarrow{q^\#_i} D^{n-\bullet}_c[Y]\xrightarrow{p^\#_i} D^{n-\bullet}_c[K^Y_i]\rightarrow 1,
\end{align} where $p^\#_i$ is the map induced by inclusion. We note that $q^\#_i\iota^\#_{i,j}=q^\#_j$ for all $i\leq j$.

We let  $\varepsilon$ and $\alpha$ be the augmentations defined in Remark \ref{rem:augmentation}.  We define an augmentation $\beta$ by $D^n_c\xrightarrow{f^\#}C^n_c\xrightarrow{\alpha} \mathbb{Z}_2$, and observe that  $\beta$ induces an isomorphism $H^n_c(D_\bullet)\cong \mathbb{Z}_2$. For each $i\in \mathbb{N}$,  $\beta \circ q^\#_i:D^{n}_c[Y,K^Y_i]\rightarrow \mathbb{Z}_2$ defines an augmentation,  allowing us to define the reduced cohomology  $\tilde  H^k_c(D_\bullet[Y,K^Y_{R}])$.

By Lemma \ref{lem:chaininverse}, there is an augmentation preserving chain map $g_\#:C_\bullet\rightarrow D^{n-\bullet}_c$ of finite displacement over $f$. By Lemma \ref{lem:pdncycleatinf}, there is a $D\in \mathbb{N}$ such that for every $R\in \mathbb{N}$,  $g_\#$ and $h_\#:=Pf^\#$ induce augmentation preserving maps 
\begin{align}g_\#^{R+D}&:  C_\bullet[X\backslash K^X_{R+D}]\rightarrow D^{n-\bullet}_c[Y,K^Y_{R}]\label{eqn:ginduce}\end{align}\begin{align}  h_\#^{R+D}&:   D^{n-\bullet}_c[Y,K^Y_{R+D}]\rightarrow  C_\bullet[X\backslash K^X_{R}].\label{eqn:hinduce}\end{align} Moreover, Lemma \ref{lem:pdncycleatinf} also tells us that the maps induced on reduced homology and cohomology by (\ref{eqn:ginduce}) and (\ref{eqn:hinduce}) give the commutative diagram shown in Figure \ref{fig:cdpdntopatinf}, whose horizontal maps are induced by inclusion.  

\begin{figure}[ht]
\caption{}
\label{fig:cdpdntopatinf}
\xymatrix@=66pt{
\dots & \ar[l]_{\iota_*^{R-2D,R}} \tilde H_{n-1}(C_\bullet[X\backslash K^X_{R}]) \ar[d]^{g_*^{R}}  & \ar[l]_{\iota_*^{R,R+2D}} \tilde H_{n-1}(C_\bullet[X\backslash K^X_{R+2D}])\ar[d]^{g_*^{R+2D}} & \dots \ar[l] \\
\dots &\ar[l]\ar[ul]_{h_*^{R-D}} \mathrm{im}(g_*^{R})  & \ar[l]_{\iota^*_{R-D,R+D}}\ar[ul]_{h_*^{R+D}} \mathrm{im}(g_*^{R+2D}) & \dots \ar[l]_{\iota^*_{R+D,R+3D}}  \ar[ul]_{h_*^{R+3D}}}
\end{figure}

As $D_\bullet$ is uniformly $0$-acyclic, there is an $R_0$ such that  for all $y\in Y$, the map $\tilde{H}_0(D_\bullet[N_1^Y(y)])\rightarrow \tilde{H}_0(D_\bullet[ N_{R_0}^Y(y)])$, induced by inclusion, is zero. By Corollary \ref{cor:findispmap},   there is some $R_1$ such that $g_\#(C_\bullet[X]_n)\subseteq D^{n-\bullet}_c[N^Y_{R_1}(X)]$.
As $H^n_c(C_\bullet)\cong \mathbb{Z}_2$, Proposition \ref{prop:topatinf}   tells us   $(\tilde H_{n-1}(C_\bullet[X\backslash K^X_{i}]))_{i=0}^\infty$ has inverse limit $\mathbb{Z}_2$. 
Hence there is an $i_0\geq 0$ such that for all $i\geq i_0$, there is a non-trivial reduced cycle  $\lambda_{i}\in  C_{n-1}[X\backslash K^X_{i+D}]$ such that for all $j\geq i$, $\iota_*^{i+D,j+D}[\lambda_{j}]=[\lambda_{i}]$.

We let $\sigma_i=g_\#^{i+D}\lambda_i\in D^1_c[Y,K_i^Y]$. It follows from the way $g_\#^{i+D}$ is defined in Lemma \ref{lem:pdncycleatinf} that $\mathrm{supp}(q_i^{\#}\sigma_i)\subseteq N^Y_{R_1}(X)\backslash K_i^Y$. 
We deduce from  Figure \ref{fig:cdpdntopatinf}  that $\iota^*_{i,j}[\sigma_j]=[\sigma_i]$ for all $j\geq i\geq i_0$. Let $i_1:=i_0+2D$. 
For $i\geq i_1$, as $$[\iota_\#^{i_0+D,i-D}h_\#^{i}g_\#^{i+D}\lambda_{i}]=[\iota_\#^{i_0+D,i+D}\lambda_{i}]=[\lambda_{i_0}]\neq 0,$$ we see that $[\sigma_{i}]=[g_\#^{i+D}\lambda_{i}]\neq 0$.

We choose $i\geq i_1$ large enough so $G K^X_i=X$, and let $j=i+2R_0+R_1+1$.  There is an $\omega\in C^0_c[Y]$ such that $\delta\omega=q^\#_j\sigma_j$. This is because  $\tilde H^1_c(D_\bullet)=0$ and, since $g_\#^{j+D}$ is augmentation preserving, $q^\#_j\sigma_j$ is a reduced cocycle.  We let $R_2:=\mathrm{diam}(\omega)$.

We claim that $\omega(\Delta)=1$ for every $\Delta\in {p_0'}^{-1}(K_i^Y)$. If $p^\#_i\omega=0$, then by the exactness of (\ref{eqn:finindses1}),   $\omega=q_i^\#(\omega_i)$ for some $\omega_i\in D^0_c[Y,K_i^Y]$. Thus $\iota_{i,j}^\#\sigma_j=\delta \omega_i$, which cannot happen as $[\iota^\#_{i,j}\sigma_j]=[\sigma_i]$ is non-trivial. Thus $p^\#_i\omega\neq 0$, so we can pick $\Delta_0\in {p_0'}^{-1}(K_i^Y)$ such  that $\omega(\Delta_0)=1$.

Let $\Delta\in {p_0'}^{-1}(K_i^Y)$. Then there exists a $1$-chain $p_0'(\Delta_0)=x_0,x_1,\dots, x_n=p_0'(\Delta)$ in $K_i^Y$. This is because $Y$ is 1-geodesic and $K_i^Y=N_i^Y(\tilde{x})$.
For $1\leq k<n$, we pick some  $\Delta_k\in (p_0')^{-1}(x_k)$ and let $\Delta_n=\Delta$. Each $\Delta_{k+1}-\Delta_k$ is a reduced $0$-cycle in $D_\bullet[N_1^Y(x_k)]$, hence there is some $\rho_k\in D_1[N^Y_{R_0}(x_k)]\subseteq D_1[K_j^Y]$ such that $\partial\rho_k=\Delta_{k+1}-\Delta_k$. Letting  $\rho:=\sum_{k=0}^{n-1} \rho_k\in C_1[K_j^Y]$, we see   $\partial\rho=\Delta-\Delta_0$, and so 
$$\omega(\Delta)-\omega(\Delta_0)=\omega(\partial\rho)=(q^\#_j\sigma_j)(\rho)=0.$$ Thus  $\omega(\Delta)=\omega(\Delta_0)=1$, proving the claim.

We define a function $Q:X\rightarrow \mathbb{R}\cup \infty$ which measures the distance from a point of $X$ to the furthest point of $Y$ that isn't closer to any other point of $X$. More precisely, for each $y\in Y$, we define $\mathrm{height}(y):=d_Y(y,X)=\inf\{d_Y(x,y)\mid x\in X\}$. For   $x\in X$, we define $$Q(x):=\mathrm{sup}_{y\in Y}\{\mathrm{height}(y)\mid \mathrm{height}(y)=d_Y(x,y)\}.$$

Suppose $\sup_{x\in X}Q(x)\leq r<\infty$ and $y\in Y$. Then there  exists an $x\in X$ such that $\mathrm{height}(y)=d_Y(x,y)$. Indeed, for any $x'\in X$, the set  $\{x\in X\mid d_Y(y,x)\leq d_Y(y,x')\}$ is non-empty and  finite. Hence there exists an $x\in X$ such that $d_Y(x,y)\leq d_Y(x'',y)$ for all $x''\in X$, so $\mathrm{height}(y)=d_Y(x,y)$. 
Since $d_Y(x,y)\leq Q(x)\leq r$, we see that $y\in N^Y_r(X)$. This holds for each $y\in Y$, so $Y=N^Y_r(X)$. As $G$ acts on $Y$ by isometries, $Q$ is $G$-equivariant. Since $GK_i^X=X$, it is sufficient to show that $Q$ is bounded on $K_i^X$.

Let $x\in K_i^X$; we claim that $Q(x)\leq R_2$. Suppose for contradiction that there is some $y\in Y$ such that $\mathrm{height}(y)=d_Y(y,x)> R_2$. We choose a 1-geodesic $x=w_0,w_1,\dots, w_n=y$, and pick $\Lambda_k\in p'^{-1}_0(w_k)$.  For each $0\leq k<n$, there exists a $\mu_k\in D_1[N_{R_0}^Y(w_k)]$ such that $\partial\mu_k=\Lambda_{k+1}-\Lambda_{k}$. 

As $x\in K_i^X\subseteq K_i^Y$, we know that $\omega(\Lambda_0)=1$. Since $d(x,y)>R_2=\mathrm{diam}(\omega)$, we see that $\omega(\Lambda_n)=0$. Therefore, $(q^\#_j\sigma_j)(\sum_{k=0}^{n-1}\mu_k)=\omega(\Lambda_n)-\omega(\Lambda_0)=1$. Hence there exists some $0\leq t<n$ such that $(q^\#_j\sigma_j)(\mu_t)=1$. Since $\mathrm{supp}(q_j^\#\sigma_j)\subseteq N_{R_1}^Y(X)\backslash K_j^Y$  and $\mu_t\in D_1[N_{R_0}^Y(w_t)]$, $w_t\notin K_{j-R_0}^Y$  and there is some $z\in X$ such that $d(z,w_t)\leq R_0+R_1$. As $w_0=x\in K_i^X$ it follows that $t=d(w_0,w_t)\geq (j-R_0)-i=R_0+R_1+1$. Hence $$d(z,y)\leq d(z,w_t)+d(w_t,w_n)\leq (R_0+R_1)+(n-t)\leq n-1<d(x,y),$$ which contradicts $\mathrm{height}(y)=d(x,y)$. Thus $Q(x)\leq R_2$ for every $x\in K_i^X$.
\end{proof}

The following two lemmas freely use notation from the preceding proof.

\begin{lem}\label{lem:chaininverse}
There is an augmentation preserving chain map $g_\#:C_{\bullet}\rightarrow D^{n-\bullet}_c$ of finite displacement over $f$, such that $Pf^\#g_\#$ is chain homotopic to the identity via a finite displacement chain homotopy $\Lambda$.
\end{lem}
\begin{proof}

Since $G$ acts cocompactly on $(X,C_\bullet,\Sigma_\bullet,p_\bullet)$,  $\Sigma_k$ has only finitely many $G$-orbits; we thus let $x^k_1,\dots, x^k_{t_k}$ be a set of representatives for the $G$-orbits of $\Sigma_k$. We define $g_\#$ on each $x^k_i$ and extend equivariantly. The finite displacement of  $g_\#$ then follows readily from $G$-equivariance. Indeed,  one can choose $D_k$ large enough so that $\mathrm{supp} (g_\#(x_i^k))\subseteq N^Y_{D_k}(f(p_k(x_i^k)))$ for each $x_i^k$. Thus for each $h\in G$ and $x_i^k$,  we see  $$\mathrm{supp} (g_\#(h x_i^k))\subseteq h N^Y_{D_k}(f(p_k( x_i^k)))=N^Y_{D_k}(f(p_k(h x_i^k))),$$ showing that $g_\#$ does indeed have finite displacement over $f$.

As the augmentation $\beta:D^n_c\rightarrow \mathbb{Z}_2$ induces an isomorphism $H^n_c(D_\bullet)\rightarrow \mathbb{Z}_2$, we may  choose an $n$-cycle $\sigma\in D^n_c$ such that $\beta(\sigma)=1$. Each $h\in G$ preserves cocycles and coboundaries, so preserves the unique non-trivial cohomology class in $H^n_c(D_\bullet)$; thus $[h\sigma]=[\sigma]$,  so $\beta(h\sigma)=\beta(\sigma)=1$.

For each $x_i^0$, we define $g_\#(x_i^0)=\sigma$ and extend equivariantly. We note that $g_\#$ is augmentation preserving, 
since for each $h\in G$, $\beta (g_\#(hx_i^0))=\beta (h g_\#(x_i^0))=\beta(h\sigma)=1=\varepsilon(hx_i^0)$. 
For each $x_i^1$, we see that $\delta g_\#(\partial x_i^1)=0$ and $\beta(g_\#\partial x_i^1)=\varepsilon(\partial x_i^1)=0$. Therefore, as $\beta$ induces an isomorphism $H^n_c(D_\bullet)\cong \mathbb{Z}_2$, there exists an $\omega_i^1\in D^{n-1}_c$ such that $\delta\omega_i^1=g_\#(\partial x_i^1)$. We define $g_\#(x_i^1):=\omega_i^1$ for each $x_i^1$ and extend equivariantly. We continue similarly for higher dimensions, using the fact $H^{k}_c(D_\bullet)=0$ for $k<n$.

We now claim that $f^\#g_\#:C_\bullet\rightarrow C^{n-\bullet}_c$ is a chain map of finite displacement over $\mathrm{id}_X$. If the claim is  true, then by Lemma \ref{lem:compdispmaps},  $Pf^\#g_\#$ also has finite displacement over $\mathrm{id}_X$.  By Lemma \ref{lem:extendmapsunif}, we thus see that $Pf^\#g_\#$ is chain homotopic to the identity chain map via a finite displacement chain homotopy $\Lambda$.

To prove the claim, we suppose $x\in \mathrm{supp}(f^\#g_\#hx_i^k)$. Then there exists a $\Delta\in \Sigma_{n-k}$ such that $p_{n-k}(\Delta)=x$ and $f^\#g_\#hx_i^k(\Delta)\neq 0$; thus $g_\#(hx_i^k)(f_\#\Delta)\neq 0$. We see that $\mathrm{supp}(g_\#(hx_i^k))\subseteq N^Y_{D_k}(f(p_k(h x_i^k)))$ and $\mathrm{supp}(f_\#\Delta)\subseteq N^Y_{R}(p_{n-k}(\Delta))$, where $f_\#$ has $n$-displacement at most $R$  over $f$. Therefore, as $f$ is the inclusion map $$d_X(x,p_k(hx_i^k))=d_Y(f(p_{n-k}(\Delta)),f(p_k(hx_i^k)))\leq R+D_k.$$
Hence $\mathrm{supp}(f^\#g_\#hx_i^k)\subseteq N^X_{R+D_k}(p_k(hx_i^k))$, so $f^\#g_\#$ does indeed have finite displacement over $\mathrm{id}_X$. 
\end{proof}

\begin{lem}\label{lem:pdncycleatinf}
Let $h_\#:=Pf^\#:D^{n-\bullet}_c\rightarrow C_\bullet$. There is a  $D\in \mathbb{N}$, such that for every $R\in \mathbb{N}$, $g_\#$ and $h_\#$  induce maps $g_\#^{R+D}$ and $h_\#^{R+D}$ as in (\ref{eqn:ginduce}) and (\ref{eqn:hinduce}). These maps then induce the commutative diagram as shown in Figure \ref{fig:cdpdntopatinf}.
\end{lem}

\begin{proof}
We use similar methods to those used in Section 6 of \cite{kapovich2005coarse}.
We pick $D_0\in \mathbb{N}$ such that $g_\#$ and  $f_\#$ have $n$-displacement at most $D_0$ over $f$, while  $P$, $\Lambda$ and the boundary maps of $C_\bullet$ and $D_\bullet$ all have  $n$-displacement at most $D_0$ over the identity; we now define $D=2(n+1)D_0$.  For each $i\in \mathbb{N}$ we define $q^\#_i$ and $p^\#_i$ as in (\ref{eqn:finindses1}), and let $r_\#^i:C_\bullet[X\backslash K_i^X]\rightarrow C_\bullet$ be the inclusion map.

Let $k\leq n$. If $\sigma\in C_k[X\backslash K^X_{R+D}]$, then  $\mathrm{supp}(r_\#^{R+D}\sigma)\subseteq X\backslash K^X_{R+D}$, so $$\mathrm{supp}(g_\#r_\#^{R+D}\sigma)\subseteq N_{D_0}^Y(X\backslash K_{R+D}^X)\subseteq  Y\backslash K^Y_{R+D-D_0}\subseteq Y\backslash K^Y_{R}.$$  Therefore $p_R^\#g_\#r_\#^{R+D}\sigma=0$, so by the exactness of (\ref{eqn:finindses1}) there exists an $\omega_\sigma\in D^{n-k}_c[Y,K_R^Y]$ such that $q^\#_R\omega_\sigma=g_\#r_\#^{R+D}\sigma$. 
As $q^\#_R$ is injective, we can define an augmentation preserving chain map $g_\#^{R+D}:C_\bullet[X\backslash K_{R+D}^X]_n\rightarrow D^{n-\bullet}_c[Y,K_R^Y]$ such that $q^\#_Rg^{R+D}_\#=g_\#r_\#^{R+D}$.

Similarly, we define $h_{\#}^{R+D}:D^{n-\bullet}_c[Y,K_{R+D}^Y]_n\rightarrow C_\bullet[X\backslash K_{R}^X]$ so that $r_\#^{R}h_{\#}^{R+D}=Pf^\#q^\#_{R+D}$. Indeed, suppose $\rho\in D^k_c[Y,K^Y_{R+D}]$ for $k\leq n$. If $\Delta\in (p'_k)^{-1}(K^Y_{R+D-nD_0})$, then by Lemma \ref{lem:dispmetric}, $\Delta\in D_k[K^Y_{R+D-nD_0+kD_0}]\subseteq D_k[K^Y_{R+D}]$, so that $q_{R+D}^\#(\rho)(\Delta)=0$. Therefore, $\mathrm{supp} (q_{R+D}^\#\rho)\subseteq Y\backslash K^Y_{R+D-nD_0}$ and so $\mathrm{supp}(Pf^\#q^\#_{R+D}\rho)\subseteq X\backslash K^X_{R+D-(n+2)D_0}=X\backslash K^X_{R+nD_0}$. By another application of Lemma \ref{lem:dispmetric}, we see that $Pf^\#q^\#_{R+D}\rho=r_\#^{R}\gamma_\rho$ for some $\gamma_\rho\in C_\bullet[X\backslash K_R^X]$; this allows us to define $h_{\#}^{R+D}$. 

By a similar argument, $\Lambda$ induces  a map $\Lambda^{R+2D}:C_\bullet[X\backslash K_{R+2D}]_n\rightarrow C_{\bullet+1}[X\backslash K_{R}]$ such that $\Lambda r_\#^{R+2D}=r_\#^R\Lambda^{R+2D}$; 
it follows that $h_{\#}^{R+D}g_\#^{R+2D}-\iota^{R,R+2D}_\#=\partial\Lambda^{R+2D}+\Lambda^{R+2D}\partial$. This proves the existence of the commutative diagram shown in Figure \ref{fig:cdpdntopatinf}.
\end{proof}

\section{Coarse Separation}\label{sec:coarsecompcomp}

\subsection{Coarse Complementary Components}
We introduce the notion of coarse complementary components, generalising the notion of an almost invariant set (see Proposition \ref{prop:coarsecompinvecspace}). It should be remarked that a coarse complementary component $C$ is not necessarily coarsely connected, i.e. $P_r(C)$ may not be connected for any $r$. Lemmas \ref{lem:coarsecompvscomp} and \ref{lem:compcompcoarseinv}  provide some motivation for the term `coarse complementary components'.

\begin{defn}\label{def:ccc}
Let $X$ be a $1$-geodesic metric space,  $C\subseteq X$ and $r\geq 0$. We define the \emph{coarse $r$-boundary} of $C$ to be $$\partial_r C:=\{x\in X\backslash C\mid d(x,C)\leq r \}.$$
If $W\subseteq X$, $r\geq 1$ and $A\geq 0$, we say that  $C\subseteq X$ is an \emph{$(r,A)$-coarse complementary component} of $W$ if $$\partial_r (C\backslash N_A(W)) \subseteq N_A(W).$$ 
We say that $C$ is a \emph{coarse complementary component} of $W$ if it is an \emph{$(r,A)$-coarse complementary component} of $W$ for some $r\geq1$ and $A\geq 0$.
\end{defn}

\begin{rem}\label{rem:bdrycoarsecompreal}
If $\partial_r(C)\subseteq N_A(W)$, then $C$ is an $(r,A)$-coarse complementary component of $W$. Conversely, if $C$ is an $(r,A)$-coarse complementary component of $W,$ then $\partial_r(C)\subseteq N_{A+r}(W)$.
\end{rem}

\begin{lem}\label{lem:coarsecompvscomp}
Let $X$ be a $1$-geodesic metric space and $W\subseteq X$. Then $C\subseteq X$ is an  $(r,A)$-coarse complementary component of $W$ if and only if  $C\backslash N_A(W)$ is the vertex set of a union of components of $P_r(X)\backslash P_r(N_{A}(W))$.
\end{lem}
\begin{proof}
Assume $C$ is an $(r,A)$-coarse complementary component of $W$. For $x\in C$ and $y\in X$, suppose $x$ and $y$ lie in the same component of $P_r(X)\backslash P_r(N_{A}(W))$. Then there exists an $r$-chain $x=x_0,x_1,\dots ,x_n=y$ in $X\backslash N_A(W)$. If some  $x_i\in C$ and $x_{i+1}\not \in C$, then $x_{i+1}\in \partial_r(C\backslash N_A(W)) \subseteq N_A(W)$ which contradicts $x_{i+1}\in X\backslash N_A(W)$. Therefore $y\in C$. 

Now suppose $C\backslash N_A(W)$ is the vertex set of a union of components of $P_r(X)\backslash P_r(N_A(W))$. If $x\in \partial_r (C\backslash N_A(W))$, then there is some $y\in C	\backslash N_A(W)$  such that $d(x,y)\leq r$. If $x\not \in N_A(W)$, then $x$ and $y$ lie in the same component of $P_r(X)\backslash P_r(N_A(W))$,  so $x\in C\backslash N_A(W)$, which is a contradiction. Therefore $x\in N_A(W)$, showing that $\partial_r(C\backslash N_A(W))\subseteq N_A(W)$ and hence $C$ is an $(r,A)$-coarse complementary component of $W$.
\end{proof}

\begin{cor}
The complement, union, intersection or symmetric difference of $(r,A)$-coarse complementary components of $W$ is itself an $(r,A)$-coarse complementary component of $W$.
\end{cor}

\begin{lem}\label{lem:compcompcoarseinv}
Let $X$ and $Y$ be 1-geodesic metric spaces with $W\subseteq X$.
Suppose $C\subseteq X$ is an $(r,A)$-coarse complementary component of $W$ and $f:X\rightarrow Y$ is an $(\eta,\phi,B)$-coarse isometry. Then for each $s\geq 1$, there exists an $A'=A'(A,\eta,\phi,B,s)$ such that $N_B(f(W))$ is an $(s,A')$-coarse complementary component of $f(W)$.
\end{lem}
\begin{proof}
Let $Z:=f(W)$ and $Q:=N_B(f(C))$. Setting $A':=\phi(A+ \tilde{\eta}(s+2B))+B$,  we claim that $Q$ is an $(s,A')$ coarse complementary component of $Z$, where $\tilde{\eta}$ is as defined in Remark \ref{rem:propinv}.	We need to show $\partial_s(Q\backslash N_{A'}(Z))\subseteq N_{A'}(Z)$. 

Let $y\in \partial_s(Q\backslash N_{A'}(Z))$; there exists an $x\in X$ such that $d_Y(f(x),y)\leq B$. If $x\in C$, then $y\in Q$ so $y\in N_{A'}(Z)$ and we are done. If  $x\notin C$, then there exists a $y'\in Q\backslash N_{A'}(Z)$ such that $d_Y(y,y')\leq s$. We thus pick $x'\in C$ such that $d_Y(f(x'),y')\leq B$. Since $A'\geq \phi(A)+B$ and $y'\notin N_{A'}(Z)$, it follows that $x'\in C\backslash N_A(W)$. As $d_Y(y,y')\leq s$, we see that $d_X(x,x')\leq \tilde{\eta}(s+2B)$.

We now observe that $x'$ and $x$ can be joined by a 1-geodesic $x'=x_0,\dots, x_n=x$ such that $n=d_X(x,x')\leq \tilde{\eta}(s+2B)$. We pick the minimal $t$ such that $x_t\notin C\backslash N_A(W)$ and  note that 
$$d_X(x,x_t)= n-t\leq \tilde{\eta}(s+2B).$$ 
By the minimality of $t$,  $x_t\in \partial_r(C\backslash N_A(W))\subseteq N_A(W)$, so there exists a $w\in W$ such that $d_X(w,x_t)\leq A$. Therefore $$d_Y(f(w),y)\leq d_Y(f(w),f(x))+B \leq \phi(A+\tilde{\eta}(s+2B))+B=A',$$ showing that $y\in N_{A'}(Z)$.\qedhere
\end{proof}

We say that $C$ is an \emph{irreducible $(r,A)$-coarse complementary component} whenever  $C\backslash N_A(W)$ is the vertex set of a single component of  $P_r(X)\backslash P_r(N_{A}(W))$.  However, unlike (general) coarse complementary components, irreducible coarse complementary components are not preserved under coarse isometries. 

\begin{cor}\label{cor:compcompparamindep}
For any finite collection $\{C_i\}_{i\in I}$ of coarse complementary components of $W$ and any $r\geq 1$, there exists an $A\geq 0$ such that every $C_i$ is an $(r,A)$-coarse complementary component of $W$.
\end{cor}
\begin{proof}
If $C$ is an $(r,A)$-coarse complementary component of $W$, then it is an $(r,A')$-coarse complementary component of $W$ for any $A'\geq A$.
Applying  Lemma \ref{lem:compcompcoarseinv} with $f=\mathrm{id}_X$, we see that every $C_i$ is an $(r,A_i)$ coarse complementary component for some $A_i$. Letting $A:=\max_{i\in I}(A_i)$, we see that each $C_i$ is an $(r,A)$-coarse complementary component of $W$.
\end{proof}

Since every coarse complementary component is a $(1,A)$-coarse complementary component for some $A\geq 0$, one might think that the `$r$' parameter in Definition \ref{def:ccc} is redundant. The following Lemma, which plays a crucial role in the proofs of Propositions \ref{prop:compcoarseacycl} and \ref{prop:mvseq}, motivates the need for two parameters in Definition \ref{def:ccc}.

\begin{lem}\label{lem:coarsecompsimplex}
Let $C\subseteq X$ be an $(r,A)$-coarse complementary component of $W$. Then every simplex of $P_r(X)$ lies in either $P_r(N_A(W)\cup C)$ or $P_r(N_A(W)\cup (X\backslash C))$.
\end{lem}
\begin{proof}
Let $\Delta=\{x_0,\dots, x_n\}$ be a simplex of $P_r(X)$. If each $x_i\in N_A(W)$, then $\Delta$ is a simplex of both $P_r(N_A(W)\cup C)$ and $P_r(N_A(W)\cup (X\backslash C))$. We therefore assume  $x_i\notin N_A(W)$ for some $x_i\in \Delta$. Interchanging $C$ and $X\backslash C$ if necessary, we may assume $x_i\in C\backslash N_A(W)$. For each $x_j\in \Delta$, either $x_j\in C\backslash N_A(W)$ or $x_j\in \partial_r(C\backslash N_A(W))\subseteq N_A(W)$. Thus $\Delta$ is a simplex of $P_r(N_A(W)\cup C)$.
\end{proof}

We now show that if $C$ is a coarse complementary component of $W$ and both $X$ and $W$ are coarsely uniformly $(n-1)$-acyclic, then $W\cup C$ is also coarsely uniformly $(n-1)$-acyclic. This  allows us  to define the coarse cohomology of $W\cup C$.

\begin{prop}\label{prop:compcoarseacycl}
Let $X$ be a 1-geodesic metric space with $W\subseteq X$, such that both $X$ and $W$ are $(\lambda,\mu)$-coarsely uniformly $(n-1)$-acyclic over $R$ for some $n>0$. If $C$ is a $(1,A)$-coarse complementary component of $W$, then $W\cup C$ is  $(\lambda',\mu')$-coarsely uniformly $(n-1)$-acyclic over $R$, where $\lambda'$ and $\mu'$ depend only on $A$, $\lambda$ and $\mu$.
\end{prop}
To prove this, we use the following lemma.

\begin{lem}\label{lem:neighbhdseq}
Let $X$ be a bounded geometry metric space with $A\subseteq B \subseteq X$. Suppose there exists some $t\geq 0$ such that $B\subseteq N_t(A)$. If $f_\#:C_\bullet(P_i(A);R)\rightarrow C_\bullet(P_i(B);R)$ is the inclusion of chain complexes, then there is a chain map $$p_\#:C_\bullet(P_i(B);R)\rightarrow C_\bullet(P_{i+2t}(A);R)$$  such that $p_\#f_\#$ is the inclusion $C_\bullet(P_i(A);R)\rightarrow C_\bullet(P_{i+2t}(A);R)$ and $$p_\#(C_\bullet(P_i(N^B_r(x));R))\subseteq C_\bullet(P_{i+2t}(N_{r+t}^A(x);R)$$ for every $x\in X$.
\end{lem}
\begin{proof}
We choose a closest point projection $p:B\rightarrow A$, noting that for each $b\in B$, $d(b,p(b))\leq t$. Then $p$ induces the chain map $p_\#:C_\bullet(P_i(B);R)\rightarrow C_\bullet(P_{i+2t}(A);R)$ given by $[x_0,\dots x_n]\mapsto [p(x_0),\dots, p(x_n)]$ on each oriented simplex $[x_0,\dots x_n]$.   Since $p|_{A}=\mathrm{id}_A$,  $p_\#f_\#$ is just the inclusion of chain complexes.
\end{proof}

\begin{proof}[Proof of Proposition \ref{prop:compcoarseacycl}]	
Without loss of generality, we may assume  $\lambda(i)\geq 1$ for every $i$.
For any $k<n$, $i\geq 0$, $r\geq 0$ and $x\in X$, let $\sigma \in C_k(P_i(N_r^{C\cup W}(x));R)$ be a reduced cycle.    There exists a $(k+1)$-chain $\omega \in C_{k+1}(P_{\lambda(i)}(N_{\mu(i,r)}^X(x));R)$ such that $\partial\omega=\sigma$.

By Corollary \ref{cor:compcompparamindep}, there is an $A'=A'(\lambda( i),A)\geq 0$ such that $C$ is a $(\lambda(i),A')$-coarse complementary component of $W$. By Lemma \ref{lem:coarsecompsimplex},  each simplex of $P_{\lambda(i)}(X)$ lies in either $P_{\lambda(i)}(N^X_{A'}(W)\cup C)$ or $P_{\lambda(i)}(N^X_{A'}(W)\cup (X\backslash C))$. Therefore, $\omega=\omega_1+\omega_2$, where $\omega_1\in C_{k+1}( P_{\lambda(i)}(N^{N^X_{A'}(W)\cup C}_{\mu(i,r)}(x));R)$ and $\omega_2\in C_{k+1}( P_{\lambda(i)}(N^{N^X_{A'}(W)\cup (X\backslash C)}_{\mu(i,r)}(x));R)$.
As $\partial\omega_2=\partial\omega-\partial\omega_1=\sigma -\partial\omega_1$, we see that $\partial\omega_2 \in C_k(P_{\lambda(i)}(N^{N^X_{A'}(W)}_{\mu(i,r)}(x));R)$ is a reduced $k$-cycle. The inclusion $W\rightarrow N^X_{A'}(W)$ is an $(\mathrm{id}_{\mathbb{R}_{\geq 0}},\mathrm{id}_{\mathbb{R}_{\geq 0}},A')$-coarse isometry. Consequently, Proposition \ref{prop:unifacyccoarseinv} tells us that $N^X_{A'}(W)$ is $(\lambda',\mu')$ coarsely uniformly $(n-1)$-acyclic, where $\lambda'$ and $\mu'$ depend only on $A$, $\lambda$ and $\mu$. Hence there is a $\gamma\in C_{k+1}(P_{i'}(N^{N^X_{A'}(W)}_{r'}(x));R)$ such that $\partial\gamma=\partial\omega_2$, where $i'=\lambda'(\lambda(i))$ and $r'=\mu'(\lambda(i),\mu(i,r))$. 

Let $f_\#$ be the inclusion $C_\bullet(P_{i'}(C\cup W);R)\rightarrow C_\bullet(P_{i'}(C\cup N^X_{A'}(W));R)$.  As $C\cup N^X_{A'}(W)\subseteq N_{A'}^X(C\cup W)$, Lemma \ref{lem:neighbhdseq} tells us there is a chain map $p_\#:C_\bullet(P_{i'}(C\cup N^X_{A'}(W));R)\rightarrow C_\bullet(P_{i'+2{A'}}(C\cup W);R) $  such that $p_\#f_\#$ is  the inclusion. We see that $$p_\#(\omega_1+\gamma)\in C_{k+1}(P_{i'+2{A'}}(N^{C\cup W}_{r'+{A'}}(x));R)$$ and  $\partial p_\#(\omega_1+\gamma)=p_\#f_\#(\sigma)=\sigma$. Therefore,  $C\cup W$ is coarsely uniformly $(n-1)$-acyclic over $R$, and the uniform acyclicity functions depend only on ${A}$, $\lambda$ and $\mu$.
\end{proof}

\subsection{Coarse \texorpdfstring{$n$}{n}-separation}
\begin{defn}
A coarse complementary component $C$ of $W$ is \emph{shallow} if $C\subseteq N_R(W)$ for some $R\geq 0$. A coarse complementary component that is not shallow is called \emph{deep}. A collection $\{C_i\}_{i\in I}$ of coarse complementary components of $W$ is said to be \emph{coarse disjoint} if for every $i\neq i'\in I$, $C_i\cap C_{i'}$ is shallow. Similarly, we say that a component of $P_r(X)\backslash P_r(N_A(W))$ is \emph{deep} (resp. \emph{shallow}) if its vertex set is deep (resp. shallow). 

We say that $W$ \emph{coarsely $n$-separates} $X$ if $X$ has at least $n$ deep, coarse disjoint, coarse complementary components of $W.$ We say that $W$ \emph{coarsely separates} $X$ if $W$ coarsely 2-separates $X$. 
\end{defn}

\begin{lem}\label{lem:nbhdcomp}
Let $X$ be a 1-geodesic metric space and suppose $C$ is an $(r,A)$-coarse complementary component of $W$ for some $r\geq 1$ and $A\geq 0$. Then for all $R\geq 0$, $$N_R(C)\subseteq C \cup N_{A+R}(W).$$
\end{lem}
\begin{proof}
Suppose $x\in N_{R}(C)\backslash C$. There exists an $x_0\in C$ with $d(x_0,x)\leq R$. Hence there is a 1-geodesic $x_0,\dots, x_n=x$ with $n=d(x_0,x)\leq R$.
We pick the minimal $t$ such that $x_t\notin C\backslash N_A(W)$ and claim $x_t\in N_A(W)$.
If $t=0$, then because $x_0\in C$,  $x_0\in N_A(W)$. If $t>0$, then  $x_t\in \partial_r(C\backslash N_A(W))\subseteq N_A(W)$. 
As $d(x,x_t)=n-t\leq R$, we see that $x\in N_R(N_A(W))\subseteq N_{A+ R}(W)$.
\end{proof}

\begin{prop}
Suppose $X$ and $Y$ are $1$-geodesic and $f:X\rightarrow Y$ is an $(\eta,\phi,B)$-coarse isometry. If $W\subseteq X$ coarsely $n$-separates $X$, then $f(W)$ coarsely $n$-separates $Y$.
\end{prop}
\begin{proof}
Since $W$ coarsely $n$-separates $X$,  there exist $n$ deep, coarse disjoint, coarse complementary components of $W$, which we denote $C_1,\dots, C_n$.
By Lemma \ref{lem:compcompcoarseinv}, we need only show that  $N_B(f(C_1)),\dots, N_B(f(C_n))$ are  deep and coarse disjoint.  If $N^Y_B(f(C_i))\subseteq N^Y_R(f(W))$ for some $R\geq 0$, then $C_i\subseteq N^X_{\tilde{\eta}(R)}(W)$. Consequently, as each $C_i$ is deep,  so is each $N_B(f(C_i))$.

For each $1\leq i<j\leq n$, we choose $R\geq 0$ such that  $C_i\cap C_j\subseteq N_R(W)$. There is some $r\geq 1$ and $A\geq 0$ such that $C_j$ is an $(r,A)$-coarse complementary component of $W$. Suppose $y\in N_B(f(C_i))\cap N_B(f(C_j))$. We pick $c_i\in C_i$ and $c_j\in C_j$ such that $d_Y(f(c_i),y),d_Y(f(c_j),y)\leq B$; hence $d_X(c_i,c_j)\leq \tilde{\eta}(2B)$. By Lemma \ref{lem:nbhdcomp} $$c_i\in C_i\cap N_{\tilde{\eta}(2B)}(C_j)\subseteq C_i\cap (C_j\cup N_{A+\tilde{\eta}(2B)}(W))\subseteq N_{A'}(W),$$ where $A':=\max(R,A+\tilde{\eta}(2B))$. Hence $y\in N_{\phi(A')+B}(f(W))$, so  $N_B(f(C_i))$ and $N_B(f(C_j))$ are coarse disjoint. 
\end{proof}

There is another notion of coarse separation, defined in terms of actual complementary components rather than coarse complementary components: if $\Gamma$ is the Cayley graph of a group $G$ with respect to some finite generating set, we say that $W$ coarsely separates $G$ if for some $R\geq 0$, $\Gamma\backslash N_R(W)$ has at least two deep components, i.e. two components not contained in $N_S(W)$ for any $S\geq 0$. This is the definition used in \cite{kapovich2005coarse} and \cite{papasoglu2007group}.

The two notions of coarse separation are not necessarily equivalent. To see why, we observe that if $C$ is an $(r,A)$-coarse complementary component of $W,$ then $C$ is a union $\cup _{i\in I}C_i$ of irreducible $(r,A)$-coarse complementary components. It is possible for $C$ to be deep and for each $C_i$ to be shallow. The following lemma rules out this behaviour when $W$ is a subgroup of a finitely generated group. More generally, whenever $W$ satisfies the \emph{shallow condition} as defined in \cite{vavrichek2012coarse}, then the two notions of coarse separation are equivalent.

\begin{lem}[{\cite[Lemma 7.1]{vavrichek2012coarse}}]\label{lem:sbgpshallow}
Let $G$ be a finitely generated group and let $H\leq G$ be a subgroup.  Then for all $r\geq 1$ and $A\geq0$, there exists an $R\geq A$ such that every shallow  $(r,A)$-coarse complementary component of $H$ is contained in $N_R(H)$.

\end{lem}
\begin{proof}
Using Lemma \ref{lem:coarsecompvscomp}, it is enough to show there is an  $R\geq 0$ such that every shallow component of $P_r(G)\backslash P_r(N_A(H))$ is contained in $P_r(N_R(H))$.
We observe that $P_r(G)\backslash P_r(N_A(H))$ has finitely many components that intersect the ball $N_{A+r}(e)$. Hence there is an $R\geq 0$ such that every shallow component of $P_r(G)\backslash P_r(N_A(H))$ that intersects $N_{A+r}(e)$ is contained in $P_r(N_R(H))$. As $P_r(G)$ is connected, an arbitrary shallow component $C$ of $P_r(G)\backslash P_r(N_A(H))$ intersects the ball $N_{A+r}(h)$ for some $h\in H$. Since $h^{-1}C$ intersects $N_{A+r}(e)$,  it is contained in $P_r(N_R(H))$; therefore $C\subseteq h P_r(N_R(H))=P_r(N_R(H))$.
\end{proof}

\subsection{Relative Ends of Groups}\label{sec:relends}

Suppose $G$ is a finitely generated group and $\mathcal{P}(G)$ is the powerset of $G$. Then $\mathcal{P}(G)$ is a vector space over $\mathbb{Z}_2$ in which the addition operation is the symmetric difference of two subsets. If $H$ is a subgroup of $G$, let $\mathcal{F}_H(G)$ be the subspace of all $H$-finite subsets of $G$.

\begin{defn}
We say that $C\subseteq G$ is \emph{$H$-almost invariant} (or $H$-a.i.) if for all $g\in G$, $C+Cg$ is $H$-finite. An $H$-a.i. subset $C$  is \emph{proper} if neither $C$ nor $G\backslash C$ is $H$-finite.
\end{defn}
In other words, $C$ is $H$-almost invariant precisely when it projects to an element in the fixed set $(\mathcal{P}(G)/\mathcal{F}_H(G))^G$. The following proposition gives a geometric interpretation of $H$-a.i. sets in terms of coarse complementary components of $W$.

\begin{prop}\label{prop:coarsecompinvecspace}
Let $G$ be a group equipped with the word metric with respect a finite generating set $S$, and suppose $H\leq G$ is a subgroup. A subset $C\subseteq G$ is $H$-a.i. if and only if it is a  coarse complementary component of $H$.
\end{prop}
\begin{proof}
Say $C\subseteq G$ is $H$-almost invariant. By Lemma \ref{lem:hfinitenbhd}, we see that for every $g\in G$ there exists an $r_g\geq 0$ such that $Cg+C\subseteq N_{r_g}(H)$. We choose $A$ large enough so that for all $s \in S^{\pm 1}$,  $Cs+C\subseteq N_A(H)$.
We claim that $C$ is a $(1,A)$-coarse complementary component of $H$. Let  $g\in \partial_1(C\backslash N_A(H))$. If $g\in C$, then as $g\notin C\backslash N_A(H)$,  $g\in N_A(H)$ and we are done. If $g\notin C$, then there exists an $h\in C$ such that $g=hs$ for some $s\in  S^{\pm 1}$, and so $g\in Cs+C\subseteq N_A(H)$.

Conversely, suppose $C$ is a coarse complementary component of $H$. By Corollary \ref{cor:compcompparamindep}, there exists an $A\geq 0$ such that $C$ is a $(1,A)$-coarse complementary component of $H$. For any $g\in G$, we write $g^{-1}=s_1\dots s_n$, where for each $i$, either $s_i$ or  its inverse is contained in $S$. If $b\in C+Cg$, then exactly one of $b$ or $bs_1\dots s_n$ lies in $C$. Therefore, there exists a $1\leq t\leq n$ such that $b':=bs_1,\dots s_t\in \partial_1(C)$. By Remark \ref{rem:bdrycoarsecompreal},  $b'\in N_{A+1}(H)$ and so  $b\in N_{A+n+1}(H)$. Thus $C+Cg$ is $H$-finite.
\end{proof}

In \cite{kropholler1989relative}, Kropholler and Roller defined the number  of relative ends of a pair of groups to be $$\tilde{e}(G,H):=\mathrm{dim}_{\mathbb{Z}_2}(\mathcal{P}(G)/\mathcal{F}_H(G))^G.$$  This can be characterised geometrically as follows:
\begin{prop}[{\cite[Lemma 7.5]{vavrichek2012coarse}}]\label{prop:relendscoarsep}
Let $G$ be a finitely generated group and let $H\leq G$. Then $H$ coarsely $n$-separates $G$ if and only if $\tilde{e}(G,H)\geq n$.
\end{prop}
\begin{proof}
If  $H$ coarsely $n$-separates $G$, then there exist $n$ deep, coarse disjoint, coarse complementary  components of $H$, which we label $C_1, \dots, C_n$. By Proposition \ref{prop:coarsecompinvecspace}, these correspond to elements of $(\mathcal{P}(G)/\mathcal{F}_H(G))^G$. Since they are coarse disjoint, they are linearly independent, so $\tilde e(G,H)\geq n$.

Now suppose $\tilde{e}(G,H)\geq n$. Then there exist $n$ deep,  coarse complementary components  $C_1,\dots, C_n$ of $H$ which represent linearly independent elements of $(\mathcal{P}(G)/\mathcal{F}_H(G))^G$. We explain how to modify these coarse complementary components pairwise so that they are coarse disjoint.

We set $A:=C_1\backslash C_{2}$,  $B:=C_{1}\cap C_{2}$ and $C=C_{2}\backslash C_1$, observing that $C_1=A+B$ and $C_{2}=B+C$. By the linear independence of $C_1,\dots, C_n$, at least two of $A$, $B$ and $C$, say  $A$ and $B$,  aren't contained in the span of $C_3,\dots, C_n$. Then $A,B,C_3,\dots, C_n$ are linearly independent and $A$ and $B$ are coarse disjoint. Applying this procedure to all pairs, we obtain $n$ deep, coarse disjoint, coarse complementary components of $H$, showing that $H$ coarsely $n$-separates $G$.
\end{proof}

There is another notion of relative ends of groups due to Haughton \cite{houghton1974ends} and Scott \cite{scott1977ends}. They define $$e(G,H):=\mathrm{dim}_{\mathbb{Z}_2}(\mathcal{P}(H\backslash G)/\mathcal{F}(H\backslash G))^G,$$ where $\mathcal{P}(H\backslash G)$ is the powerset of the coset space $H\backslash G$, and $\mathcal{F}(H\backslash G)$ consists of finite subsets of $H\backslash G$. 
The following is an analogue of Proposition \ref{prop:relendscoarsep} and can be proved in the same way.
\begin{prop}[{\cite[Lemma 1.6]{scott1977ends}}]\label{prop:scotthaughtoncoarsecomp}
Let $G$ be finitely generated and $H\leq G$. Then $e(G,H)\geq n$ if and only if there exist $n$ deep, coarse disjoint, coarse complementary components of $H$, denoted $C_1,\dots, C_n$, such that $HC_i=C_i$ for each $C_i$.
\end{prop}

It follows from Proposition \ref{prop:relendscoarsep} and \ref{prop:scotthaughtoncoarsecomp} that $e(G,H)\leq \tilde{e}(G,H)$ for all $H\leq G$. 
We say that $H\leq G$ is a  \emph{codimension one subgroup} of $G$ if $e(G,H)>1$. If $G$ splits over $H$, then $\tilde e(G,H)\geq e(G,H)>1$.

If we know that $G$ is coarsely separated by a subgroup $H$,  the following lemma allows us to construct a codimension one subgroup $H'\leq H$ of $G$.

\begin{lem}\label{lem:stabcoarsesep}
Let $G$ be a finitely generated group and $H$ be a subgroup. Let $C$ be a deep, irreducible $(1,A)$-coarse complementary component of $H$ and suppose $G\backslash C$ is deep. If $H':=\mathrm{Stab}_H(C\backslash N_A(H))$, then $e(G,H')>1$.
\end{lem}
\begin{proof}
We use an argument from the proof of \cite[Lemma 7.6]{vavrichek2012coarse}.
We first claim that $H'$ and $\partial_1 (C\backslash N_A(H))$ are finite Hausdorff distance from one another. Since $H'$ stabilizes $C\backslash N_A(H)$, it stabilizes $\partial_1 (C\backslash N_A(H))$. We let $r_1:=d(e,\partial_1 (C\backslash N_A(H)))$. Then for each $h\in H'$, we see that $$d(h,\partial_1 (C\backslash N_A(H)))=d(e,h^{-1}\partial_1 (C\backslash N_A(H)))=d(e,\partial_1 (C\backslash N_A(H)))=r_1,$$ so $H'\subseteq N_{r_1}(\partial_1 (C\backslash N_A(H)))$.

As $C$ is an irreducible $(1,A)$-coarse complementary component of $H$, $C\backslash N_A(H)$ is the vertex set of a component $\tilde{C}$ of $P_1(G)\backslash P_1(N_A(H))$. As $G$ has bounded geometry, there are only finitely many components $\{h \tilde{C}\}_{h\in H}$  of $P_1(G)\backslash P_1(N_A(H))$ that intersect $N_{A+1}(e)$.  We label these components $h_1\tilde{C}, \dots, h_m\tilde{C}$ and set  $r_2:=A+\max_{1\leq k\leq m} \mathrm{length}_S(h_k)$.

Suppose $x\in \partial_1(C\backslash N_A(H))\subseteq N_A(H)$; we choose $h\in H$ such that $d(h,x)\leq A$. There exists a $y\in C\backslash N_A(H)$ such that $d(x,y)\leq 1$. As $d(h,y)\leq A+1$, we see that $h^{-1}\tilde{C}$ intersects $N_{A+1}(e)$, hence $h^{-1}\tilde{C}=h_r\tilde{C}$ for some $r$. Thus $\tilde{C}=hh_r\tilde{C}$, so $hh_r\in H'$.  Hence $x\in N_{r_2}(H')$ and so $\partial_1(C\backslash N_A(H))\subseteq N_{r_2}(H')$.

Letting $C':=C\backslash N_A(H)$, Remark \ref{rem:bdrycoarsecompreal} tells us that $C'$ is a $(1,r_2)$-coarse complementary component of $H'$. Since $C$ and $G\backslash C$ are deep as coarse complementary components of $H$, $C'$ and $G\backslash C'$ are also deep as coarse complementary components of $H$. As $H'\leq H$, $C'$ and $G\backslash C'$ are deep as coarse complementary components of $H'$. Since $H'C'=C'$ and $H'(G\backslash C')=(G\backslash C')$, Proposition \ref{prop:scotthaughtoncoarsecomp} tells us that $e(G,H')>1$.
\end{proof}

\begin{cor}\label{cor:stabfiniteindex}
Let $G$, $H$, $C$ and $H'$ be as above, and suppose that $\tilde{e}(G,K)=1$ for all infinite index subgroups $K\leq H$. Then $[H:H']$ is finite.
\end{cor}

\subsection{A Coarse Mayer--Vietoris Sequence}\label{sec:mvseq}
Throughout this section, we fix a commutative ring $R$ with unity and assume all cohomology  is taken with coefficients in $R$.

 In \cite{higson1993coarse}, Higson, Roe and Yu describe a coarse Mayer--Vietoris sequence for coarse cohomology. We suppose $C_1$ and $C_2$ are  disjoint, coarse complementary components of $W$ such that $G=C_1\cup C_2$ and $W=C_1\cap C_2$. Lemma \ref{lem:nbhdcomp} tells us that for all $B\geq 0$, there is a $B'\geq 0$ such that $N_B(C_1)\cap N_B(C_2)\subseteq N_{B'}(W)$. Thus $C_1$ and $C_2$ are \emph{coarsely excisive} in the sense of \cite{higson1993coarse}. Using Roe's coarse cohomology  as described in Appendix \ref{app:roecoarse}, it follows from Theorem 1 of \cite{higson1993coarse} that  there is an exact sequence {
\begin{align*}
0\rightarrow H^0_{\mathrm{coarse}}(X)\xrightarrow{q^*} H^0_{\mathrm{coarse}}(C_1\cup W)\oplus H^0_{\mathrm{coarse}}(C_2\cup W) \xrightarrow{p^*} H^0_{\mathrm{coarse}}(W)\xrightarrow{\delta} H^1_{\mathrm{coarse}}(X)\xrightarrow{q^*}\cdots
\\\cdots \xrightarrow{q^*} H^{n-1}_{\mathrm{coarse}}(C_1\cup W)\oplus H^{n-1}_{\mathrm{coarse}}(C_2\cup W) \xrightarrow{p^*}H^{n-1}_{\mathrm{coarse}}(W)\xrightarrow{\delta} H^n_{\mathrm{coarse}}(X).
\end{align*}
}We derive this Mayer--Vietoris sequence using metric complexes. This preserves quantitative information that cannot be deduced using \cite{higson1993coarse}.

\begin{restatable}{prop}{mvseq}\label{prop:mvseq}
Let $X$ be a 1-geodesic, bounded geometry metric space with $W\subseteq X$. Suppose both $X$ and $W$ are coarsely uniformly $(n-1)$-acyclic over $R$. Let $C_1$  be a $(1,A)$-coarse complementary component of $W$ and let $C_2:=X\backslash C_1$.
Then there exist  uniformly $(n-1)$-acyclic $R$-metric complexes  $(W,D_\bullet)$, $(C_1\cup W, A_\bullet)$, $(C_2\cup W,B_\bullet)$ and $(X,C_\bullet)$, and an  exact sequence 
\begin{align*}
0\rightarrow H^0_c(C_\bullet)\xrightarrow{q^*} H^0_c(A_\bullet)\oplus H^0_c(B_\bullet) \xrightarrow{p^*} H^0_c(D_\bullet)\xrightarrow{\delta} H^1_c(C_\bullet)\xrightarrow{q^*}\cdots
\\\cdots \xrightarrow{q^*} H^{n-1}_c(A_\bullet)\oplus H^{n-1}_c(B_\bullet) \xrightarrow{p^*}H^{n-1}_c(D_\bullet)\xrightarrow{\delta} \hat H^n_c(C_\bullet).
\end{align*}

Let $p^\#_{C_1}: A^\bullet_c\rightarrow D^\bullet_c$, $p^\#_{C_2}:B^\bullet_c\rightarrow D^\bullet_c$, $q^\#_{C_1}:C^\bullet_c\rightarrow A^\bullet_c$ and $q^\#_{C_2}:C^\bullet_c\rightarrow B^\bullet_c$ be maps which have finite displacement over the respective inclusions.
The above maps in cohomology are induced by the cochain maps $p^\#(\sigma,\tau)=p^\#_{C_1}\sigma +p^\#_{C_2}\tau$ and $q^\#(\lambda)=(q^\#_{C_1}\lambda,-q^\#_{C_2}\lambda)$.
Moreover, for $r<n$:
\begin{enumerate}
\item \label{prop:mvseqbdrmap} there is an $R=R(X,W,A)\geq 0$ such that whenever $\sigma\in D^{r}_c$ is a cocycle supported in $K\subseteq W$, then $\delta[\sigma]$ can be represented by a cocycle in $C^{r+1}_c$ supported in $N_R^X(K)$;
\item \label{prop:mvseqtwocomps} there exists an $s_0=s_0(X,W,A)$ such that for any $s\geq s_0$ the following holds: a cohomology class  $[\omega]\in \hat H^{r+1}_c(C_\bullet)$ is contained in the image of the boundary map if and only if $[\omega]$ can be represented by two modified cocycles $\alpha,\beta\in C^{r+1}_c$ supported in $C_1\backslash N_s(W)$ and  $C_2\backslash N_s(W)$ respectively.
\end{enumerate}
\end{restatable}

The difficulty in  deriving the above exact sequence when compared to the singular or simplicial Mayer--Vietoris sequence,  lies in the fact that maps which we want to be injective or surjective may not be. For example, although $W$ is a subspace of $C_1\cup W$, $D_\bullet$ isn't a subcomplex of $A_\bullet$, so  the map $D_\bullet\rightarrow A_\bullet$, induced by inclusion, may not be injective. 

The proof is similar to that found in \cite{higson1993coarse}, except that as	 spaces are coarsely uniformly $(n-1)$-acyclic, one doesn't need to pass to  inverse limits. We make use of the following lemma, which relates coarse cohomology to the cohomology of successive Rips complexes.

\begin{lem}\label{lem:mvseqtech}
Let $X$ be a 1-geodesic, bounded geometry metric space with $W\subseteq X$. Suppose both $X$ and $W$ are coarsely uniformly $(n-1)$-acyclic over $R$. For  every $A\geq 0$, there exist $i\leq j\leq k$ and $A\leq n_i\leq n_j\leq n_k$ such that  the following holds:
for every $(1,A)$-coarse complementary component $C$ of $W$, there exists a uniformly $(n-1)$-acyclic $R$-metric complex $(C\cup W,C_\bullet)$  and proper chain maps 
\begin{align}\label{eqn:mvseqtech}
[C_\bullet]_n \xrightarrow{\alpha_C} [C_\bullet(P_i(C_i))]_n \xrightarrow{\iota_1} [C_\bullet(P_j(C_j))]_n \xrightarrow{\iota_2} [C_\bullet(P_k(C_k))]_n \xrightarrow{\beta_C} [C_\bullet]_n
\end{align}
where:
\begin{enumerate}
\item \label{lem:mvseqtech1} $C_i:=N_{n_i}^X(W)\cup C$, $C_j:=N_{n_j}^X(W)\cup C$ and $C_k:=N_{n_k}^X(W)\cup C$;
\item  every $(1,A)$-coarse complementary component of $W$ is also an $(i,n_i)$, $(j,n_j)$ and $(k,n_k)$-coarse complementary component of $W$;
\item the maps $\iota_1$ and $\iota_2$ are inclusions;
\item the map $\alpha_C$ has displacement at most $r_\alpha=r_\alpha(X,W,A)$ over the inclusion $W\rightarrow C_i$;
\item  the map $\beta_C$ has displacement at most $r_\beta=r_\beta(X,W,A)$ over a coarse inverse to the inclusion $W\rightarrow C_k$;
\item \label{lem:mvseqtechfinal} on the level of modified cohomology with compact supports in dimensions at most $n$, the following identities hold:\begin{align}\alpha_C^*\iota_1^*\iota_2^*\beta_C^*&=\mathrm{id}\\ 
\iota_1^*\iota_2^*\beta_C^*\alpha_C^*\iota_1^*&=\iota_1^*\\
\label{eqn:2ndlevelhomtop}\iota_2^*\beta_C^*\alpha_C^*\iota_1^*\iota_2^*&=\iota_2^*.
\end{align} 
\end{enumerate}
Moreover, (\ref{eqn:mvseqtech}) is natural in the following sense: for every pair of $(1,A)$-coarse complementary components $C\subseteq D$ of $W$, there are proper chain maps 
$$ \begin{CD}
[C_\bullet]_n @>\alpha_C>> [C_\bullet(P_i(C_i))]_n @>\iota_1>> [C_\bullet(P_j(C_j))]_n@>\iota_2>> [C_\bullet(P_k(C_k))]_n @>\beta_C>> [C_\bullet]_n\\
@VVfV@VVf_iV@VVf_jV@VVf_kV@VVfV\\
[D_\bullet]_n @>\alpha_D>> [D_\bullet(P_i(D_i))]_n @>\iota_1>> [D_\bullet(P_j(D_j))]_n@>\iota_2>> [D_\bullet(P_k(D_k))]_n @>\beta_D>> [D_\bullet]_n
\end{CD}$$
where:
\begin{enumerate}
\item each row is an instance of (\ref{eqn:mvseqtech}) applied to $C$ and $D$ respectively;
\item the maps $f_i,f_j$ and $f_k$ are inclusions of chain complexes;
\item the chain map $f:[C_\bullet]_n\rightarrow [D_\bullet]_n$ has finite displacement over the inclusion $C\cup W\rightarrow D\cup W$;
\item on the level of modified cohomology with compact supports in dimensions at most $n$, the following identities hold:\begin{align}
\label{eqn:betapcommute}\beta_C^*f^*&=f_k^*\beta_D^*\\
\label{eqn:alphapcommute}\alpha_C^*f_i^*\iota_1^*&=f^*\alpha_D^*\iota_1^*.
\end{align} 
\end{enumerate}
\end{lem}

\begin{proof}
By Proposition \ref{prop:compcoarseacycl} there exist  $\lambda'$ and $\mu'$, depending only on $X$, $W$ and $A$, such that for every $(1,A)$-coarse complementary component $C$, $C\cup W$ is $(\lambda',\mu')$-coarsely uniformly $(n-1)$-acyclic.
Using Proposition \ref{prop:controlspace}, we construct a $\tilde \mu$-uniformly $(n-1)$-acyclic metric complex $(C\cup W,C_\bullet)$ such that $\tilde \mu$ and the $n$-displacement of $C_\bullet$ depend only on $X$, $W$ and $A$.

	We proceed by repeated applications of  Lemmas \ref{lem:extendmapsunif} and \ref{lem:extendmapscoarse}. Using the notation of Lemma \ref{lem:extendmapscoarse} we set $i=i(\lambda')$, and using Lemma \ref{lem:compcompcoarseinv} we  pick $n_i$ such that every $(1,A)$-coarse complementary component of $W$ is also an $(i,n_i)$-coarse complementary component of $W$. By Lemma \ref{lem:extendmapscoarse} and our choice of $i$, for every $(1,A)$-coarse complementary component $C$ of $W$, there exists a finite displacement chain map $[C_\bullet]_n\rightarrow C_\bullet(P_i(W\cup C))$. We postcompose this map with the inclusion to get the required map $\alpha_C$.

By Proposition \ref{prop:unifacyccoarseinv}, $C_i$ is $(\lambda'',\mu'')$-coarsely uniformly $(n-1)$-acyclic for some $\lambda''$ and $\mu''$, depending only on $\lambda$, $\mu$ and $n_i$. Using the notation of Lemma \ref{lem:extendmapscoarse}, we set $j=j(i,\lambda'')$. We define suitable $n_j$, $k$ and $n_k$ similarly. We choose a map $C_k\rightarrow W\cup C$ which is a coarse inverse to the inclusion, and by Lemma \ref{lem:extendmapsunif}, we can define a suitable $\beta_C$.

By Lemma \ref{lem:compdispmaps}, $\beta_C\iota_2\iota_1\alpha_C$ has finite displacement over the identity, so Lemma \ref{lem:extendmapsunif} tells us it is properly chain homotopic to the identity chain map. Thus $\alpha_C^*\iota_1^*\iota_2^*\beta_C^*=\mathrm{id}$. Similarly, we see that $\alpha_C\beta_C\iota_2\iota_1$ has finite displacement over the identity. By our choice of $j$, we see that $\iota_1\alpha_C\beta_C\iota_2\iota_1$ is properly chain homotopic to $\iota_1$, hence $\iota_1^*\iota_2^*\beta_C^*\alpha_C^*\iota_1^*=\iota_1^*$. We deduce (\ref{eqn:2ndlevelhomtop}) similarly. 

To show naturality, we use Lemma \ref{lem:extendmapsunif} to define an $f:[C_\bullet]_n\rightarrow [D_\bullet]_n$ that has finite displacement over the inclusion $C\subseteq D$. We then deduce (\ref{eqn:betapcommute}) and (\ref{eqn:alphapcommute}) using Lemmas \ref{lem:extendmapsunif} and \ref{lem:extendmapscoarse} respectively.
\end{proof}
\begin{rem}\label{rem:nklarge}
In the proof of Lemma \ref{lem:mvseqtech}, one can choose $n_k$ arbitrarily large whilst fixing $(X,C_\bullet)$, $i,j,k,n_i,n_j$ and $\alpha_C$. However, as we vary $n_k$ we lose the bound $r_\beta$ on the displacement of $\beta_C$.
\end{rem}
In the proof of Proposition \ref{prop:mvseq}, we also use the following easy lemma:
\begin{lem}\label{lem:restrictzero}
Let $X$ be a bounded geometry metric space and let $(X,C_\bullet)$ be an $R$-metric complex. Suppose that for some $k$ and $n$,  $\beta:[C_\bullet(P_k(X))]_n\rightarrow C_\bullet$ is a chain map of  $n$-displacement at most $r_\beta$. Let $L\subseteq X$ and suppose $q_\#:C_\bullet(P_k(L))\rightarrow C_\bullet(P_k(X))$ is the inclusion of chain complexes. If $\mathrm{supp}(\alpha)\cap N^X_{r_\beta}(L)=\emptyset$ for some $\alpha\in C^n_c$, then $q^\#\beta^\#\alpha=0$.
\end{lem}
\begin{proof}
It is enough to show that $q^\#\beta^\#\alpha$ vanishes on every $n$-simplex of $P_k(L)$. Let $\Delta$ be an $n$-simplex of $P_k(L)$. Then $\mathrm{supp}(q_\#\Delta)\subseteq L$, so $\mathrm{supp}(\beta_\#q_\#\Delta)\subseteq N^X_{r_\beta}(L)$; therefore $q^\#\beta^\#\alpha(\Delta)=\alpha(\beta_\#q_\#\Delta)=0$.
\end{proof}

\begin{proof} [Proof of \ref{prop:mvseq}]
We note that $W$ and $X$ are themselves $(1,A)$-coarse complementary components of $W$.
We use Lemma \ref{lem:mvseqtech} to produce the proper chain maps in Figure \ref{fig:mvseq}: the `$p$ maps' are sums of maps induced by the inclusions $W\rightarrow W\cup C_1$ and $W\rightarrow W\cup C_2$ and the `$q$ maps' are differences of the maps induced by the inclusions $W\cup C_1\rightarrow X$ and $W\cup C_2\rightarrow X$.  We let  $W_i=N^X_{n_i}(W)$, $A_i=C_1\cup W_i$, $B_i=C_2\cup W_i$, and similarly for $W_j$, $A_j$ etc. 

\begin{figure}[ht]
\caption{}
\label{fig:mvseq}
{$$\begin{CD}
[W_\bullet]_n	@>p>>	[A_\bullet]_n\oplus [B_\bullet]_n @>q>> [C_\bullet]_n\\
@VV{\alpha}V @VV{\alpha}V @VV{\alpha}V\\
[C_\bullet(P_i(W_i))]_n@>p>>[C_\bullet(P_i(A_i))]_n\oplus [C_\bullet(P_i(B_i))]_n@>q>> [C_\bullet(P_i(X))]_n\\
@VV{\iota_1}V @VV{\iota_1}V @VV{\iota_1}V\\
[C_\bullet(P_j(W_j))]_n@>p>>[C_\bullet(P_j(A_j))]_n\oplus [C_\bullet(P_j(B_j))]_n@>q>> [C_\bullet(P_j(X))]_n\\
@VV{\iota_2}V @VV{\iota_2}V @VV{\iota_2}V\\
[C_\bullet(P_k(W_k))]_n@>p>>[C_\bullet(P_k(A_k))]_n\oplus [C_\bullet(P_k(B_k))]_n@>q>> [C_\bullet(P_k(X))]_n\\
@VV{\beta}V @VV{\beta}V @VV{\beta}V\\
[W_\bullet]_n	@>p>>	[A_\bullet]_n\oplus [B_\bullet]_n @>q>> [C_\bullet]_n\\
\end{CD}$$}
\end{figure}
By Lemma \ref{lem:coarsecompsimplex}, each simplex of $P_i(X)$ is contained in either $P_i(A_i)$ or $P_i(B_i)$. This means that the second row of Figure \ref{fig:mvseq} is a short exact sequence, so the Mayer--Vietoris sequence in simplicial cohomology with compact supports gives us: \begin{equation}\label{eqn:mvinrips}
\cdots\xrightarrow{p^*} H^*_c(P_i(N_{n_i}(W)))\xrightarrow{\tilde{\delta}} H^{*+1}_c(P_i(X))\xrightarrow{q^*}\cdots.
\end{equation}We obtain similar Mayer--Vietoris sequences in the third and forth rows, and by naturality of the simplicial  Mayer--Vietoris sequence,  the boundary map $\tilde \delta$ commutes with the maps $\iota_2^*$ and $\iota_1^*$.

For $k \leq n-1$, we define the required boundary map  $\delta:=\alpha^*\iota^*_1\iota^*_2\tilde{\delta}\beta^*=\alpha^*\iota^*_1\tilde{\delta}\iota^*_2\beta^*$.
As $H^{n-1}(W_\bullet)=0$, we see that on the level of ordinary cohomology, i.e. cohomology possibly without compact supports, the boundary map $\delta$ is zero. 
Therefore, $\delta$ defines a map $\delta: H^{n-1}_c(W_\bullet)\rightarrow \hat H^n_c(C_\bullet)$. In lower dimensions, Lemma \ref{lem:modifvsordin} tells us that modified cohomology with compact supports and cohomology with compact supports are equal.

It is easy to verify exactness using the identities from Lemma \ref{lem:mvseqtech}. We shall show exactness at $H^r_c(W_\bullet)$ for $r<n$. Let $x\in H^r_c(W_\bullet)$. Using (\ref{eqn:betapcommute}), we see that $\beta^*$ and $p^*$ commute.
If $x=p^*(x')$, then as (\ref{eqn:mvinrips}) is exact, we see that $$\delta(x)=\alpha^*\iota^*_1\iota^*_2\tilde{\delta}\beta^*p^*(x')=\alpha^*\iota^*_1\iota^*_2\tilde{\delta}p^*\beta^*(x')=0.$$
We now suppose $\delta(x)=0$. Using (\ref{eqn:2ndlevelhomtop}), we see 
$$\tilde{\delta}\iota^*_2\beta^*(x)=
\iota^*_2\tilde{\delta}\beta^*(x)=
\iota_2^*\beta^*\alpha^*\iota^*_1\iota^*_2\tilde{\delta}\beta^*(x)=
\iota_2^*\beta^*\delta(x)=0.$$ By the exactness of (\ref{eqn:mvinrips}), $p^*(y)=\iota^*_2\beta^*(x)$ for some $y\in H^r_c(P_j(A_j))\oplus H^r_c(P_j(B_j))$. Let $z=\alpha^*\iota_1^*(y)\in H^r_c(A_\bullet)\oplus H^r_c(B_\bullet)$. Then $$p^*(z)=p^*\alpha^*\iota_1^*(y)=\alpha^*p^*\iota_1^*(y)=\alpha^*\iota_1^*p^*(y)=\alpha^*\iota_1^*\iota^*_2\beta^*(x)=x,$$ using (\ref{eqn:alphapcommute}) and the fact that $p^*$ and $\iota_1^*$ commute. Thus $\ker(\delta)=\mathrm{im}(p^*)$, so we have exactness at $H^r_c(W_\bullet)$.

(\ref{prop:mvseqbdrmap}): Let $r<n$.  Given an $r$-cycle $\sigma\in Z^r_c(P_k(W_k))$, we
 extend $\sigma$ to $\rho_\sigma\in C^r_c(P_k(A_k))$ by setting $\rho_\sigma$ to be zero outside of $P_k(W_k)$. We extend $\delta\rho_\sigma\in C^{r+1}_c(P_k(A_k))$ to $\omega_\sigma\in C^{r+1}_c(P_k(X))$ by  setting $\omega_\sigma$ to be zero outside $P_k(A_k)$. Since $p^\#(\rho_\sigma,0)=\sigma$ and  $q^\#(\omega_\sigma)=(\delta\rho_\sigma,0)$, we see that $\tilde \delta [\sigma]=[\omega_\sigma]$. 
It  follows that if $\gamma\in D^k_c$ is a cocycle supported in $K\subseteq W$, then $\alpha^\#\iota^\#_1\iota^\#_2\omega_{\beta^\# \gamma}$ is a cocycle representing $\delta[\gamma]$ and supported in $N^X_R(K)$, where $R$ depends only on $k$, $r_\alpha$ and $r_\beta$.

(\ref{prop:mvseqtwocomps}):  We first fix a cocycle $\gamma\in D^r_c$ and  some $s\geq 0$. We will show that  $\delta[\gamma]$ can be represented by two cocycles supported in $C_1\backslash N_s(W)$ and $C_2\backslash N_s(W)$ respectively. Defining $\omega_{\beta^\#\gamma}$ as above, we see that $q_{C_2}^\#\omega_{\beta^\#\gamma}=0$, where $q^\#_{C_2}:C^\bullet_c(P_k(X))\rightarrow C^\bullet_c(P_k(N_{n_k}(W)\cup C_2))$  is the restriction map.  Thus if $\omega_{\beta^\#\gamma}(\Delta)\neq 0$ for some $(r+1)$-simplex $\Delta\in C_{r+1}(P_k(X))$, then $\Delta$ intersects $C_1\backslash N_{n_k}(W)$. Hence by Lemma \ref{lem:nbhdcomp}, $\omega_{\beta^\#\gamma}$ is supported in $N_k(C_1\backslash N_{n_k}(W))\subseteq (C_1\cup N_{k+A}(W))\backslash N_{n_k-k}(W)$.

As in Remark \ref{rem:nklarge}, we can choose $n_k$  arbitrarily large while keeping $k$ and $\alpha$ fixed. Since $\mathrm{supp}(\omega_{\beta^\#\gamma})\subseteq (C_1\cup N_{k+A}(W))\backslash N_{n_k-k}(W)$ and $\alpha^\#\iota^\#_1\iota^\#_2\omega_{\beta^\# \gamma}$ is a cocycle representing $\delta[\gamma]$, we can choose $n_k$ large enough so that $\alpha^\#\iota^\#_1\iota^\#_2\omega_{\beta^\# \gamma}$ is supported in $C_1\backslash N_s(W)$. Reversing the roles of $C_1$ and $C_2$, we see that $\delta[\gamma]$ can be represented by a cocycle supported in $C_2\backslash N_s(W)$. 

For the converse, we keep $n_k$ fixed and  choose $s_0= r_\beta+n_k$. By Lemma \ref{lem:nbhdcomp}, we see $N_{r_\beta}(C_i\cup N_{n_k}(W))\subseteq C_i \cup N_{s_0}(W)$ for $i=1,2$.
Let $s\geq s_0$. By  Lemma \ref{lem:restrictzero} and our choice of $s$, we see that for any $\mu,\nu \in C^{n+1}_c$ supported in $C_1\backslash N_s(W)$ and $C_2\backslash N_s(W)$ respectively, then
$$q^\#_{C_1}\beta^\#\nu=q^\#_{C_2}\beta^\#\mu=p^\#_{C_1}q^\#_{C_1}\beta^\#\mu=0.$$

Suppose there exist modified cocycles $\mu,\nu\in C^{n+1}_c$,  supported in $C_1\backslash N_s(W)$ and  $C_2\backslash N_s(W)$ respectively, each representing $[\omega]\in \hat H^{n+1}_c[C_\bullet]$. Then there is a $\tau\in C^n_c$ such that $\delta\tau=\mu-\nu$.  
We see that $p^\#(q^\#_{C_1}\beta^\#\tau,0)=p^\#_{C_1}q^\#_{C_1}\beta^\#\tau$ and $$\delta(q^\#_{C_1}\beta^\#\tau,0)=(q^\#_{C_1}\beta^\#(\mu-\nu),0)=(q^\#_{C_1}\beta^\#\mu,-q^\#_{C_2}\beta^\#\mu)=q^\#\beta^\#\mu.$$ Since $\delta p^\#_{C_1}q^\#_{C_1}\beta^\#\tau=p^\#_{C_1}q^\#_{C_1}\beta^\#(\mu-\nu)=0$,  
$p^\#_{C_1}q^\#_{C_1}\beta^\#\tau$ is a cocycle with $\tilde \delta[p^\#_{C_1}q^\#_{C_1}\beta^\#\tau]=[\beta^\#\mu]=\beta^*[\omega]$. We therefore see that \begin{align*}
\delta(\alpha^*\iota_1^*\iota_2^*[p^\#_{C_1}q^\#_{C_1}\beta^\#\tau])&=\alpha^*\iota_1^*\tilde \delta\iota_2^* \beta^*\alpha^*\iota_1^*\iota_2^*[p^\#_{C_1}q^\#_{C_1}\beta^\#\tau]\\&=\alpha^*\iota_1^*\iota_2^*\tilde \delta[p^\#_{C_1}q^\#_{C_1}\beta^\#\tau]=\alpha^*\iota_1^*\iota_2^*\beta^*[\omega]=[\omega].\qedhere
\end{align*}
\end{proof}

\section{Constructing a Splitting}\label{sec:final}
\subsection{Essential Components}\label{sec:essential}
Throughout this section, we assume that $X$ is a $1$-geodesic, bounded geometry metric space that is coarsely uniformly $n$-acyclic over $\mathbb{Z}_2$, and that $W\subseteq X$ is a coarse $PD_n^{\mathbb{Z}_2}$ space. We assume all metric complexes are $\mathbb{Z}_2$-metric complexes and that  cohomology is taken with coefficients in $\mathbb{Z}_2$.

The following definition makes sense in light of Proposition \ref{prop:compcoarseacycl} which shows that if $C$ is a coarse complementary component of $W,$ then $C\cup W$ is coarsely uniformly $n$-acyclic; hence we can define the coarse cohomology $H^k_\mathrm{coarse}(C\cup W)$ for $k\leq n+1$.
\begin{defn}
 A coarse complementary component $C\subseteq X$ of $W$ is said to be \emph{essential} if the map $H^n_\mathrm{coarse}(C\cup W)\rightarrow H^n_\mathrm{coarse}(W)$, induced by the inclusion $W\rightarrow C\cup W$, is zero.
\end{defn}
The following proposition shows that essential coarse complementary components are preserved under coarse isometries.
\begin{prop}\label{prop:essentialqiinv}
Suppose $C\subseteq X$ is an essential coarse complementary component of $W$ and $f:X\rightarrow Y$ is a coarse isometry with $N^Y_B(f(X))=Y$. Then $N_B^Y(f(C))$ is an essential coarse complementary component of $f(W)$.
\end{prop}
\begin{proof}
It follows from Lemma \ref{lem:compcompcoarseinv} that $N_B^Y(f(C))$ is a coarse complementary component of $f(W)$, so we need only show that it is essential. Let $(W,D_\bullet)$, $(f(W),D'_\bullet)$, $(C\cup W,A_\bullet)$ and $(N_B^Y(f(C))\cup f(W),A'_\bullet)$ be uniformly $n$-acyclic metric complexes. Using Lemma \ref{lem:extendmapsunif}, we define proper chain maps 
$$\begin{CD} 
[D_\bullet]_{n+1} @>a_\#>> [A_\bullet]_{n+1}\\
@Vp_\#VV @Vq_\#VV\\
[D'_\bullet]_{n+1} @>a'_\#>> [A'_\bullet]_{n+1}
\end{CD}$$
where $a_\#$ and $a'_\#$ have finite displacement over the inclusions $W\rightarrow C\cup W$ and $f(W)\rightarrow N_B^Y(f(C))\cup f(W)$ respectively, and $p_\#$ and $q_\#$ have finite displacement over $f|_W$ and $f|_{C\cup W}$ respectively. By Lemma \ref{lem:extendmapsunif}, we see that $q_\#a_\#$ is properly chain homotopic to $a'_\#p_\#$. Since $C$ is essential, the map $a^*: H^n_c(A_\bullet)\rightarrow H^n_c(D_\bullet)$ is zero, therefore $a^*q^*=p^*(a')^*: H^n_c(A'_\bullet)\rightarrow H^n_c(D_\bullet)$ is also zero.  As $f|_W:W\rightarrow f(W)$ is a coarse isometry,  $p^*$ is an isomorphism and so $(a')^*$ is zero; hence  $N_B(f(C))$ is essential.
\end{proof}

\begin{prop}\label{prop:essentialbasicprops}
Suppose $C$ is a subspace of $X$ such that $C\cup W$ is coarsely uniformly $n$-acyclic and  $H^n_\mathrm{coarse}(C\cup W)\rightarrow H^n_\mathrm{coarse}(W)$, induced by inclusion, is zero.  Then:
\begin{enumerate}
\item \label{prop:essentialbasicprops_deep} for every $R\geq 0$, $C$ is not contained in $N^X_R(W)$;
\item \label{prop:essentialbasicprops_goingup} if $D$ is a coarse complementary component of $W$ and $C\subseteq D$, then $D$ is essential;
\item \label{prop:essentialbasicprops_goingdown} if there exist coarse complementary components $C_1$ and $C_2$ of $W$ such that $C_1\cup C_2=C$, then at least one of $C_1$ or $C_2$ is essential.
\end{enumerate}
\end{prop}
\begin{proof}
(\ref{prop:essentialbasicprops_deep}): If $C\subseteq N^X_R(W)$ for some $R\geq 0$, then  the inclusion $W\rightarrow C\cup W$ is a coarse isometry, so induces an isomorphism $H^n_\mathrm{coarse}(C\cup W)\rightarrow H^n_\mathrm{coarse}(W)\cong \mathbb{Z}_2$. 

(\ref{prop:essentialbasicprops_goingup}): Let $f:W\rightarrow C\cup W$ and $g:C\cup W\rightarrow D\cup W$ be inclusions. The composition $H^n_\mathrm{coarse}(D\cup W)\xrightarrow{g^*} H^n_\mathrm{coarse}(C\cup W)\xrightarrow{f^*} H^n_\mathrm{coarse}(W)$ is zero, since $f^*$ is zero. Therefore $(gf)^*=f^*g^*=0$, so $D$ is essential.

(\ref{prop:essentialbasicprops_goingdown}): For $i=1,2$, let $a_i:W\rightarrow W\cup C_i$ and $b_i:W\cup C_i\rightarrow W\cup C$ be inclusions. The relevant terms in the coarse Mayer--Vietoris sequence from Proposition \ref{prop:mvseq} are $$\resizebox{1.0 \textwidth}{!}{$\cdots \rightarrow H^n_\mathrm{coarse}(C\cup W)\xrightarrow{b^*=b_1^*-b_2^*} H^n_\mathrm{coarse}(C_1\cup W)\oplus H^n_\mathrm{coarse}(C_2\cup W)\xrightarrow{a^*=a_1^*+a_2^*} H^n_\mathrm{coarse}(W)\rightarrow \cdots.$}$$
If both $a_1^*$ and $a_2^*$ are non-zero, then as $H^n_\mathrm{coarse}(W)\cong \mathbb{Z}_2$, there exist $x,y$ such that $a_1^*(x)=a_2^*(y)=1$. As $a^*(x,-y)=a_1^*(x)-a_2^*(y)=0$, there exists a $z \in H^n_\mathrm{coarse}(C\cup W)$ such that $b^*(z)=(b_1^*(z),-b_2^*(z))=(x,-y)$. Therefore $(b_1a_1)^*(z)=a_1^*b_1^*(z)=a_1^*(x)=1$, contradicting our hypothesis on $C$.
\end{proof}

\begin{rem}\label{rem:essentacycl}
If $C\cup W$ is $(n-1)$-acyclic at infinity over $\mathbb{Z}_2$, then Theorem \ref{prop:topatinf} tells us that $H^n_\mathrm{coarse}(C\cup W)=0$, and so the conclusions of Proposition \ref{prop:essentialbasicprops} hold.
\end{rem}

\begin{defn}
Let $C$ be a bounded geometry metric space and $W\subseteq C$ be a coarse $PD_n^{\mathbb{Z}_2}$ space. We say that $(C,W)$ is a \emph{coarse $PD_{n+1}^{\mathbb{Z}_2}$ half-space} if there exists a coarse $PD_{n+1}^{\mathbb{Z}_2}$ space $Y$ with $C\subseteq Y,$ such that $C$ and $Y\backslash C$ are both deep coarse complementary components of $W$ in $Y$.
\end{defn}

The coarse Jordan separation theorem \cite[Corollary 7.8]{kapovich2005coarse} states that if $f:W\rightarrow Y$ is a coarse embedding of a coarse $PD_n^{\mathbb{Z}_2}$ space into a coarse $PD_{n+1}^{\mathbb{Z}_2}$ space, then $Y$ contains two disjoint coarse complementary components of $f(W)$ that are coarse $PD_{n+1}^{\mathbb{Z}_2}$ half-spaces. 

\begin{prop}\label{prop:pdnhalf}
Let $W\subseteq C\subseteq X$ and suppose $(C,W)$ is a coarse $PD_{n+1}^{\mathbb{Z}_2}$ half-space. Then the map $H^n_\mathrm{coarse}(C\cup W)\rightarrow H^n_\mathrm{coarse}(W)$, induced by inclusion, is zero.
\end{prop}
\begin{proof}
Since $(C,W)$ is a coarse $PD_{n+1}^{\mathbb{Z}_2}$ half-space, there exists a coarse $PD_{n+1}^{\mathbb{Z}_2}$ space $Y$ with $C\subseteq Y$, such that $C$ and $C':=Y\backslash C$ are deep, coarse disjoint, coarse complementary components of $W$. Proposition \ref{prop:mvseq} then tells us that there is an exact sequence $$\cdots \rightarrow H^n_\mathrm{coarse}(C\cup W)\oplus H^n_\mathrm{coarse}(C'\cup W)\rightarrow H^n_\mathrm{coarse}(W)\rightarrow H^{n+1}_\mathrm{coarse}(Y)\rightarrow \cdots.$$ 
We pick a metric complex $(Y,D_\bullet)$ such that Proposition \ref{prop:mvseq} holds; by Proposition \ref{prop:pdncoarseisom}, $(Y,D_\bullet)$ is a coarse $PD_{n+1}^{\mathbb{Z}_2}$ complex. Using the duality  map, we observe that there is a number $R$ such that for each $y\in Y$, the non-trivial class $[\sigma]$ of  $H^{n+1}_c(D_\bullet)\cong \mathbb{Z}_2$ can be represented by a cocycle supported in $N^Y_R(y)$. Thus we can represent $[\sigma]$ by cocycles supported in $C\backslash N_s^Y(W)$ and $C'\backslash N_s^Y(W)$ for any $s\geq 0$. Then Part (\ref{prop:mvseqtwocomps}) of Proposition \ref{prop:mvseq} tells us that $H^n_\mathrm{coarse}(W)\rightarrow H^{n+1}_\mathrm{coarse}(Y)$ is non-zero, hence an isomorphism. Therefore the map $H^n_\mathrm{coarse}(C\cup W)\rightarrow H^n_\mathrm{coarse}(W)$, induced by inclusion, is zero.
\end{proof}

\begin{cor}\label{cor:pdnhalf}
Suppose $D\subseteq X$ is a coarse complementary component of $W$ and there exists a coarse $PD_{n+1}^{\mathbb{Z}_2}$ half-space $(C,W)$ such that $C\subseteq D$. Then $D$ is essential.
\end{cor}
\begin{proof}
We apply Proposition \ref{prop:pdnhalf}, followed by Part (\ref{prop:essentialbasicprops_goingup}) of Proposition \ref{prop:essentialbasicprops} .
\end{proof}

The following proposition shows that an essential coarse complementary component of $W$ is  `uniformly close' to every point of $W$.

\begin{prop}\label{prop:deepcondition}
For every  $A\geq 0$, there is a number $B=B(X,W,A)$ such that whenever $C\subseteq X$ is an essential $(1,A)$-coarse  complementary component of $W,$ then $W\subseteq N_B(C\backslash N_A(W))$. 
\end{prop}
\begin{proof}
By Lemma \ref{lem:mvseqtech}, there exists a uniformly acyclic metric complex $(W,D_\bullet)$, and numbers $j\leq k$ and $A\leq n_j\leq n_k$ depending only on $X$, $W$ and $A$ such that the following holds: there exists a proper chain map $\beta:[C_\bullet(P_k(N_{n_k}(W)))]_{n+1}\rightarrow D_\bullet$ such that if $\iota^*_2:H^n_c(P_k(N_{n_k}(W)))\rightarrow H^n_c(P_j(N_{n_j}(W)))$ is the restriction map, then $\iota^*_2\beta^*$ is injective. As $C$ is essential, Lemma \ref{lem:mvseqtech} also tells us that $\iota_2^*p_C^*$ is zero, where    $p_C^\#:C^{\bullet}_c(P_k(N_{n_k}(W)\cup C))\rightarrow C^{\bullet}_c(P_k(N_{n_k}(W)))$ is the restriction map.

By Proposition \ref{prop:pdncoarseisom}, $(W,D_\bullet)$ is a coarse $PD_n^{\mathbb{Z}_2}$  complex, hence there is a duality map $\overline P:D_\bullet\rightarrow D^{n-\bullet}_c$. For each $x\in W$, we pick $\sigma_x\in p_0^{-1}(x)\subseteq\Sigma_0$. Each cocycle $\overline{P}(\sigma_x)\in D^n_c$ represents the non-trivial element of $H^n_c(D_\bullet)\cong \mathbb{Z}_2$.

Since $\beta$ and $P$ have finite displacement, there exists a $D\geq 0$ depending only on $X$, $W$ and $A$ such that for every $x\in W$, $\beta^\#\overline P\sigma_x$ is supported in $N_D(x)$.  We extend $\beta^\#\overline P\sigma_x$ to a cochain $\omega_x\in C^n_c(P_k(N_{n_k}(W)\cup C))$ by setting $\omega_x$ to be zero outside $P_k(N_{n_k}(W))$. We note that $p_C^\#\omega_x=\beta^\#\overline P\sigma_x$.

We suppose for contradiction that $\omega_x$ is a cocycle. Then $\iota_2^*\beta^*[\overline P\sigma_x]=\iota_2^*p_C^*[\omega_x]=0$. This contradicts the fact that $\iota_2^*\beta^*$ is injective. 
Thus $\omega_x$ is not a cocycle. It follows that there is some $(n+1)$-simplex $\Delta$ of $P_k(N_{n_k}(W)\cup C)$, not contained in $P_k(N_{n_k}(W))$, such that $(\delta\omega_x)(\Delta)=\omega_x(\partial\Delta)\neq 0$. Thus there exists some $y_x\in C\backslash N_{n_k}(W)\subseteq C\backslash N_A(W)$ with $d(x,y_x)\leq k+D$.
\end{proof}

A coarse complementary component $C$ is said to be \emph{almost essential} if it satisfies the conclusion of Proposition \ref{prop:deepcondition}, i.e. for every $A\geq0$ there is a number $B\geq 0$ such that $W\subseteq N_B(C\backslash N_A(W))$. The top component of Figure \ref{fig:ess2} in the introduction gives an almost essential component which is not essential.

\begin{ques}
Does there exist a group $G$ of type $FP_{n+1}^{\mathbb{Z}_2}$ containing a coarse $PD_n^{\mathbb{Z}_2}$ subgroup $H\leq G$ and a coarse complementary component $C$ of $H$ that is almost essential but not essential?
\end{ques}
We would expect the answer to this to be positive, but  have been unable to find such a group.

\begin{defn}
We say that $W \subseteq X$ is \emph{essentially embedded} if every deep, coarse complementary component of $W$ is essential. We say that $W\subseteq G$ is  \emph{almost essentially embedded} if for every $A\geq 0$, there exists a $B=B(X,W,A)\geq 0$  such that $W\subseteq N_B(C\backslash N_A(W))$ for every deep $(1,A)$-coarse complementary component $C$ of $W$.
\end{defn}

The almost essentially embedded condition is a reformulation of the deep condition of \cite{vavrichek2012coarse}. Using Proposition \ref{prop:relendscoarsep}, we see a subgroup $K$ coarsely separates $G$ if and only if $\tilde e(G,K)>1$. We can therefore deduce the following from Lemmas 7.6 and 8.12 of \cite{vavrichek2012coarse}.

\begin{prop}[{\cite[Lemmas 7.6 and 8.12]{vavrichek2012coarse}}]\label{prop:deepcvsminimalsepp}
Let $G$ be a finitely generated group and $H\leq G$ be a subgroup. Then $H$ is almost essentially embedded if and only if no infinite index subgroup $K$ of $H$ coarsely separates $G$.
\end{prop}

It is also shown in \cite{vavrichek2012coarse} that being almost essentially embedded is a quasi-isometry invariant  in the following sense:

\begin{prop}[{\cite[Lemma 2.10]{vavrichek2012coarse}}]\label{prop:deepconditionqiinv}
Let $G$ and $G'$ be  finitely generated groups with subgroups $H\leq G$ and $H'\leq G$. Suppose there is a quasi-isometry $f:G\rightarrow G'$ such that $d_\mathrm{Haus}(f(H),H')<\infty$. Then $H$ is almost essentially embedded if and only if $H'$ is almost essentially embedded.
\end{prop}

By imposing additional hypothesis, we can strengthen Proposition \ref{prop:deepcvsminimalsepp} to give a criterion for when $H$ is essentially embedded.

\essentialminimal
\begin{proof}
If $H$ is essentially embedded, then Proposition \ref{prop:deepcondition} says  it is almost essentially embedded. Thus Proposition \ref{prop:deepcvsminimalsepp} tells us that no infinite index subgroup $K$ of $H$ coarsely separates $G$.

Now suppose that no infinite index subgroup $K$ of $H$ coarsely separates $G$.
Let $C$ be a deep coarse complementary component of $H$. We will show that $C$ is essential.  We explain how to reduce to the case where $HC=C$.  As $C$ is deep, Lemma \ref{lem:sbgpshallow} tells us that for some $A\geq 0$, $C$ contains a deep, irreducible $(1,A)$-coarse complementary component $C'\subseteq C$. By Corollary \ref{cor:stabfiniteindex}, there is a finite index subgroup $H'\leq H$ such that $H'(C'\backslash N_A(H))=C'\backslash N_A(H)$. Proposition \ref{prop:essentialbasicprops} then tells that if $C'\backslash N_A(H)$ is essential as a coarse complementary component of $H'$, then $C$ is essential as a coarse complementary component of $H$. Thus we may assume that $HC=C$.

Since $G$ is $(n-1)$-acyclic at infinity, Proposition \ref{prop:topatinf} says that $H^k_\mathrm{coarse}(G)=0$ for $k\leq n$. 
We now apply the coarse Mayer--Vietoris sequence to the coarse complementary components $C$ and $G\backslash C$. This gives the exact sequence $$0=H^k_\mathrm{coarse}(G)\rightarrow H^k_\mathrm{coarse}(H\cup C)\oplus H^k_\mathrm{coarse}(H\cup (G\backslash C)) \rightarrow H^k_\mathrm{coarse}(H)$$ for each $k\leq n$. If $C$ is not essential, then the maps $H^k_\mathrm{coarse}(H\cup C)\rightarrow H^k_\mathrm{coarse}(H)$, induced by inclusion, are isomorphisms for each $k\leq n$. Thus by Lemma \ref{lem:finindex}, $C$ must be shallow; this contradicts our choice of $C$, so  $C$ is essential.
\end{proof}

\begin{lem}[Non-Crossing Lemma]\label{lem:noncrossing}
Let $C_1$, $C_2$ and $C_3$ be deep, disjoint, $(1,A)$-coarse complementary components of $W$ in $X$, and suppose   $C_3$ is essential.
Then there  exists an $s_0=s_0(X,W,A)$ and a uniformly $n$-acyclic metric complex $(X,C_\bullet)$ such that whenever $[\omega]\in \hat H^{n+1}_c(C_\bullet)$ is represented by two cocycles supported in $C_1\backslash N_{s_0}(W)$ and  $C_2\backslash N_{s_0}(W)$ respectively, then $[\omega]=0$.
\end{lem}
\begin{proof}
Proposition \ref{prop:mvseq} says that there exist uniformly $n$-acyclic metric complexes $(C_1\cup W,A_\bullet)$, $(C_2\cup C_3\cup W, B_\bullet)$ and $(W,D_\bullet)$ and a Mayer--Vietoris sequence $$\cdots \xrightarrow{q^*} H^n_c(A_\bullet)\oplus H^n_c(B_\bullet)\xrightarrow{p^*} H^n_c(D_\bullet) \xrightarrow{\delta} \hat H^{n+1}_c(C_\bullet).$$ We also choose $i,j,k,n_i,n_j,n_k$ and $\beta:C_\bullet[P_k(X)]_n\rightarrow C_\bullet$ as in Lemma \ref{lem:mvseqtech}.

The proof proceeds in a similar fashion to the proof of (\ref{prop:mvseqtwocomps}) in Proposition \ref{prop:mvseq}, except we have three components rather than two.
As in Proposition \ref{prop:mvseq}, we set $s_0:=r_\beta+n_k$, which depends only on $X$, $W$ and $A$.
We consider the restriction maps
\begin{align*}
q_{C_i}^\#&:C^\bullet_c(P_k(X))\rightarrow C^\bullet_c(P_k(N_{n_k}(W)\cup C_i))\textrm{ for $i=1,2,3$}\\
p_{C_i}^\#&:C^\bullet_c(P_k(N_{n_k}(W)\cup C_i))\rightarrow C^\bullet_c(P_k(N_{n_k}(W)))\textrm{ for $i=1,2,3$},
\end{align*}
noting	 that $p^\#_{C_1}q^\#_{C_1}=p^\#_{C_3}q^\#_{C_3}$ is the restriction map $C^\bullet_c(P_k(X))\rightarrow C^\bullet_c(P_k(N_{n_k}(W)))$.

Suppose there exist cocycles $\mu,\nu\in C^{n+1}_c$,  supported in $C_1\backslash N_{s_0}(W)$ and  $C_2\backslash N_{s_0}(W)$ respectively, each representing $[\omega]$. Then there is a $\tau\in C^n_c$ such that $\delta\tau=\mu-\nu$. Let $\tilde \delta$ be the boundary map associated to the Mayer--Vietoris sequence of Rips complexes.  As in the proof of Proposition \ref{prop:mvseq}, we see that $\tilde \delta[p^\#_{C_1}q^\#_{C_1}\beta^\#\tau]=[\beta^\#\mu]=\beta^*[\omega]$.
 By Lemma \ref{lem:restrictzero} and our choice of $s_0$, we see that  $q^\#_{C_3}\beta^\#\delta\tau=q^\#_{C_3}\beta^\#(\mu-\nu)=0$. Therefore $q^\#_{C_3}\beta^\#\tau$ is a cocycle and $\tilde\delta p^*_{C_3}[q^\#_{C_3}\beta^\#\tau]=\tilde \delta[p^\#_{C_1}q^\#_{C_1}\beta^\#\tau]=\beta^*[\omega]$.

As $C_3$ is essential, Lemma \ref{lem:mvseqtech} tells us that the map $\iota_2^*p^*_{C_3}:H^n_c(P_k(N_{n_k}(W)\cup C_3))\rightarrow H^n_c(P_j(N_{n_j}(W)))$ is zero. As noted in the proof of Proposition \ref{prop:mvseq}, the boundary map $\tilde\delta$ commutes with the restriction map $\iota_2^*$, so 
$[\omega]=\alpha^*\iota_1^*\iota_2^*\beta^*[\omega]=\alpha^*\iota_1^*\tilde\delta\iota_2^* p^*_{C_3}[q^\#_{C_3}\beta^\#\tau]=0$.
\end{proof}

\subsection{Kleiner's Mobility Sets}\label{sec:mobility}
We now outline a method from an unpublished manuscript of Kleiner.  For each coarse cohomology class $[\sigma]\in H^k_\mathrm{coarse}(G)$, Kleiner produces a subset $\mathrm{Mob}([\sigma])\subseteq G$ which  is the support of all possible cocycles of uniformly bounded diameter representing $[\sigma]$. All the results in this subsection are contained in \cite{kleinercohomology}.

Suppose $G$ is a finitely generated group and $(G,C_\bullet)$ is a metric complex admitting a free $G$-action. We let $Z^k_c\leq C^k_c$ be the set of $k$-cocycles with compact support.
For $D>0 $, we define  $$Z^k_c(D):=\{\alpha \in Z^k_c\mid \mathrm{diam}(\alpha)\leq D\}.$$
If a group $G$ acts  cocompactly on $(X,C_\bullet)$, then $G$ acts on $Z^k_c(D)$ via the right action
$(\alpha g)(\sigma)=\alpha(g\sigma).$ Note that $\mathrm{supp}(\alpha)=g\mathrm{supp}(\alpha g)$ for all $g\in G$.  We define $\mathrm{Stab}([\alpha_0])\leq G$ to be the subgroup of $G$ which preserves the cohomology class $[\alpha_0]$.
For each $\alpha_0\in Z^k_c$, let $Z([\alpha_0],D):=\{\alpha\in Z^k_c(D) \mid  [\alpha]=[\alpha_0]\}.$
We now define the \emph{mobility set} to be  $$\mathrm{Mob}([\alpha_0],D)=\bigcup_{\alpha\in Z([\alpha_0],D)} \mathrm{supp}(\alpha).$$

\begin{lem}\label{lem:mobset}
Let $G$ be  a finitely generated group  and suppose  $(G,C_\bullet)$ is a metric complex admitting free $G$-action. For all $D\geq 0$ and $k\in \mathbb{N}$, suppose $\alpha_0\in Z^k_c(D)$ is a nonzero cocycle. Then there exists an $R\geq 0$ such that $d_\mathrm{Haus}(\mathrm{Stab}([\alpha_0])\mathrm{supp}(\alpha_0),\mathrm{Mob}([\alpha_0],D))\leq R$.
\end{lem}

\begin{proof}
We fix some $\alpha_0\in Z^{k}_c(D)$ and consider the set $$Q:=\{\alpha g\mid \alpha\in Z([\alpha_0],D), g\in G \textrm{ and } \mathrm{supp}(\alpha g)\subseteq N_D(e)\}.$$ As there are only finitely many cocycles supported in $N_D(e)$, the set $Q$ is finite. Therefore we can write $Q=\{\alpha_1g_1,\dots, \alpha_ng_n\}$, where $\alpha_i\in Z([\alpha_0],D)$ and $g_i\in G$ for each $i$. 

We now choose ${R_{\alpha_0}}$ sufficiently large so that  $\mathrm{supp}(\alpha_i)\subseteq N_{R_{\alpha_0}}(\mathrm{supp}(\alpha_0))$ for each $\alpha_i$. We claim that $d_\mathrm{Haus}(\mathrm{Stab}([\alpha_0])\mathrm{supp}(\alpha_0),\mathrm{Mob}([\alpha_0],D))\leq {R_{\alpha_0}}$.
Indeed, 
 $g \mathrm{supp}(\alpha_0)=\mathrm{supp}(\alpha_0 g^{-1})\subseteq \text{Mob}([\alpha_0],D)$ for each $g\in \text{Stab}([\alpha_0])$.
Conversely, suppose $g\in \text{Mob}([\alpha_0],D)$. Then there exists an $\alpha\in Z([\alpha_0],D)$ such that $g\in \mathrm{supp}(\alpha)$. Since $\mathrm{supp}(\alpha g)\in Q$, $\alpha g =\alpha_i g_i$ for some $i$. Hence $g\in \mathrm{supp}(\alpha_i g_i g^{-1})\subseteq gg_i^{-1} N_{R_{\alpha_0}}(\mathrm{supp}(\alpha_0))$. This proves the claim, since $gg_i^{-1}\in \text{Stab}([\alpha_0])$.

We define $R:=\max \{R_\alpha\mid \alpha\in Z^{n+1}_c(D) \text{ and }\mathrm{supp}(\alpha)\subseteq N_D(e)\}$.
Suppose $\alpha\in Z^{n+1}_c(D)$ and $g\in \mathrm{supp}(\alpha)$. As $\mathrm{supp}(\alpha g)\subseteq N_D(e)$,  $d_\mathrm{Haus}(\mathrm{Stab}([\alpha g ])\mathrm{supp}(\alpha g),\mathrm{Mob}([\alpha g],D))\leq R.$ Since $\mathrm{Stab}([\alpha g ])\mathrm{supp}(\alpha g)=g^{-1}\mathrm{Stab}([\alpha ])\mathrm{supp}(\alpha)$ and $\mathrm{Mob}([\alpha g],D)=g^{-1}\mathrm{Mob}([\alpha],D)$, it follows that $d_\mathrm{Haus}(\mathrm{Stab}([\alpha ])\mathrm{supp}(\alpha),\mathrm{Mob}([\alpha],D))\leq R.$
\end{proof}

The following corollary  shows that the mobility set of a coarse cohomology class has finite Hausdorff distance from the subgroup which stabilizes that class. This provides a  connection between geometry and algebra that allows us to construct a splitting.

\begin{cor}\label{cor:mobsetcocomp}
For any $D>0$ such that $\text{Mob}([\alpha_0],D)\neq \emptyset$,  $\text{Stab}([\alpha_0])$ has finite Hausdorff distance from $\text{Mob}([\alpha_0],D)$.
\end{cor}
\begin{proof}
By Lemma \ref{lem:mobset}, we see that $d_{\mathrm{Haus}}(\text{Mob}([\alpha_0],D),\text{Stab}([\alpha_0]))\leq R+ d(e,\mathrm{supp}(\alpha_0))$.
\end{proof}

\begin{prop}\label{prop:qimob}
Suppose  $(G,C_\bullet)$ and $(G',D_\bullet)$ are uniformly $(n-1)$-acyclic $R$-metric complexes. Say $f:G\rightarrow G'$ is a coarse isometry and $0\neq [\alpha_0]\in \hat{H}_c^n(D_\bullet)$. Then if for any $D,D'$  sufficiently large,  $$d_\mathrm{Haus}( f(\mathrm{Mob}(f^*[\alpha_0],D)),\mathrm{Mob}([\alpha_0],D'))<\infty.$$ 
\end{prop}
\begin{proof}
Suppose $C_\bullet$ and $D_\bullet$ have $n$-displacement at most $d$.
Lemma \ref{lem:extendmapsunif} tells us that $f$ induces a chain map $f_\#:[C_\bullet]_n\rightarrow [D_\bullet]_n$ of $n$-displacement at most $M$ over $f$. We assume $D'$ is large enough so that $\text{Mob}([\alpha_0],D')$ is nonempty. Let $\alpha\in Z([\alpha_0],D')$. The proof of Proposition \ref{prop:propchainmap} then shows that $f(\mathrm{supp}(f^\#\alpha))\subseteq N_M^{G'}(\mathrm{supp}(\alpha))$. In particular, there is a $D$ sufficiently large such that $\mathrm{diam}(f^\#\alpha)\leq D$ for all  $\alpha\in Z([\alpha_0],D')$.

 Since as $f$ is a coarse isometry, $f^\#$ induces an isomorphism in modified cohomology; therefore $f^\#\alpha\neq 0$. Thus $\mathrm{supp}(\alpha)\subseteq N_{M+D'}(f(\mathrm{supp}(f^\#\alpha)))$. Therefore $$\mathrm{Mob}([\alpha_0],D')\subseteq N_{M+D'}(f(\mathrm{Mob}(f^*[\alpha_0],D))).$$
Using a coarse inverse $g:Y\rightarrow X$ to $f$, the same argument shows that $f(\mathrm{Mob}(f^*[\alpha_0],D))\subseteq N_R(\mathrm{Mob}([\alpha_0],D''))$   for some suitable $D''$ and $R$. Lemma \ref{lem:mobset} ensures that $\mathrm{Mob}([\alpha_0],D'')$ and $\mathrm{Mob}([\alpha_0],D')$ have finite Hausdorff distance, hence $\mathrm{Mob}([\alpha_0],D')$ and $f(\mathrm{Mob}(f^*[\alpha_0],D))$ have finite Hausdorff distance.
\end{proof}
\subsection{The Main Theorem}\label{sec:splitting}

\begin{lem}\label{lem:aisetkrop}
Let $G$ be a finitely generated group with $H\leq G$ a subgroup. Suppose there exists a constant $A\geq 0$ and a coarse complementary component $C$ of $H$ such both $C$ and $G\backslash C$ are deep, and  for every $g\in G$, either $gH\subseteq C\cup N_A(H)$ or $gH\subseteq (G\backslash C) \cup N_A(H)$. Then there exists a proper $H$-almost invariant set $X$ such that $XH=X$.
\end{lem}
\begin{proof}
Let $X:=\{g\in G\mid gH\subseteq C\cup N_A(H)\}$. It is clear that $XH=X$, so we need only show $X$ is a proper $H$-almost invariant set. 
Suppose $g\in C\backslash N_A(H)$. Then $gH\cap (C\backslash N_A(H))\neq \emptyset$ so $gH\subseteq C\cup N_A(H)$; therefore $g\in X$. Conversely, suppose $g\in X\backslash N_A(H)$. Then because $g\in gH\subseteq C\cup N_A(H)$, it follows $g\in C$. Thus $C$ and $X$ are equal outside $N_A(H)$. In particular, $X$ and $G\backslash X$ are  deep coarse complementary components of $H$ (since $C$ and $G\backslash C$ are), so Proposition \ref{prop:coarsecompinvecspace} tells us that $X$ is a proper $H$-almost invariant set.
\end{proof}

The proof of the following Lemma is based on an analogous statement for the  $n=1$ case found in \cite{kleinercohomology}.

\begin{lem}\label{lem:sbgp}
Let $G$ be a group of type $FP_{n+1}^{\mathbb{Z}_2}$ and let $W\subseteq G$ be a coarse $PD_n^{\mathbb{Z}_2}$ subspace. Suppose $G$ contains three essential, coarse disjoint, coarse complementary components of $W$. Then there exists a subgroup $H\leq G$, at finite Hausdorff distance from $W$, and a proper $H$-almost invariant set $X$ such that $XH=X$.
\end{lem}
\begin{proof}
Say $C_1$, $C_2$ and $C_3$ are three essential, coarse disjoint, coarse complementary components of $W$.
Without loss of generality, we may replace $C_1$ with $G\backslash (C_2\cup C_3)$ and $C_2$ with $C_2\backslash C_3$, so that $C_1$, $C_2$ and $C_3$ are disjoint. Using Corollary \ref{cor:compcompparamindep}, we may assume $C_1$, $C_2$ and $C_3$ are each $(1,A)$-coarse complementary components for some $A\geq 0$ large enough. 

By Proposition \ref{prop:mvseq}, there exist uniformly $(n-1)$-acyclic metric complexes $(W,D_\bullet)$, $(C_1\cup W,A_\bullet)$, $(C_2\cup C_3 \cup W,B_\bullet)$ and $(G,C_\bullet)$, and an exact sequence $$\cdots \xrightarrow{q^*} H^n_c(A_\bullet)\oplus H^n_c(B_\bullet)\xrightarrow{p^*} H^n_c(D_\bullet)\xrightarrow{\delta} \hat H^{n+1}_c(C_\bullet).$$ As $C_1$ and $C_2\cup C_3$ are essential, the map $\delta:H^n_c(D_\bullet)\rightarrow \hat H^{n+1}_c(C_\bullet)$ is injective and has image containing some non-trivial $[\alpha_1]\in \hat H^{n+1}_c(C_\bullet)$. 

As in the proof of Proposition \ref{prop:deepcondition},  there exists a $D\geq 0$ such that for every $x\in W$, there is a cocycle supported in $N^W_D(x)$ representing the non-trivial element of $H^n_c(D_\bullet)\cong \mathbb{Z}_2$. Thus  Proposition \ref{prop:mvseq} tells us that there exists a $D'=D'(D,G,W,A)\geq 0$ such that for each $x\in W$, $[\alpha_1]\in \hat H^{n+1}_c(C_\bullet)$ can be represented by a cocycle supported in $N_{D'}^G(x)$. Therefore  $W\subseteq N_{D'}(\text{Mob}([\alpha_1],D'))$.

Part (2) of Proposition \ref{prop:mvseq} and Lemma \ref{lem:noncrossing} say that for $s$ sufficiently large, $[\alpha_1]\neq 0$ can be represented by a cocycle supported in $C_1\backslash N_s(W)$, but not in  either of $C_2\backslash N_{s}(W)$ or $C_3\backslash N_{s}(W)$.  Therefore, for any $\sigma \in Z([\alpha_1],D')$,  $\mathrm{supp}(\sigma)$ cannot intersect either $C_2\backslash N_{s+D'}(W)$ or $C_3\backslash N_{s+D'}(W)$. Thus $\text{Mob}([\alpha_1],D')\subseteq C_1 \cup N_{s+D'}(W)$. 

The constants $s$ and $D'$ depend only on $G$, $W$, and $A$. Thus we can  interchange the roles of $C_1$ and $C_2$ to obtain a non-trivial cohomology class $[\alpha_2]\in H^{n+1}_\mathrm{coarse}(X)$ such that $W\subseteq N_{D'}(\text{Mob}([\alpha_2],D'))$ and $\text{Mob}([\alpha_2],D')\subseteq C_2 \cup N_{s+D'}(W)$.

Let $H_i=\mathrm{stab}([\alpha_i])$ and $F_i=\text{Mob}([\alpha_i],D')$ for $i=1,2$. Since each $F_i$ is non-empty, Corollary \ref{cor:mobsetcocomp} says that each $H_i$ has finite Hausdorff distance from $F_i$. Thus there is an $R\geq D'$ large enough so that $H_i\subseteq N_R(F_i)$ and $F_i\subseteq N_R(H_i)$ for $i=1,2$. By Lemma \ref{lem:nbhdcomp}, $$H_i\subseteq N_R(F_i)\subseteq N_R(C_i \cup N_{s+D'}(W))\subseteq C_i\cup N_{s+D'+A+R}(W)$$ for $i=1,2$. Hence $$H_1\cap H_2\subseteq N_R(F_1)\cap N_R(F_2)\subseteq N_{s+D'+A+R}(W)$$ and $$W\subseteq N_{D'}(F_1)\cap N_{D'}(F_2)\subseteq N_{R+D'}(H_1)\cap N_{R+D'}(H_2).$$ By Lemma \ref{lem:cocompactintersec}, $H:=H_1\cap H_2$ has finite Hausdorff distance from $N_{R+D'}(H_1)\cap N_{R+D'}(H_2)$. Therefore, $W$ has finite Hausdorff distance from $H:=H_1\cap H_2$.

Let $g\in G$. We recall that for each $x\in W$, there exists a cocycle supported in $N_{D'}(x)$ representing $[\alpha_1]\neq 0$. Hence for each $x\in gW$, there is a cocycle supported in $N_{D'}(x)$ representing $[\alpha_1\cdot g^{-1}]\neq 0$; therefore $gW\subseteq N_{D'}(\mathrm{Mob}([\alpha_1g^{-1}],D'))$. 

By Lemma \ref{lem:noncrossing} and our choice of $s$,  $[\alpha_1\cdot g^{-1}]$ cannot be represented by cocycles supported in any two of  $C_1\backslash N_s(W)$ and $C_2\backslash N_s(W)$ and $C_3\backslash N_s(W)$. Therefore, either $\mathrm{Mob}([\alpha_1g^{-1}],D')\subseteq C_1 \cup N_{s+D'}(W)$ or $\mathrm{Mob}([\alpha_1g^{-1}],D')\subseteq C_2 \cup C_3 \cup N_{s+D'}(W)$. Hence by Lemma \ref{lem:nbhdcomp}, either $gW\subseteq C_1\cup N_{s+D'+A}(W)$ or $gW\subseteq C_2\cup C_3\cup N_{s+D'+A}(W)$.

Since $H$ and $W$ are at finite Hausdorff distance, $C_1$ is a coarse complementary component of $H$. Moreover,  both $C_1$ and $G\backslash C_1=C_2\cup C_3$ are deep. By a further application of Lemma \ref{lem:nbhdcomp}, there exists a sufficiently large  $A'$ such that for every $g\in G$, either $gH\subseteq C_1 \cup N_{A'}(H)$, or $gH\subseteq (G\backslash C_1)\cup N_{A'}(H)$. Thus $C_1$ satisfies the hypotheses of Lemma \ref{lem:aisetkrop}, so there exists a proper $H$-almost invariant set $X$ such that $XH=X$.
\end{proof}
We recall the Kropholler conjecture, which has been answered affirmatively by Dunwoody.\begin{thm}[\cite{dunwoody2017structure}]\label{thm:kropholler}
Let $G$ be a finitely generated group and let $H\leq G$ be a subgroup. If $G$ contains a proper $H$-almost invariant set $X$ such that $XH=X$, then $G$ admits a splitting over a subgroup $C$ which is commensurable with a subgroup of $H$.
\end{thm}
By combining Proposition \ref{prop:sbgphdist}, Lemma \ref{lem:sbgp} and Theorem \ref{thm:kropholler}, we deduce the following:
\main

We note that \cite{dunwoody1993splitting} proves Theorem \ref{thm:kropholler} under the assumption that $H$ is a polycyclic-by-finite group. However, the argument of \cite{dunwoody1993splitting} holds for general $H$ under the additional assumption that $\tilde e(G,K)=1$ for every infinite index subgroup $K$ of $H$ (as stated in \cite[Theorem 3.4]{dunwoody2000algebraic}).  This weakened version of the Kropholler conjecture also appears in  \cite{scott2000splittings} and  \cite{niblo2002singularity}.

\begin{thm}[\cite{dunwoody1993splitting},\cite{dunwoody2000algebraic}]\label{thm:krophollerspecial}
Let $G$ be a finitely generated group and let $H\leq G$ be a subgroup. Suppose $\tilde e(G,K)=1$ for  every infinite index subgroup $K$ of $H$. If $G$ contains a proper, $H$-almost invariant subset $X$ such that $XH=X$, then $G$ admits a splitting over a subgroup $C\leq G$ that is commensurable to $H$.
\end{thm}

We recall Proposition \ref{prop:deepcvsminimalsepp}, which says that $H$ is almost essentially embedded if and only if $\tilde e(G,K)=1$ for every infinite index subgroup $K$ of $H$.
Using the more classical Theorem \ref{thm:krophollerspecial} instead of Theorem \ref{thm:kropholler}, we deduce the following slightly weaker alternative to Theorem \ref{thm:main}.
 \begin{restatable}{thm}{mainalt}\label{thm:mainalt}
Let $G$ be a group of type $FP_{n+1}^{\mathbb{Z}_2}$ and let $W\subseteq G$ be a coarse $PD_n^{\mathbb{Z}_2}$ subspace. Suppose $G$ contains three essential, coarse disjoint, coarse complementary components of $W$ and that $W$  is almost essentially embedded. Then there exists a subgroup $H\leq G$, at finite Hausdorff distance from $W,$ such that $G$ splits over $H$.
\end{restatable}

\subsection{Applications}
We now discuss applications of Theorem \ref{thm:main}.  To do this, we consider various geometric and algebraic conditions which imply that $W$ coarsely separates $G$ into three  essential, coarse disjoint, coarse complementary components. With the exception of Theorem \ref{thm:maincomm},  Theorem \ref{thm:mainalt} is sufficient to prove results in this section.

Suppose  $G$ is $(n-1)$-acyclic at infinity over $\mathbb{Z}_2$ and $H \leq G$ is a subgroup such that no infinite index subgroup of $H$ coarsely 3-separates $G$.  Proposition  \ref{prop:essentialminimal} then tells us that  every deep coarse complementary component of $H$ is essential. We therefore deduce:
\qiinv
We note that the subgroups $H$ and $H'$ in Theorem \ref{thm:qiinv} are themselves quasi-isometric.

As discussed in the introduction, there are several known examples of groups that are $(n-1)$-acyclic at infinity. For example, it is shown in \cite{brown2000improper} that if $G$ is the fundamental group of a finite graph of groups whose vertex groups are $(n-1)$-acyclic at infinity, whose edge groups are $(n-2)$-acyclic at infinity, and whose vertex and edge groups both satisfy suitable finiteness properties, then $G$ is itself $(n-1)$-acyclic at infinity. In particular, we can deduce the following.

\splittings
This can be coupled with Theorem 8.7 of \cite{vavrichek2012coarse}, which gives necessary and sufficient conditions for  $H$ to coarsely 3-separate $G$. For example, if $G=A*_HB$ is a splitting and there exists some $g\in A\backslash H$  such that $gH g^{-1}$ is commensurable with $H$, then $H$ coarsely 3-separates $G$.

For $k>n$, coarse $PD^{\mathbb{Z}_2}_k$ groups are $(n-1)$-acyclic at infinity over $\mathbb{Z}_2$. Thus Corollary \ref{cor:splittings} has some overlap with Theorem 1.7 from \cite{mosher2011quasiactions} and Theorem 3.4 from \cite{papasoglu2007group}, which concern graphs of groups whose vertices and edges are coarse Poincar\'e duality groups. However, Corollary \ref{cor:splittings} applies to vertex groups that are  $(n-1)$-acyclic at infinity which need not be  coarse Poincar\'e duality groups.

Under the additional assumption  that $H$ is virtually polycyclic, we can improve Theorem \ref{thm:qiinv} by dropping the hypothesis that no infinite index subgroup coarsely separates.

\qiinvpolycyc
\begin{proof}

Example \ref{exmp:polycylic} says that a torsion-free polycyclic group of Hirsch length $r$ is a $PD_r^\mathbb{Z}$ group, hence it must also be a coarse $PD_r^{\mathbb{Z}_2}$ group.  Every virtually polycyclic group of Hirsch length $r$ contains a finite index, torsion-free polycyclic subgroup, hence is a coarse $PD_r^{\mathbb{Z}_2}$ group.

By Theorem \ref{thm:qiinv}, it is sufficient to show that $\tilde e(G,K)=1$ for every infinite index $K\leq H$.
Let $K\leq H$ be a subgroup of minimal Hirsch length such that $\tilde{e}(G,K)>1$. We suppose for contradiction that $r:=h(K)< h(H)=n$.  As $K$ is itself a coarse $PD_r^{\mathbb{Z}_2}$ group, Proposition \ref{prop:essentialminimal} says that $K$ is essentially embedded. Therefore, there exists essential coarse complementary components $C_1$ and $C_2=G\backslash C_1$ of $H$. Therefore, there exists a coarse Mayer--Vietoris sequence  $$\cdots \rightarrow H^r_\mathrm{coarse}(C_1\cup K)\oplus H^r_\mathrm{coarse}(C_2\cup K) \rightarrow H^r_\mathrm{coarse}(K)\rightarrow H^{r+1}_\mathrm{coarse}(G)\rightarrow \cdots.$$ As $C_1$ and $C_2$ are essential, the boundary map $H^r_\mathrm{coarse}(K)\rightarrow H^{r+1}_\mathrm{coarse}(G)$ is injective. Since $G$ is $(n-1)$-acyclic at infinity, $H^{r+1}_\mathrm{coarse}(G)=0$, which is a contradiction.  Hence there is no infinite index subgroup $K\leq H$  such that $\tilde{e}(G,K)>1$.
\end{proof}

We conclude with the following application of Theorem \ref{thm:main}. Unlike Theorem \ref{thm:qiinv} and Corollary \ref{cor:qiinvpolycyc}, which assume geometric hypotheses, Theorem \ref{thm:maincomm} assumes algebraic hypotheses. By making the additional  assumption that $G_e$ is almost essentially embedded i.e. no infinite subgroup of $G_e$ coarsely separates $G$, we can deduce the conclusions of Theorem \ref{thm:maincomm} using Theorem \ref{thm:krophollerspecial} rather than Theorem \ref{thm:kropholler}.

\maincomm

Before beginning the proof of Theorem \ref{thm:maincomm}, we introduce the concept of \emph{trees of spaces} from \cite{scott1979topological}.
Consider the  splitting $G=A*_HB$. We choose CW complexes $K_A$, $K_B$ and $K_H$, each with finite 1-skeleton, such that $\pi_1(K_A)=A$, $\pi_1(K_B)=B$ and $\pi_1(K_H)=H$. We define the quotient space $X$ by gluing $K_H\times [0,1]$ to $K_A$ and $K_B$ as follows:
we attach $K_H\times \{0\}$ to $K_A$ (resp.  $K_H\times \{1\}$ to $K_B$) so that the inclusion of subspaces induces the subgroup inclusion $H\hookrightarrow A$ (resp. $H\hookrightarrow B$) on the level of fundamental groups. By the van Kampen theorem, we see $\pi_1(X)=G=A*_HB$. 

The universal cover $\tilde X$ is equipped with an equivariant projection map $p:\tilde X\rightarrow T$, where $T$ is the Bass-Serre tree associated to the splitting $G=A*_HB$. To each edge $e$ and vertex $v$ of $T$, we associate the edge space $X_e:=p^{-1}(e)$ and vertex space $X_v:=p^{-1}(v)$ respectively. The 1-skeleton of $\tilde{X}$, equipped with the path metric in which edges have length 1,  is quasi-isometric to $G$. Moreover, the 1-skeleton of  each edge space is quasi-isometric to $H$, and the 1-skeleton of each vertex space is quasi-isometric to either $A$ or $B$. For subspaces $U,V\subseteq X$, we say that $U$ is \emph{coarsely contained} in $V$ if for some $r>0$, $U\subseteq N_r(V)$. Each edge space is coarsely contained in its adjacent vertex spaces. We apply such a procedure for a general graph of groups $\mathcal{G}$ to obtain a tree of spaces.

\begin{proof}[Proof of Theorem \ref{thm:maincomm}]
Let $A=G_v$, $B=G_w$ and $H=G_e$.
We choose  $a\in\mathrm{Comm}_A(H)\backslash H$ and $b\in \mathrm{Comm}_B(H)\backslash H$. Thus $aH$ and $bH$ have finite Hausdorff distance from $H$. Let $\tilde X$ be a tree of spaces associated to the graph of groups $\mathcal{G}$ and let $p:\tilde X\rightarrow T$ be an equivariant projection map onto the Bass-Serre tree $T$. As  $G$ acts properly and cocompactly on the $1$-skeleton of $\tilde X$, the Milnor-\v{S}varc lemma tells us there is a quasi-isometry $g:\tilde X^{(1)}\rightarrow G$ such that $g(p^{-1}(e))$ has finite Hausdorff distance from $H\leq G$.

We choose a geodesic $L$ in $T$  with edges $\dots,babH, baH,bH,H,aH,abH,abaH,\dots$. The edge spaces associated to any  two consecutive edges of $L$ have uniformly finite Hausdorff distance from one another.  We define $W:=\cup_{e\in L} p^{-1}(e)$, as $e$ ranges over all the edges in the line $L$. Then $W$ is a \emph{coarse fibration} as described by Kapovich and Kleiner in Section 11.5 of \cite{kapovich2005coarse}.

This is similar to Example 11.11 in \cite{kapovich2005coarse}, except we don't require $H$ to have finite index in $A$ and $B$. We take the line $\hat L$ dual to the geodesic $L$ in the graph of groups;  each vertex of $\hat L$ is an edge of $L$, and two vertices in $\hat L$ are joined by an edge if the corresponding edges in $L$ share a vertex. To each vertex of $\hat L$ we associate the corresponding edge space $X_e$, which is a coarse $PD_n^{\mathbb{Z}_2}$ space. We thus see that $W$ is a coarse fibration whose base is $\hat L$ and whose fibres are edge spaces. It follows from Theorem 11.13 of \cite{kapovich2005coarse} that $W$ is a coarse $PD_{n+1}^{\mathbb{Z}_2}$ space. 	

An edge $e$ of $L$ separates $T$ into two components $T^+$ and $T^-$. Thus $C_1:=p^{-1}(T^+)$ and $C_2:=p^{-1}(T^-)$ are two deep, coarse disjoint, coarse complementary components of $X_e$ in $\tilde X$. As $C_1$ and $C_2$ contain the coarse $PD_{n+1}^{\mathbb{Z}_2}$ half-spaces $(p^{-1}(T^+)\cap W, X_e)$ and $(p^{-1}(T^-)\cap W, X_e)$ respectively,  Corollary \ref{cor:pdnhalf} tells us that $C_1$ and $C_2$ are essential.  

As either  $|\mathrm{Comm}_A(H): H|$ or $|\mathrm{Comm}_B(H): H|$ is at least three, we can choose another geodesic $L'$ in $T$ that  contains the edge $e$ and intersects $L$ in a finite subtree of $T$. A slight modification of the above argument shows that there are three essential, coarse disjoint, coarse complementary components of $X_e$. Since $g:\tilde{X}^{(1)}\rightarrow G$ is a quasi-isometry and $g(X_e)$ has finite Hausdorff distance from $H$, Proposition \ref{prop:essentialqiinv} tells us that there exist three essential, coarse disjoint, coarse complementary components of $H\leq G$. We therefore apply Theorem \ref{thm:main}.
\end{proof}

We remark that if  $|G_v:i_0(G_e)|$ and $|G_w:i_1(G_e)|$ are finite, then $\mathrm{Comm}_{G_v}(i_0(G_e))=G_v$ and $\mathrm{Comm}_{G_w}(i_1(G_e))=G_w$. Therefore, if $G$ is the fundamental group of a finite graph of groups $\mathcal{G}$ in which all edge and vertex groups are coarse $PD_n^{\mathbb{Z}_2}$ groups (for some fixed $n$), our result overlaps with Theorem 2 of \cite{mosher2003quasi} and Theorem 3.1 of \cite{papasoglu2007group}. However, Theorem \ref{thm:maincomm} applies in situation in which edge groups don't necessarily have finite index in adjacent vertex groups.

Under an even stronger algebraic condition, namely that the subgroup over which the group splits commensurises the group, we obtain the following dichotomy:

\comm
\begin{proof}
We consider the tree of spaces $\tilde X$ associated to the splitting of $G$ over $H$, and let $p:\tilde X\rightarrow T$ be the equivariant projection to the Bass-Serre tree. As $\mathrm{Comm}_G(H)=H$, any two edge spaces have finite Hausdorff distance from one another. Therefore, either  $T$ is a line or $H$ coarsely $3$-separates $G$. In the former case, the argument in the proof of Theorem \ref{thm:maincomm} tells us that $G$ is coarse fibration whose base is a line and whose fibres are coarse $PD_n^{\mathbb{Z}_2}$ groups. Thus Theorem 11.13 of \cite{kapovich2005coarse} tells us that $G$ is a coarse $PD_{n+1}^{\mathbb{Z}_2}$ group. Therefore $G'$ is a coarse $PD_{n+1}^{\mathbb{Z}_2}$ group.

In the case where $H$ coarsely $3$-separates $G$, we see that for every coarse complementary component $C$ of $H$,   $p(C)$ contains a ray in the Bass-Serre tree in which consecutive edge spaces have uniformly bounded Hausdorff distance from one another. Thus the argument in the proof of Theorem \ref{thm:maincomm} tell us that $C$ contains a coarse $PD_{n+1}^{\mathbb{Z}_2}$ half-space, so must be essential. Thus $H$ is essentially embedded. Since $H$ is essentially embedded and coarsely $3$-separates $G$, Theorem \ref{thm:mainalt} tells us $G'$ splits over a subgroup at finite Hausdorff distance from $f(H)$.
\end{proof}

\subsection{Quasi-Isometry Invariance of Codimension One Subgroups}

Rather than considering when a group admits a splitting, suppose we now consider the weaker condition that a group has a codimension one subgroup. Doing this allows us to obtain a result that doesn't require coarse 3-separation. To state our results we  need the following definition:

\begin{defn}
A group $G$ is a \emph{coarse $n$-manifold group} if it is of type $FP_{n}^{\mathbb{Z}^2}$ and $H^n(G,\mathbb{Z}_2 G)$ has a non-zero, finite dimensional, $G$-invariant subspace.
\end{defn}

It is shown in  \cite{kleinercohomology} that a group $G$ is coarse $2$-manifold group if and only if it is virtually a surface group (see also \cite{bowditch2004planar}). The following lemma shows that we may think of coarse $n$-manifold groups as a generalisation of coarse $PD_n^{\mathbb{Z}_2}$ groups.

\begin{lem}\label{lem:charplanar}
Let $G$ be a group of type $FP_{n}^{\mathbb{Z}^2}$ and let $(G,C_\bullet)$ be a uniformly $(n-1)$-acyclic $R$-metric complex admitting a $G$-action. Then $G$ is a  coarse $n$-manifold group if and only if there exists a nonzero $[\alpha_0]\in \hat H^n_c(C_\bullet)$ and a $D\geq 0$ such that $G=N_D(\mathrm{Mob}([\alpha_0],D))$. In particular, if $G$ is a coarse $PD_n^{\mathbb{Z}_2}$ group, then it is a coarse $n$-manifold group.
\end{lem}
\begin{proof}
Suppose that $G$ is a coarse $n$-manifold group and that $F\leq H^n(G,\mathbb{Z}_2 G)$ is a non-zero, finite dimensional, $G$-invariant subspace. Since $F$ is finite, there is a nonzero $[\alpha_0]\in F$ such that $\mathrm{Stab}([\alpha_0])$ has finite index in $G$. By Proposition \ref{prop:chargroupcohom}, we identify $[\alpha_0]$ with an element of $\hat H^n_c(C_\bullet)$. It follows from Lemma \ref{lem:mobset} that for some $D$ sufficiently large, $G=N_D(\mathrm{Mob}([\alpha_0],D))$.

Now suppose $G=N_D(\mathrm{Mob}([\alpha_0],D))$ for some non-zero $[\alpha_0]\in \hat H^n_c(C_\bullet)$. Then Lemma \ref{lem:mobset} tells us that $\mathrm{Stab}([\alpha_0])$ has finite index in $G$. Hence $[\alpha_0]G$ can be identified with a finite dimensional subspace of $H^n(G,\mathbb{Z}_2 G)$ by Proposition \ref{prop:chargroupcohom}.
\end{proof}
Proposition \ref{prop:qimob} and Lemma \ref{lem:charplanar} imply that if $G$ is a coarse $n$-manifold group and $G'$ is quasi-isometric to $G$, then $G'$ is also a coarse $n$-manifold group.
We say a space $X$ is \emph{a coarse $n$-manifold space} if there exists  an $(n-1)$-acyclic $R$-metric complex $(X,C_\bullet)$, an $[\alpha_0]\in \hat H^n_c(C_\bullet)$ and a $D\geq 0$ such that $X=N_D(\mathrm{Mob}([\alpha_0],D))$. Clearly $G$ is a coarse $n$-manifold group if and only if it is a coarse $n$-manifold space.

The following Theorem allows one to detect codimension one subgroups from the coarse geometry of a group. The Sageev construction \cite{sageev1995ends} shows that a group has a codimension one subgroup if and only if it acts \emph{essentially} on a CAT(0) cube complex; in particular such a group cannot have property (T).

\codimsbgp
\begin{proof}
We suppose that $G$ (and hence $G'$) is not a coarse $(n+1)$-manifold group.
Let $C_1$ and $C_2$ be two essential, coarse disjoint, coarse complementary components of $W$. Without loss of generality, we may replace $C_2$ with $G\backslash C_1$. We thus obtain the coarse Mayer-Vietoris sequence  $$\cdots \rightarrow H^n_\mathrm{coarse}(C_1\cup K)\oplus H^n_\mathrm{coarse}(C_2\cup K) \rightarrow H^n_\mathrm{coarse}(K)\rightarrow H^{n+1}_\mathrm{coarse}(G)\rightarrow \cdots.$$ As $C_1$ and $C_2$ are essential, $H^n_\mathrm{coarse}(K)\rightarrow \hat H^{n+1}_\mathrm{coarse}(G)$ is injective so has image containing some $0\neq [\alpha_0]$. 

Let $T:=\mathrm{stab}([\alpha_0])$. The proof of Lemma \ref{lem:sbgp} then tells us that  $H\subseteq N_R(\mathrm{Mob}([\alpha_0],R))$ for some  $R\geq 0$.  Therefore, by  Lemma \ref{lem:mobset} we see that $H$ is commensurable to subgroup of $T$, so $H':=H\cap T$ has finite index in $H$. If $[G:T]<\infty$, then Lemma \ref{lem:mobset} ensures that there is some $D$ sufficiently large such that $G=N_D(\mathrm{Mob}([\alpha_0],D))$; this contradicts our hypothesis that  $G$ is not a coarse $(n+1)$-manifold group. Therefore $[G:T]=\infty$. 

Since $\tilde{e}(G,H)>1$ and $[H,H']<\infty$, it follows from Lemma 2.4 (v) of \cite{kropholler1989relative}  that  $\tilde{e}(G,H')=\tilde{e}(G,H)>1$. Since $[G:T]=\infty$, Lemma 2.4 (vi) of \cite{kropholler1989relative} tells us that $\tilde{e}(G,T)\geq\tilde{e}(G,H)>1$. Corollary \ref{cor:mobsetcocomp} and Proposition \ref{prop:qimob} now tell us that if $f:G'\rightarrow G$ is a quasi-isometry, then $f(T')$  has finite Hausdorff distance from $T$, where $T':=\mathrm{stab}(f^*[\alpha_0])$. Thus, $\tilde{e}(G',T')>1$. By applying Lemma \ref{lem:stabcoarsesep} to some irreducible coarse complementary component of $T'$, we see that $e(G',H')>1$ for some subgroup $H'\leq T'$.
\end{proof}

\begin{rem}
Theorem \ref{thm:main}  still holds if $W$ is only required to only be a coarse $n$-manifold space. Similarly, Theorem \ref{thm:codimsbgp} still holds if $H$ is only required to be a coarse $n$-manifold group. In this more general context, we can weaken our definition of essential components as follows: suppose $W$ is a coarse $n$-manifold space  and $C$ is a coarse complementary component of $W$. Then we say that $C$ is essential if  $$[\alpha_0]\notin \mathrm{im}(H^n_\mathrm{coarse}(C\cup W)\rightarrow H^n_\mathrm{coarse}(W)),$$ where $[\alpha_0]$ is a coarse cohomology class such that $W=N_D(\mathrm{Mob}([\alpha_0],D))$ for some $D\geq 0$. Theorem \ref{thm:main} and \ref{thm:codimsbgp} hold with this extended notion of essential components.
\end{rem}

\appendix
\section{Uniformly Acyclic Metric Complexes}\label{app:metriccomp}
\begin{rem}\label{rem:ripsdisp}
Using the metric complex structure of the Rips complex $P_i(X)$ from Example \ref{exampl:rips}, we see that if $\sigma\in C_k(P_i(X);R)$ and $\mathrm{supp}(\sigma)\subseteq K$, then $\sigma\in C_\bullet(P_i(N_i^X(K));R)$.
\end{rem}
\extendmapscoarse*

\begin{proof}
We define $y_\Delta:=f(p_i(\Delta))\in Y$ for every $\Delta\in \Sigma_i$. 

(1): We proceed by induction.  For the base case, we extend the map $f\circ p_0: \Sigma_0 \rightarrow Y$  to an augmentation preserving map $f_0:C_0\rightarrow C_0(P_0(Y);R)$, which  has displacement zero over $f$.

Let $k<n$. We assume that there exists some $i_k=i_k(\lambda)$ and a chain map $f_\#:[C_\bullet]_k\rightarrow C_\bullet(P_{i_k}(Y);R)$ of $k$-displacement at most $M_k=M_k(\lambda,\mu,\phi,d)$ over $f$. 
By Lemma \ref{lem:compdispmaps}, we see that $\mathrm{supp}(f_k\partial_{k+1} \Delta)\subseteq N^Y_{\phi(d)+M_k}(y_\Delta)$ for every  $\Delta\in \Sigma_{k+1}$. Furthermore, by Remark \ref{rem:ripsdisp}, we see that $f_k\partial_{k+1} \Delta\in C_k(P_{i_k}(N^Y_{\phi(d)+M_k+i_k}(y_\Delta));R).$
Letting $i_{k+1}:=\lambda(i_k)$ and $M_{k+1}:=\mu(i_k,\phi(d)+M_k+i_k)$,  there exists an $\omega_\Delta\in C_{k+1}(P_{i_{k+1}}(N^Y_{M_{k+1}}(y_\Delta));R)$ such that $\partial_{k+1}\omega_\Delta=f_k\partial_{k+1}\Delta$. We define $f_{k+1}(\Delta)=\omega_\Delta$ for each $\Delta\in \Sigma_{k+1}$, thus defining a chain map $$f_\#:[C_\bullet]_{k+1}\rightarrow C_\bullet(P_{i_{k+1}}(Y);R)$$ of $(k+1)$-displacement at most $M_{k+1}$ over $f$.

(2): We also proceed by induction. If  $\Delta\in \Sigma_0$, $\mathrm{supp}(f_0(\Delta)-g_0(\Delta))\subseteq N^Y_r(y_\Delta)$ thus by Remark \ref{rem:ripsdisp}, $f_0(\Delta)-g_0(\Delta)\in C_0(P_i(N^Y_{r+i}(y_\Delta));R)$. Since $f_\#$ and $g_\#$ are both augmentation preserving,  $f_0(\Delta)-g_0(\Delta)$ is a reduced $0$-cycle. Thus there is an $\omega_\Delta\in C_1(P_{\lambda(i)}(N^Y_{\mu(i,r+i)}(y_\Delta)))$ such that $\partial\omega_\Delta=f_0(\Delta)-g_0(\Delta)$. Letting $h_0(\Delta)=\omega_\Delta$ for each $\Delta\in \Sigma_0$, we define a chain homotopy $h_\#:[C_\bullet]_0\rightarrow C_1(P_{\lambda(i)}(Y);R)$ that has $0$-displacement $\mu(i,r+i)$ over $f$ such that $\partial h_\#+h_\# \partial=g_\#-f_\#$.

Let $k< n-1$. We assume there is a $j_k=j_k(i,\lambda)$ and a proper chain homotopy $h_\#:[C_\bullet]_k\rightarrow C_{\bullet+1}(P_{j_k}(Y);R)$ such that $\partial h_\#+h_\# \partial=g_\#-f_\#$ and $h_\#$ has $k$-displacement at most $N_k=N_k(i,\lambda,\mu,\phi,d)$ over $f$. Lemma \ref{lem:compdispmaps} and Remark \ref{rem:ripsdisp} tell us that for each  $\Delta\in \Sigma_{k+1}$, $$h_\#(\partial\Delta)- g_\#(\Delta)+ f_\#(\Delta)\in C_{k+1}(P_{i_k}(N^Y_{R_k}(y_\Delta));R),$$ where $R_k:=\max(\phi(d)+N_k+i_k, r+i_k)$. 
Hence there is an $\omega_\Delta
\in C_{k+2}(P_{\lambda(i_k)}(N^Y_{\mu(i_k,R_k)}(y_\Delta));R)$
such that $\partial_{k+2}\omega_\Delta=h_\#(\partial\Delta)- g_\#(\Delta)+ f_\#(\Delta)$. We thus define $h_{k+1}(\Delta)=\omega_\Delta$ for each $\Delta\in \Sigma_{k+1}$.
\end{proof}

\extendmapsunif*
\begin{proof}
The proof is very similar to the proof of Lemma \ref{lem:extendmapscoarse}, so we will only outline where the proofs differ. For the base case of $(1)$, we use the fact that $p'_0$ is surjective to find a $\Delta'\in \Sigma'_0$ such that $p'_0(\Delta')=f(p_0(\Delta))$ for each $\Delta\in \Sigma_0$. We thus define $f_0(\Delta)=\Delta'$ for each $\Delta\in \Sigma_0$ and then extend linearly.

For the inductive step, we use the displacement of $[D_\bullet]_n$. Defining $y_\Delta$ as in the proof of Lemma \ref{lem:extendmapscoarse}, Lemma \ref{lem:compdispmaps} says that for each $\Delta\in \Sigma_{k+1}$, $\mathrm{supp}(f_k\partial_{k+1} \Delta)\subseteq N^Y_{\phi(d_1)+M_k}(y_\Delta)$. Thus by applying Lemma \ref{lem:dispmetric}, we see that $f_k\partial_{k+1} \Delta\in D_k[N^Y_{\phi(d_1)+M_k+kd_2}(y_\Delta)].$ We extend $f_\#$ by setting $f_{k+1}(\Delta)=\omega_\Delta$ for some $\omega_\Delta\in D_{k+1}[N^Y_{\mu(\phi(d_1)+M_k+kd_2)}(y_\Delta)]$ such that $\partial\omega_\Delta=f_k\partial_{k+1}\Delta$. 
Part (2) is proved analogously, also making use of Lemma \ref{lem:dispmetric}.
\end{proof}

Making minor modifications to the above proofs, we deduce the following:

\begin{lem}\label{lem:extendmapsunifrel}
Let  $(X,C_\bullet,\Sigma_\bullet,p_\bullet)$ and $(Y,D_\bullet)$ be $R$-metric complexes such that $(Y,D_\bullet)$ is  $\mu$-uniformly $(n-1)$-acyclic and $C_\bullet$ and $D_\bullet$ have $n$-displacement at most $d_1$ and $d_2$ respectively.
Suppose also there is a subcomplex $(X,C'_\bullet,\Sigma'_\bullet,p'_\bullet)$ of $(X,C_\bullet)$, and a chain map $f'_\#:[C'_\bullet]_n\rightarrow D_\bullet$ of $n$-displacement at most $M'$ over an $(\eta,\phi)$-coarse embedding $f:X\rightarrow Y$. Then $f'_\#$ extends to a chain map $f_\#:[C_\bullet]_n\rightarrow D_\bullet$ of $n$-displacement at most $M=M(\mu,\phi,d_1,d_2,M')$ over $f$. 
\end{lem}
\begin{proof}
For $i\leq n$ and $\sigma\in \Sigma'_i\subseteq \Sigma_i$, we define $f_\#(\sigma)=f'_\#(\sigma)$. For $\sigma\in \Sigma_i\backslash \Sigma'_i$, we proceed as in Lemma \ref{lem:extendmapsunif}.
\end{proof}

\controlspace*
\begin{proof}
(\ref{eqn:controlspacemain}): We proceed by induction on $n$. For each $x\in X$ and $r\geq 0$, we see that the map $\tilde{H}_0(P_0(N_r(x));R)\rightarrow \tilde{H}_0(P_{\lambda(0)}(N_{\mu(0,r)}(x));R)$, induced by inclusion, is zero.
This means $N_r(x)$ is contained in a single connected component of $P_{\lambda(0)}(N_{\mu(0,r)}(x))$, so $\tilde{H}_0[P_{\lambda(0)}(N_r(x))]\rightarrow \tilde{H}_0[P_{\lambda(0)}(N_{\mu(0,r)}(x))]$, also induced by inclusion, is zero. Let $\mu_1(r):=\mu(0,r)$. Therefore $C_\bullet(P_{\lambda(0)}(X);R)$ is $\mu_1$-uniformly 0-acyclic and has displacement at most $\lambda(0)$.

Let $0<k<n$. For our inductive hypothesis, we assume there is a $k$-dimensional $R$-metric complex $(X,C_\bullet,\Sigma_\bullet,p_\bullet)$  that is $\mu_{k}$-uniformly $(k-1)$-acyclic with  $k$-displacement at most $d_{k}$, where $\mu_k$ and $d_k$ depend only on $\lambda$ and $\mu$. By Lemma \ref{lem:extendmapscoarse}, there is an $i_{k}=i_{k}(\lambda)$ such that  the identity $\mathrm{id}_X:X\rightarrow X$ induces the chain map $f_\#:C_\bullet\rightarrow C_\bullet(P_{i_{k}}(X);R)$ of $k$-displacement at most $M_k=M_k(\lambda,\mu)$ over $\mathrm{id}_X$. 
Let $\iota_\#:C_\bullet(P_{i_k}(X);R)\rightarrow C_\bullet( P_{\lambda(i_{k})}(X);R)$ be the inclusion.

We now consider the algebraic mapping cylinder $D_\bullet$ of $\iota_\#f_\#$, defined by $D_j:=C_j\oplus C_{j-1} \oplus C_j(P_{\lambda(i_{k})}(X))$ with boundary maps $\tilde{\partial}(a,b,c)=(\partial a+b,-\partial b,\partial c - \iota_\#f_\# b)$. 
Each $D_j$ is the direct sum of finite type free $R$-modules over $X$, so inherits the structure of a finite type free $R$-module over $X$.  Since $C_\bullet$, $C_\bullet(P_{\lambda(i_k)}(X))$ and $\iota_\#f_\#$ have $(k+1)$-displacements at most $d_k$, $\lambda(i_k)$ and $M_k$ respectively, we see that  $(X,D_\bullet)$ is an $R$-metric complex of $(k+1)$-displacement at most $d_{k+1}:=\max(d_k, \lambda(i_k), M_k)$. 
 
We note that $(X,C_\bullet)$ can be considered a subcomplex of $(X,D_\bullet)$ via the inclusion map $\tau_\#$ given by $a\mapsto(a,0,0)$. Thus Lemma \ref{lem:extendmapsunifrel} allows us to define a chain map $r_\#:[D_\bullet]_{k}\rightarrow C_\bullet$ of $k$-displacement at most $P_k=P_k(\mu_k,d_{k+1})$ over $\mathrm{id}_X$ such that $r_\#\tau_\#=\mathrm{id}_{[C_\bullet]_k}$.

We define the boundary map $\partial:D_{k+1}\rightarrow C_{k}$  to be the composition $r_\#\circ \tilde{\partial}$.
We now consider the   $R$-metric complex 
\begin{equation}\label{eqn:metriccomplexextension}
D_{k+1}\xrightarrow{\partial} C_{k}\rightarrow \cdots \rightarrow C_0 \rightarrow 0,
\end{equation}
  which  has $(k+1)$-displacement at most $d_{k+1}+P_k$ over the identity. 
For  $x\in X$, let $\sigma\in C_{k}[N_r(x)]$ be a reduced cycle.  As $f_\#$ has $k$-displacement at most $M_k$ over $\mathrm{id}_X$,  $\mathrm{supp}(f_\#\sigma)\subseteq N_{r+M_k}(x)$; thus by Remark \ref{rem:ripsdisp}, $f_\#\sigma\in C_k(P_{i_k}(N_{r+M_k+i_k}(x));R)$. Hence there is an $\omega_\sigma\in C_{k+1}(P_{\lambda(i_k)}(N_{\mu(i_k,r+M_k+i_k)}(x));R)$ such that $\partial\omega_\sigma=\iota_\#f_\#\sigma$.
Letting $\lambda:=(0,\sigma,\omega_\sigma)\in D_{k+1}$, we see $\tilde{\partial}\lambda=(\sigma,0,0)=\tau_\#\sigma$. Therefore $\partial\lambda=r_\#\tau_\#\sigma=\sigma$. Since $\lambda\in D_{k+1}[N_{\mu(i_k,r+M_k+i_k)}(x)]$, we see that (\ref{eqn:metriccomplexextension}) is $\mu'$-uniformly $k$-acyclic, where $\mu'(r):=\mu(i_k,r+M_k+i_k)$.

(\ref{eqn:controlspaceinf}): We apply the method used to prove (\ref{eqn:controlspacemain}) in each dimension.

(\ref{eqn:controlspaceequiv}): We inductively modify the chain complex produced in the proof of (\ref{eqn:controlspacemain}) so that it admits a $G$-action.  The base case is trivial since $G$ acts freely on the zero skeleton of every Rips complex.
For the inductive hypothesis, we assume the metric complex $C_{k}\rightarrow \cdots \rightarrow C_0 \rightarrow 0$  admits a free $G$-action, is $\mu_{k}$-uniformly $(k-1)$-acyclic and has  $k$-displacement at most $d_{k}$. As in the proof of (\ref{eqn:controlspacemain}), we then construct a finite type free $R$-module $(D_{k+1},\Sigma_{k+1},p_{k+1})$ over $X$ such that $D_{k+1}\xrightarrow{\partial} C_{k}\rightarrow \cdots \rightarrow C_0 \rightarrow 0$ is a $\mu'$-uniformly $k$-acyclic $R$-metric complex with $(k+1)$-displacement at most $d_{k+1}+P_k$. Let $S:=\mu'(d_{k+1}+P_k+kd_k)$.

   We choose $T\subseteq X$  such that each $G$-orbit in $X$ contains precisely one element of $T$. 
Let $\Sigma'_{k+1}=\{(g,t,\rho)\in G\times T\times \Sigma_{k+1}\mid\rho\in D_{k+1}[N_{S}(t)] \} $ and let $D'_{k+1}$ be the free $R$-module with $R$-basis $\Sigma'_{k+1}$. Then  $(D'_{k+1},\Sigma'_{k+1},p'_{k+1})$  is a finite type free $R$-module over $X$, where the projection map is defined by $p'_{k+1}(g,t,\rho):=gt$. We define the boundary map $\partial':D'_{k+1}\rightarrow C_k$ by $(g,t,\rho)\mapsto g\partial\rho$. 
As  $D'_{k+1}$ admits a $G$-action given by $g'(g,t,\rho)\mapsto (g'g,t,\rho)$, we see that \begin{align}\label{eqn:gequivmeticc}
D'_{k+1}\xrightarrow{\partial'} C_{k}\rightarrow \cdots \rightarrow C_0 \rightarrow 0
\end{align}  is a metric complex that admits  a $G$-action and has $(k+1)$-displacement at most $S$.

We claim that for every $\Delta\in \Sigma_{k+1}$, there exists a $\hat \Delta\in D'_{k+1}$ such that $\partial'\hat \Delta=\partial\Delta$. Indeed, since $\partial$ has displacement at most $d_{k+1}+P_k$ over the identity,  $\mathrm{supp}(\partial\Delta)\subseteq N_{d_{k+1}+P_k}(p_{k+1}(\Delta))$. Lemma \ref{lem:dispmetric} then tells us that $\partial\Delta\in C_k[N_{d_{k+1}+P_k+kd_k}(p_{k+1}(\Delta))]$. 
There exists a unique $g\in G$ and $t\in T$ such that $p_{k+1}(\Delta)=gt$; thus $g^{-1}\partial\Delta\in C_{k}[N_{d_{k+1}+P_k+kd_k}(t)]$. As $g^{-1}\partial\Delta$ is a reduced $k$-cycle, there exists a $\gamma_\Delta\in  D_{k+1}[N_{S}(t)]$ such that $\partial\gamma_\Delta= g^{-1}\partial\Delta$. We can write $\gamma_\Delta=\sum_{i=1}^nr_i\gamma_i$, where $r_i\in R$ and $\gamma_i\in \Sigma_{k+1}\cap D_{k+1}[N_{S}(t)]$ for each $i$. Let $\hat \Delta:=\sum_{i=1}^n r_i(g,t,\gamma_i)\in D'_{k+1}$. Then $\partial'\hat \Delta=\sum_{i=1}^n r_i g\partial\gamma_i=g\partial\sum_{i=1}^n r_i\gamma_i=g\partial\gamma_\Delta=\partial\Delta$. As $\mathrm{supp}(\hat \Delta)=\{p_{k+1}(\Delta)\}$ and $\partial\Delta\in C_k[N_{d_{k+1}+P_k+kd_k}(p_{k+1}(\Delta))]\subseteq C_k[N_S(p_{k+1}(\Delta))]$, it follows from Remark \ref{rem:subcomplexexplicit} that $\hat\Delta\in D'_{k+1}[N_S(p_{k+1}(\Delta))]$.

Let $\sigma\in C_k[N_r(x)]$ be a reduced $k$-cycle. Then there exists an $\omega\in D_{k+1}[N_{\mu'(r)}(x)]$ such that $\partial\omega=\sigma$. We can write $\omega=\sum_{i=1}^m r'_i \omega_i$, where $\omega_i\in \Sigma_{k+1}$ and $r'_i\in R$ for each $i$. Thus $\partial'\sum_{i=1}^m r'_i \hat\omega_i=\partial \omega=\sigma$ and $\sum_{i=1}^m r'_i \hat\omega_i\in D'_{k+1}[N_{\mu'(r)+S}(x)]$. Letting $\mu''(r)=\mu'(r)+S$, we thus see that (\ref{eqn:gequivmeticc}) is $\mu''$-uniformly $k$-acyclic.
\end{proof}

\section{Roe's Coarse Cohomology}\label{app:roecoarse}
We will use the technology of coarse cohomology as defined in \cite{roe1993coarse} and \cite{roe2003lectures}. We fix some commutative ring $R$ with unity.
\begin{defn}
Let $(Y,d)$ be a bounded geometry metric space. An open cover $\mathcal{U}$ is called \emph{good} if it is locally finite and the closure of each $U\in \mathcal{U}$ is compact. A sequence  $\mathcal{U}_1,\mathcal{U}_2,\dots $ of good covers of $Y$ is said to be an \emph{anti-\v{C}ech approximation} if there exists a sequence of real numbers $R_n\rightarrow \infty$ such that for all $n$:
\begin{enumerate}
\item every $U\in \mathcal{U}_n$ has diameter at most $R_n$;
\item $\mathcal{U}_{n+1}$ has Lebesgue number at least $R_n$, i.e. every set of diameter at most $R_n$ is contained in some $U\in \mathcal{U}_{n+1}$.
\end{enumerate}
\end{defn}
The \emph{nerve} $|\mathcal{U}|$ of an open cover $\mathcal{U}$ of $Y$ is a simplicial complex with vertex set $\mathcal{U}$ such that $\{U_0,\dots, U_n\}\subseteq \mathcal{U}$ span an $n$-simplex if and only if $\cap_{i=0}^n U_i\neq \emptyset$.
If  $\mathcal{U}_1,\mathcal{U}_2,\dots $ is an anti-\v{C}ech approximation, then for every $n$, $\mathcal{U}_n$ is a refinement of $\mathcal{U}_{n+1}$, i.e. every $U\in \mathcal{U}_n$ is contained in some element of $\mathcal{U}_{n+1}$; this shows the existence of simplicial inclusion maps $|\mathcal{U}_n|\rightarrow |\mathcal{U}_m|$ for every $n\leq m$.

We thus see that $\{C_c^\bullet(|\mathcal{U}_n|;R)\}_{n\in \mathbb{N}}$ is an inverse system whose bonds are induced by the inclusion. Let $\hat C^\bullet$ be the inverse limit of this inverse system. We define $HX^*(Y;R)$ to be the cohomology of this inverse system. Roe shows that $HX^*(Y;R)$ is independent of the choice of anti-\v{C}ech approximation, and that a coarse embedding $f:X\rightarrow Y$ induces a map $f^*:HX^*(Y;R)\rightarrow HX^*(X;R)$. The full details in this are found in \cite{roe1993coarse} and \cite{roe2003lectures}.

We can now pick an anti-\v{C}ech approximation where for each $n$,  $\mathcal{U}_n:=\{B_n(y)\mid y\in Y\}$. The nerve $|\mathcal{U}_n|$ is simply the Rips complex $P_n(Y)$. Thus $HX^*(Y;R)$ is the cohomology of the inverse limit $\hat C^\bullet=\varprojlim C_c^\bullet(P_n(Y);R)$. By Theorem 3.5.8 of \cite{weibel1994introduction},  there is a short exact sequence  \begin{align}
0 \rightarrow {\varprojlim}^1 H^{k-1}_c(P_n(Y);R)  \rightarrow HX^k(Y;R) \rightarrow \varprojlim H^k_c(P_n(Y);R) \to 0,\label{eq:milnor}
\end{align} where the ${\varprojlim}^1$ term is known as the  \emph{derived limit} and will not be discussed here. This makes it hard to calculate coarse cohomology in general.

However, things are simpler if $Y$ is coarsely uniformly acyclic over $R$. In such a situation, the ${\varprojlim}^1 H^{k-1}_c(P_n(Y);R)$ term vanishes and $\varprojlim H^k_c(P_n(Y);R)$ is isomorphic to the image of the map  $H^k_c(P_{j}(Y);R)\rightarrow H^k_c(P_{i}(Y);R)$, induced by inclusion, for some $j\gg i \gg 0$. This can be seen by applying Lemma \ref{lem:extendmapscoarse} in a similar way to proof of  Lemma \ref{lem:mvseqtech}. Moreover, Lemma \ref{lem:mvseqtech} shows that the image of the map  $H^k_c(P_{j}(Y);R)\rightarrow H^k_c(P_{i}(Y);R)$ is naturally isomorphic to $H^k_\mathrm{coarse}(Y)$. This is also evident from Theorem 5.28 of \cite{roe2003lectures}.  Thus the condition that $Y$ is uniformly acyclic means that we can calculate $HX^k(Y;R)$  without using inverse limits. 

By not passing to the inverse limit and simply using Rips complexes with sufficiently large parameters, we  preserve quantitative information e.g. the diameter of cocycles and the displacement of maps. We use the notation $H^k_\mathrm{coarse}(Y;R)$ rather than $HX^k(Y;R)$ to emphasise this point. We remark that if $Y$ is only coarsely uniformly $(n-1)$-acyclic over $R$, the preceding discussion still holds in dimensions at most $n$, taking modified cohomology with compact supports in the top dimension.
\bibliographystyle{amsalpha}
\bibliography{bibtex}
\end{document}